\documentclass[11pt]{amsart}

\usepackage{amssymb}
\usepackage{enumitem}
\usepackage{graphicx}
\usepackage{verbatim}
\usepackage{epstopdf}
\usepackage{color}
\newcommand{\eps}{\varepsilon}
\newcommand{\ph}{\varphi}
\newcommand{\NN}{\mathbb{N}}

\newcommand{\RR}{\mathbb{R}}
\newcommand{\one}{\mathbf{1}}
\newcommand{\AAA}{\mathcal{A}}

\newcommand{\FFF}{\mathcal{F}}
\newcommand{\EE}{\mathbb{E}}
\newcommand{\PP}{\mathbb{P}}
\newcommand{\MMM}{\mathcal{M}}
\newcommand{\WWW}{\mathcal{W}}
\newcommand{\Mac}{\MMM^{\mathrm{ac}}}

\newcommand{\KK}{\mathbf{K}}
\newcommand{\II}{\mathbf{I}}
\newcommand{\JJ}{\mathbf{J}}
\newcommand{\Mh}{\MMM^{\mathrm{h}}}
\newcommand{\GGG}{\mathcal{G}}
\newcommand{\PPP}{\mathcal{P}}
\newcommand{\QQQ}{\mathcal{Q}}
\newcommand{\RRR}{\mathcal{R}}
\newcommand{\LLL}{\mathcal{L}}
\newcommand{\SSS}{\mathcal{S}}
\newcommand{\NNN}{\mathcal{N}}
\newcommand{\XX}{\mathcal{X}}
\newcommand{\llim}{\varliminf}
\newcommand{\ulim}{\varlimsup}

\newcommand{\ba}{{\overline{\theta}}}

\newcommand{\bl}{{\overline{\lambda}}}

\newcommand{\jj}{\mathbf{j}}

\newcommand{\ld}{\underline{\delta}}
\newcommand{\ud}{\overline{\delta}}
\newcommand{\grad}{\nabla}

\DeclareMathOperator{\graph}{graph}

\DeclareMathOperator{\Leb}{Leb}

\DeclareMathOperator{\Id}{Id}

\theoremstyle{plain}
\newtheorem{theorem}{Theorem}[section]
\newtheorem{proposition}[theorem]{Proposition}
\newtheorem{corollary}[theorem]{Corollary}
\newtheorem{lemma}[theorem]{Lemma}
\newtheorem{thma}{Theorem}
\newtheorem{thm}[theorem]{Theorem}

\theoremstyle{definition}
\newtheorem{definition}[theorem]{Definition}

\theoremstyle{remark}

\newtheorem{remark}[theorem]{Remark}

\numberwithin{equation}{section}

\title[Non-stationary non-uniform hyperbolicity]{Non-stationary non-uniform hyperbolicity:  SRB measures for dissipative maps}%non-uniformly hyperbolic attractors}
\author{Vaughn Climenhaga}
\address{Department of Mathematics \\ University of Houston \\ Houston, TX  77204, USA}
\email{climenha@math.uh.edu}
\urladdr{http://www.math.uh.edu/$\sim$climenha/}

\author{Dmitry Dolgopyat}
\address{Department of Mathematics \\ University of Maryland \\ College Park, MD 20742, USA}
\email{dmitry@math.umd.edu}
\urladdr{http://www.math.umd.edu/$\sim$dmitry/}

\author{Yakov Pesin}
\address{Department of Mathematics \\ McAllister Building \\ Pennsylvania State University \\ University Park, PA 16802, USA}
\email{pesin@math.psu.edu}
\urladdr{http://www.math.psu.edu/pesin/}

\begin{document}

\date{\today}
\begin{abstract}
We prove the existence of SRB measures for diffeomorphisms where a positive volume set of initial conditions satisfy an ``effective hyperbolicity'' condition that guarantees certain recurrence conditions on the iterates of Lebesgue measure.  We give examples of systems that do not admit a dominated splitting but can be shown to have SRB measures using our methods.
\end{abstract}

\thanks{V.C.\ and Ya.P.\ were partially supported by NSF grant DMS-0754911.  V.C.\ was partially supported by NSF grant DMS-1362838. Ya.P.\ was partially supported by NSF grants DMS-1101165 and DMS-1400027.  D.D.\ was partially supported by NSF grant DMS-1101635.  Parts of this research were carried out at
%during visits to 
the Fields Institute, the EPFL and ICERM} 
%while the authors were visiting The Fields Institute and the Centre Interfacultaire Bernoulli.} % at the \'Ecole Polytechnique F\'ed\'erale de Lausanne.}

\maketitle

\section{Introduction and main results}

\subsection{Attractors and SRB measures}

Let $M$ be a $d$-dimensional smooth Riemannian manifold, $U\subset M$ an open set such that $\overline{U}$ is compact, and $f\colon U\to M$ a $C^{1+\alpha}$ diffeomorphism onto its image such that $\overline{f(U)} \subset U$; in this case $U$ is called a \emph{trapping region} for $f$.  The simplest case is when $U=M$, so $f$ is defined on the entire manifold, but there are many important examples in which $U \neq M$, and as usual we write
$\Lambda = \bigcap_{n\geq 0} f^n(U)$
for the \emph{topological attractor} onto which the trajectories of $f$ accumulate.

A central question in smooth ergodic theory is to determine whether $f$ admits a \emph{physical measure}: an invariant measure that governs the statistical properties of volume-typical points.  For hyperbolic dynamical systems, the most important class of such measures are the \emph{SRB measures}, which were first constructed for uniformly hyperbolic systems by Sinai \cite{yS72}, Ruelle \cite{dR76,dR78}, and Bowen \cite{rB75}, using symbolic techniques based on Markov partitions.  A significant generalization of this approach to non-uniformly hyperbolic systems, based on countable-state symbolic dynamics using \emph{towers} of a special type, was achieved by Young \cite{lY98}.

We study a different construction of SRB measures, which is more directly geometric and has its roots in work of Pesin and Sinai \cite{PS82} on
%, who considered \emph{$u$-measures} for 
partially hyperbolic systems, for which symbolic models may not be available.
%\footnote{These measures are absolutely continuous on unstable manifolds but are not required to be hyperbolic, and hence may not be SRB measures.}  
A similar geometric approach was used by Pesin in \cite{yP92} for hyperbolic systems with singularities, by Carvalho in \cite{mC93} for a class of `derived from Anosov' diffeomorphisms, and by Alves, Bonatti, and Viana in \cite{BV00,ABV00} for more general attractors with a dominated splitting. 

Our main results give general criteria for existence of an SRB measure that can be applied without a dominated splitting; these use the notion of `effective hyperbolicity' introduced by the first and third authors in \cite{CP16}.

We recall the definition of SRB measure  from~\cite[Definition 13.1.1]{BP07}.  Let $\mu$ be an invariant measure and suppose that $\mu$ is hyperbolic -- that is, all Lyapunov exponents of $\mu$ are non-zero.  For $\mu$-a.e.\ $x$ there are local stable and unstable manifolds $V^s(x)$ and $V^u(x)$ through $x$.  It is easy to see that for each such $x$ one has $V^u(x) \subset \Lambda$.
Let $Y_\ell$ be a \emph{regular set}, i.e.\ a set of points for which the size of the local stable and unstable manifolds is bounded away from 0.  Given $r>0$, let
$Q_\ell(x) = \bigcup_{w\in Y_\ell \cap B(x,r)} V^u(w)$ for every $x\in Y_\ell$, where $V^u(w)$ is the local unstable manifold through $w$.  Denote by $\xi(x)$ the partition of $Q_\ell(x)$ by these manifolds, and let $\mu^u(w)$ be the conditional measure on $V^u(w)$ generated by $\mu$ with respect to the partition $\xi$.

\begin{definition}\label{def:SRB}
A hyperbolic invariant measure $\mu$ is called an \emph{SRB measure} if for any regular set $Y_\ell$ of positive measure and almost every $x\in Y_\ell$ and
$w\in Y_\ell \cap B(x,r)$, the conditional measure $\mu^u(w)$ is absolutely continuous with respect to the leaf volume $m_{V^u(w)}$ on $V^u(w)$.
\end{definition}

Every SRB measure has measure-theoretic entropy equal to the sum of its positive Lyapunov exponents; in fact, this is equivalent to the absolute continuity condition in Definition \ref{def:SRB} \cite{LY85}.

Certain ergodic properties follow automatically once we have an SRB measure.  For example, it was shown by  Ledrappier \cite{fL84} that the number of ergodic SRB measures supported on $\Lambda$ is at most countable, and that each ergodic SRB measure $\mu$ is Bernoulli up to a period (there is a $n\in \NN$ such that $(\Lambda, f^n, \mu)$ is conjugate to a Bernoulli shift).  In many cases one can in fact guarantee that the number of SRB measures is finite, or even that there is at most one SRB measure \cite{BV00,ABV00,HHTU11}.

\subsection{Effective hyperbolicity}
%\subsubsection{Measurable cone families}

Given $x\in M$, a subspace $E(x) \subset T_x M$, and $\tilde\theta(x) > 0$, the \emph{cone} at $x$ around $E(x)$ with angle $\tilde\theta(x)$ is
\begin{equation}\label{eqn:cone}
K(x,E(x),\tilde\theta(x)) = \{ v\in T_x M \mid \measuredangle (v,E(x)) < \tilde\theta(x) \}.
\end{equation}
If $E$ is a measurable distribution on $A\subset M$ and the angle function $\tilde\theta\colon A\to \RR^+$ is measurable, then~\eqref{eqn:cone} defines a \emph{measurable cone family} %$K(x)=K(x,E,\theta)$ 
on $A$.
Throughout the paper we will make the following standing assumption.
\begin{enumerate}[label=\textup{\textbf{(H)}}]
\item\label{H1} There exists a forward-invariant set $A\subset U$ of positive volume with two measurable cone families 
$K^s(x), K^u(x)\subset T_xM$ such that
\begin{enumerate}
\item $\overline{Df(K^u(x))} \subset K^u(f(x))$ for all $x\in A$;
\item $\overline{Df^{-1}(K^s(f(x)))} \subset K^s(x)$ for all $x\in f(A)$.
\item $K^s(x) = K(x,E^s(x),\tilde\theta_s(x))$ and $K^u(x) = K(x,E^u(x),\tilde\theta_u(x))$ are such that $T_x M = E^s(x) \oplus E^u(x)$; moreover $d_s = \dim E^s(x)$ and $d_u = \dim E^u(x)$ do not depend on $x$.
\end{enumerate}
\end{enumerate}

Such cone families automatically exist if $f$ is uniformly hyperbolic on $\Lambda$, or more generally, if there is a dominated splitting on a compact forward-invariant set (see \eqref{eqn:DS}); in this case $K^{s,u}$ are continuous.  We emphasize, however, that in our setting $K^{s,u}$ are not assumed to be continuous, but only measurable. Moreover, the families of subspaces $E^{u,s}(x)$ are not assumed to be invariant, although we could follow the procedure described in \cite[Remark 2.2]{CP16} to replace them with invariant families.

Let $A\subset U$ be a forward-invariant set satisfying \ref{H1}.  
Following \cite{CP16},  define $\lambda^u, \lambda^s\colon A\to \RR$ by\footnote{In \cite{CP16} these definitions are made for a family of maps $f_n\colon \RR^d\to \RR^d$ defined in a neighborhood of the origin; each orbit in $A$ gives such a family by writing $f$ in local coordinates around the points $x, f(x), f^2(x),\dots$.}
\begin{equation}\label{eqn:luls}
\begin{aligned}
\lambda^u(x) &= \inf \{\log \|Df(v)\| \mid v\in K^u(x), \|v\|=1 \}, \\
\lambda^s(x) &= \sup \{\log \|Df(v)\| \mid v\in K^s(x), \|v\|=1 \}.
\end{aligned}
\end{equation}
%If $K^{u,s}$ are the natural cone families associated to a dominated splitting, then $\lambda^s(x) < \lambda^u(x)$ at every point.  This may fail for the maps we consider (which do not necessarily have a dominated splitting), and we define
We define the \emph{defect from domination} at $x$ to be
\begin{equation}\label{eqn:defect}
\Delta(x) = \tfrac 1\alpha \max(0,\ \lambda^s(x) - \lambda^u(x)),
\end{equation}
where we recall that $\alpha \in (0,1]$ is the H\"older exponent of $Df$.\footnote{If $f$ is $C^r$ for some $r\geq 2$, then we have $\alpha = 1$; our techniques do not distinguish between $C^2$ maps and $C^r$ maps with $r>2$.  On the other hand, $\alpha>0$ is essential; our results do not apply to $C^1$ maps.}
Roughly speaking, $\Delta(x)$ controls how much the curvature of unstable manifolds can grow as we go from $x$ to $f(x)$; see Theorem \ref{thm:HP2} and Remark \ref{rmk:HP2} for further discussion of the role this quantity plays.  When $f$ has a dominated splitting, we have $\Delta(x)=0$ everywhere; see \S\ref{sec:DS}.

The following quantity is positive whenever $f$ expands vectors in $K^u(x)$ and contracts vectors in $K^s(x)$:  
\begin{equation}\label{eqn:le}
\lambda(x) = \min ( \lambda^u(x) - \Delta(x), -\lambda^s(x)).
\end{equation}
The \emph{upper asymptotic density} of 
$\Gamma\subset \NN$ is
%\begin{equation}\label{eqn:asymptotic-density}
$\ud(\Gamma) = \ulim_{N\to\infty} \tfrac 1N \#\Gamma \cap [0,N).$
%\end{equation}
An analogous definition gives the lower asymptotic density $\ld(\Gamma)$.
Denote the angle between the boundaries of $K^s(x)$ and $K^u(x)$ by
\begin{equation}\label{eqn:thetax}
\theta(x) = \inf\{\measuredangle(v,w) \mid v\in K^u(x), w\in K^s(x) \}.
\end{equation}
We say that a point $x\in A$ is \emph{effectively hyperbolic} if 
\begin{enumerate}[label=\textup{\textbf{(EH\arabic{*})}}]
\setcounter{enumi}{0}
\item\label{eqn:eh1} $\llim_{n\to\infty} \frac 1n \sum_{k=0}^{n-1} \lambda(f^k x) > 0$,
\end{enumerate}
and if in addition we have
\begin{enumerate}[label=\textup{\textbf{(EH\arabic{*})}}]
\setcounter{enumi}{1}
\item \label{eqn:eh2}
$\lim_{\bar\theta\to 0} \ud\{n \mid \theta(f^nx) < \bar\theta\} = 0$.
\end{enumerate}
A related (but not identical) definition of effective hyperbolicity is given in \cite{CP16}.  Condition \ref{eqn:eh1} says that not only are the Lyapunov exponents of $x$ positive for vectors in $K^u$ and negative for vectors in $K^s$, but  $\lambda^u$ gives enough expansion to overcome the `defect from domination' given by $\Delta$.\footnote{Because we first take the minimum of $\lambda^u$ and $-\lambda^s$, and then take a limit, \ref{eqn:eh1} also requires a certain amount of contraction in $K^s$ and expansion in $K^u$ to happen \emph{simultaneously}.  This can be weakened by the introduction of a more technical condition involving \emph{effective hyperbolic times}; see Condition \ref{eqn:eh3} in \S\ref{sec:eht}.  %The weaker condition there is useful because it avoids the requirement implicit in \eqref{eqn:le} and \ref{eqn:eh1} that expansion and contraction happen simultaneously along the orbit of $x$.
}
Condition \ref{eqn:eh2} requires that  the frequency with which the angle between the stable and unstable cones drops below a specified threshold $\bar\theta$ can be made arbitrarily small by taking the threshold to be small.  

%\begin{remark}\label{rmk:simul}
%In fact, Condition \ref{eqn:eh1} can be weakened slightly; see \ref{eqn:eh3} in \S\ref{sec:eht}.  The weaker condition there is useful because of the requirement implicit in \eqref{eqn:le} and \ref{eqn:eh1} that expansion and contraction happen simultaneously along the orbit of $x$.
%\end{remark}

If $\Lambda$ is a hyperbolic attractor for $f$, then \emph{every} point $x\in U$ is effectively hyperbolic.  If $f$ has a dominated splitting, then since Condition \ref{eqn:eh2} is automatic and $\Delta(x)=0$,  effective hyperbolicity of $x$ reduces to the requirement that the orbit of $x$ be asymptotically expanding along $E^u(x)$ and asymptotically contracting along $E^s(x)$, as in \cite{ABV00,BV00};\footnote{Modulo the simultaneity issue associated to \ref{eqn:eh1}, which can be addressed as in \S\ref{sec:eht}, and which also arises in \cite{ABV00}.} 
we discuss this in more detail in \S\ref{sec:DS}, after stating our main results.

\subsection{Main results}

Let $A$ satisfy \ref{H1}, and let $S\subset A$ be the set of effectively hyperbolic points.
Observe that effective hyperbolicity is determined in terms of a forward asymptotic property of the orbit of $x$, and hence $S$ is forward invariant under $f$.  

\begin{thma}\label{thm:main}
Let $f$ be a $C^{1+\alpha}$ diffeomorphism of a compact manifold $M$, and $\Lambda$ a topological attractor for $f$.  Assume that 
\begin{enumerate}
\item $f$ admits measurable invariant cone families as in \ref{H1}; 
\item the set $S$ of effectively hyperbolic points satisfies $\Leb S > 0$.
\end{enumerate}
Then $f$ has an SRB measure supported on 
$\Lambda$.
\end{thma}

A similar result can be formulated given information about the set of effectively hyperbolic points on a single `approximately unstable' submanifold\footnote{Such manifolds are usually called \emph{admissible}; the precise definition is not needed for the statement of Theorem \ref{thm:main2}, and will be given in \S\ref{sec:overview}.  All we need here is to have $T_x W \subset K^u(x)$ for `enough' points $x$.} $W\subset U$.  
Let $d_u$, $d_s$, and $A$ be as in \ref{H1}, and let $W\subset U$ be an embedded submanifold of dimension $d_u$; write $m_W$ for the volume induced on $W$ by the Riemannian metric.

\begin{thma}\label{thm:main2}
Let $f$ be a $C^{1+\alpha}$ diffeomorphism of a compact manifold $M$, and $\Lambda$ a topological attractor for $f$.  Assume that 
\begin{enumerate}
\item $f$ admits measurable invariant cone families as in \ref{H1};
\item there is a $d_u$-dimensional embedded submanifold $W\subset U$ such that $m_W(\{x\in S\cap W \mid T_xW \subset K^u(x)\}) > 0$.
\end{enumerate}
Then $f$ has an SRB measure supported on 
$\Lambda$.
\end{thma}

The geometric approach that we follow is to consider the  measures $\mu_n = \frac 1n \sum_{k=0}^{n-1} f_* m_W$, pass to a convergent subsequence $\mu_{n_k} \to \mu$, and prove that some ergodic component of $\mu$ is an SRB measure.  In \S\ref{sec:overview} we will describe how to establish this fact, but first in \S\ref{sec:DS} we discuss the relationship to previous similar results using dominated splittings, and then in  \S\ref{sec:app} we give some new examples to which our results can be applied.

%In \S\ref{sec:reduction} we will show that Theorem \ref{thm:main} is a consequence of Theorem \ref{thm:main2}.
%We will prove Theorem \ref{thm:main} by reducing it to Theorem \ref{thm:main2}; this is done in \S\ref{sec:reduction}.  
%To prove Theorem \ref{thm:main2}, we will 

\subsection{Related results}\label{sec:DS}

Let $f$ be a $C^2$ diffeomorphism and $A$ a forward-invariant compact set.  A splitting $T_A M = E^s \oplus E^u$ is \emph{dominated} if there is $\chi < 1$ such that
\begin{equation}\label{eqn:DS}
\|Df|_{E^s(x)}\| < \chi \|Df|_{E^u(x)}^{-1}\|^{-1} \text{ for all } x\in A;
\end{equation}
equivalently, the splitting is dominated if $\lambda^s(x) < \lambda^u(x)$ for all $x\in A$.\footnote{Here $\lambda^s,\lambda^u$ are as in \eqref{eqn:luls}, using $E^{s,u}$ in place of $K^{s,u}$.  Since we can make the cones arbitrarily small, the change in the definition does not affect any of our inequalities.}  In \cite{ABV00}, Alves, Bonatti, and Viana considered systems with a dominated splitting for which
\begin{itemize}
\item $E^s$ is uniformly contracting: $\lambda^s(x) \leq -\bl < 0$ for all $x\in A$;
\item $E^u$ is `mostly expanding': there is $\tilde S\subset A$ with positive volume and
\begin{equation}\label{eqn:NUE}
\llim_{n\to\infty} \frac 1n \sum_{j=0}^{n-1} \lambda^u(f^j x) > 0 \text{ for all } x\in \tilde S.
\end{equation}
\end{itemize}
Under these conditions they proved \cite[Theorem A]{ABV00} that $f$ has an SRB measure supported on $\Lambda = \bigcap_{j=0}^\infty f^j(A)$, and that the same result is true if \eqref{eqn:NUE} holds on a positive Lebesgue measure subset of some disk transverse to $E^s$.  A similar result for the (easier) case when $E^u$ is uniformly expanding and $E^s$ is mostly contracting was given in \cite{BV00}.  A stronger version of the result in \cite{ABV00} was recently given in \cite{ADLP} using a tower construction.

Given a dominated splitting with a uniformly contracting $E^s$, we see immediately from \eqref{eqn:defect} and \eqref{eqn:le} that $\Delta(x)=0$ and $\lambda(x) = \lambda^u(x)$ for all $x\in A$, so that \eqref{eqn:NUE} is equivalent to \ref{eqn:eh1}.  Moreover, by continuity and compactness, the angle between $E^u$ and $E^s$ is bounded away from 0, so \ref{eqn:eh2} is automatic, and we conclude that the set $\tilde S$ in the above result is exactly the set $S$ from Theorems \ref{thm:main} and \ref{thm:main2}.  In this sense, our results generalize the main results on existence of SRB measures from \cite{BV00,ABV00}.\footnote{The results there also give criteria for uniqueness of the SRB measure, as well as establishing that almost every point in $S$ is in the basin of some SRB measure; we do not consider these questions here.}

%Let us be a little more precise. Observe that the dominated splitting $E^s \oplus E^{cu}$ immediately gives continuous invariant cone families as in \ref{H1} such that $\lambda^s(x) \leq \log\chi < 0$ and $\lambda^s(x) \leq \lambda^u(x)$.  In particular, \ref{eqn:eh2} is automatically satisfied for every $x\in K$, and since $\Delta(x) = 0$ for every $x\in K$, \eqref{eqn:le} gives
%\[\lambda(x) = -\log \|Df^{-1}|_{E^{cu}_{f(x)}}\|,\]
%so \ref{eqn:eh1} reduces to \eqref{eqn:NUE}.  Thus in this setting, the set $H$ in Theorem \ref{thm:abv1} is exactly  the set $S$ of effectively hyperbolic points from Theorems \ref{thm:main} and \ref{thm:main2}, and the two versions of the positive volume condition in Theorem \ref{thm:abv1} correspond to the hypotheses of Theorem \ref{thm:main} and \ref{thm:main2}, respectively.  

The proof in \cite{ABV00} requires the notion of \emph{hyperbolic times}, introduced by Alves in \cite{jA00}.  These are times $n$ such that for some fixed $\sigma<1$, and every $0\leq k\leq n$, we have
\[
\prod_{j=n-k+1}^n \|Df^{-1}|_{E^{cu}_{f^j(x)}}\| \leq \sigma^k;
\text{ equivalently, }
 \sum_{j=n-k}^{n-1} \lambda^u(f^jx) \geq k|\log\sigma|.
\]
If $x$ satisfies \eqref{eqn:NUE}, then Pliss' lemma guarantees that the set of hyperbolic times for $x$ has positive lower asymptotic density.  A similar strategy runs through the heart of our main results: our conditions \ref{eqn:eh1}--\ref{eqn:eh2} guarantee a positive lower asymptotic density of \emph{effective hyperbolic times} at which we can apply a version of the Hadamard--Perron theorem proved in \cite{CP16}, allowing us to carry out the geometric construction of an SRB measure.

We point out that the simultaneity issue associated to \ref{eqn:eh1} arises already for dominated splittings.  Indeed, if one weakens the uniform contraction on $E^s$ and allows $\|Df|_{E^{s}(x)}\| > 1$ for some $x\in A$, then the approach in \cite{ABV00} requires more than merely combining \eqref{eqn:NUE} with the corresponding asymptotic inequality for $\lambda^s$.  As described in \cite[Proposition 6.4]{ABV00}, one can recover the result by requiring points in $H$ to have a positive asymptotic density of \emph{simultaneous $\sigma$-hyperbolic times}; that is, times $n$ such that
\begin{equation}\label{eqn:simul}
\prod_{j=n-k+1}^n \|Df^{-1}|_{E^{cu}_{f^j(x)}}\| \leq \sigma^k,
\quad
\prod_{j=n-k}^{n-1} \|Df|_{E^{cs}_{f^j(x)}}\| \leq \sigma^k
\end{equation}
for every $0\leq k\leq n$.  In the language of Theorems \ref{thm:main} and \ref{thm:main2}, the domination condition still gives $\Delta(x)=0$ everywhere, and so
\[
\lambda(x) = \min(-\log \|Df|_{E^{cs}_x}\|, -\log\|Df^{-1}|_{E^{cu}_{f(x)}}\|).
\]
Our condition \ref{eqn:eh1} guarantees a positive asymptotic density of `simultaneous effective hyperbolic times'.  In fact, our proofs can be carried out using a weaker condition that only requires us to control the stable direction for a period of time that is small relative to $n$; see \S\ref{sec:eht}, where \ref{eqn:eh1} is replaced with a condition \ref{eqn:eh3} that is easier to verify in some applications.

Overall, then, we can summarize the situation as follows.
In the geometric approach to construction of SRB measures, one needs good information on the dynamics and geometry of admissible manifolds and their images.  Ideally one wants \emph{hyperbolicity}: the unstable direction expands, the stable direction contracts.  If this happens all the time, we are in the uniformly hyperbolic setting and one can carry out the construction without too much trouble; this is described in \S\ref{sec:uh-construction}.  If hyperbolicity does not hold all the time, then we are in the non-uniformly hyperbolic setting and need two further conditions in order to play the game.
\begin{enumerate}
%table direction contracts.
\item Domination: if one of the directions does not behave hyperbolically, then it at least is still dominated by the other direction.
\item Separation: the stable and unstable directions do not get too close to each other.
\end{enumerate}
For the dominated splittings considered in \cite{BV00,ABV00}, these two conditions hold uniformly and so one only needs to control the asymptotic hyperbolicity (expansion and contraction along stable and unstable directions).  For our more general setting, both domination and separation may fail at some points, and in order to control the geometry and dynamics of images of admissible manifolds, we need to replace `hyperbolicity' with `effective hyperbolicity'.  The two conditions \ref{eqn:eh1} and \ref{eqn:eh2} control the failures of domination and separation, respectively: the presence of $\Delta(x)$ in \ref{eqn:eh1} lets us control curvature of admissible manifolds when domination fails, and the condition on $\theta(x)$ in \ref{eqn:eh2} guarantees that separation does not fail too often.

In \S\ref{sec:app} we describe applications of our main results to specific examples of non-uniformly hyperbolic systems.  In \S\ref{sec:overview} we give an overview of the strategy for the proofs.  Then in \S\ref{sec:main-pfs} we give the proofs, modulo some technical lemmas that we defer to \S\ref{sec:pf}.  In \S\S\ref{sec:general-conditions}--\ref{sec:katok-pf} we prove the results on applications from \S\ref{sec:app}.

\subsection*{Acknowledgments}
We are grateful to the anonymous referee for many comments that helped us to improve the exposition significantly, and to Agnieszka Zelerowicz for pointing out an error in an earlier version of \S\ref{sec:bounded-distortion}.

\section{Applications}\label{sec:app}

\subsection{Large local perturbations of Axiom A systems: abstract conditions}

We will describe a class of non-uniformly hyperbolic examples to which our main results can be applied, establishing existence of an SRB measure. These examples are obtained by beginning with a uniformly hyperbolic system and making a large local perturbation that satisfies certain conditions.  In \S \ref{sec:katok} we describe explicitly a family of maps satisfying these conditions -- these are dissipative versions of the Katok map \cite{aK79}.

Let $M$ be a $d$-dimensional smooth Riemannian manifold and 
$U\subset M$ an open set such that $\overline{U}$ is compact.  Let $f\colon U\to M$ be a $C^{1+\alpha}$ diffeomorphism onto its image with $\overline{f(U)}\subset U$, and let $\Lambda = \bigcap_{n\geq 0} f^n(U)$ be the attractor for $f$.  Assume that $\Lambda$ is a hyperbolic set for $f$, so that for every $x\in \Lambda$ we have
\begin{equation}\label{eqn:axioma}
\begin{aligned}
T_x M &= E^u(x) \oplus E^s(x), \\
\|Df(x)(v^u)\| &\geq \chi \|v^u\| \text{ for all } v^u \in E^u(x), \\
\|Df(x)(v^s)\| &\leq \chi^{-1} \|v^s\| \text{ for all } v^s \in E^s(x),
\end{aligned}
\end{equation}
where $\chi>1$ is fixed.  Note that we pass to an adapted metric if necessary.  Note also that since the splitting is continuous in $x$, it extends to a small neighborhood of $\Lambda$, and so in particular we may assume without loss of generality that \eqref{eqn:axioma} continues to hold for all $x\in U$.

We assume that the unstable distribution $E^u$ is one-dimensional, and consider a map $g\colon U\to M$ that is a $C^{1+\alpha}$ diffeomorphism  onto its image such that $g=f$ outside of an open set $Z\subset U$.  Conditions \ref{C1}--\ref{C3} below are formulated in terms of the action of $g$ as trajectories pass through $Z$.  We are most interested in the case when $Z$ is a small neighborhood of a fixed point, so that there are some points whose $g$-orbits never leave $Z$.

Let $G\colon U\setminus Z\to U\setminus Z$ be the first return map.  Given 
$\gamma>0$, let $K_\gamma^{s,u}(x)$ be the stable and unstable cones of width $\gamma$ for the unperturbed map $f$. We require the following condition:
\begin{enumerate}[label=\textbf{(C\arabic{*})}]
\item\label{C1}
There is $\gamma>0$ such that $\overline{DG(K_\gamma^u(x))} \subset K_\gamma^u(G(x))$ and $DG(K_\gamma^s(x)) \supset \overline{K_\gamma^s(G(x))}$ for every $x\in U\setminus Z$.
\end{enumerate}
Extend the cone families $K_\gamma^{u,s}(x)$ from $U\setminus Z$ to $Z$ by pushing them forward with the dynamics of $g$.  Condition \ref{C1} guarantees that we obtain measurable\footnote{Indeed, continuous everywhere except possibly the boundary of $Z$.} invariant cone families $K^{u,s}$ on all of $U$.  

Let $\AAA$ be the set of $C^{1+\alpha}$ curves $W\subset U\setminus Z$ such that $T_x W \subset K^u(x)$ for all $x\in W$.\footnote{For our purposes, it will not matter whether or not $W$ contains its endpoints, since these carry zero weight under $m_W$.}  We say that $W\in \AAA$ has \emph{H\"older curvature bounded by $L>0$} if the unit tangent vector to $W$ is $(L,\alpha)$-H\"older with respect to the point on the curve.  Fixing $L,\eps>0$, let $\tilde\AAA = \tilde\AAA(L,\eps)$ be the set of curves in $\AAA$ with length between $\eps$ and $2\eps$ and H\"older curvature bounded by  $L$.

Given $W\in \AAA$, we say that an \emph{admissible decomposition} for $W$ is a (possibly infinite) collection of disjoint subcurves $W_j \subset W$ and $\tau_j\in \NN$ such that $W \setminus \bigcup_j W_j$ is $m_W$-null and every $W_j$ satisfies $g^{\tau_j}(W_j) \subset U\setminus Z$.  Given an admissible decomposition, we write $\tau(x) = \tau_j$ for all $x\in W_j$, and $\bar G(x) = g^{\tau(x)}(x)$ for the induced map, so $\bar G(W_j) \subset U\setminus Z$. %  For $t\in \NN$, we write $W(t) = \{x\in W \mid \tau(x) = t\}$.

\begin{remark}
If $g(W) \subset U\setminus Z$, then any partition yields an admissible decomposition with $\tau\equiv 1$. When $g(W)$ enters $Z$, the time $\tau_j$ must be taken large enough to allow $\bar G(W_j)$ to escape $Z$.  We stress that $\bar G$ depends on $W$ and on the choice of admissible decomposition, and need not be the first return map $G$.  In our examples, 
$\bar G$ will be either $G$ or $G\circ g$.
\end{remark}

By invariance of $K^u$, we see that $\bar G(W_j)\in \AAA$.  The following condition requires that there be an admissible decomposition for which we control the size and curvature of $\bar G(W_j)$, as well as the expansion of 
$\bar G$ on $W_j$.

\begin{enumerate}[label=\textbf{(C\arabic{*})}]\setcounter{enumi}{1}
\item\label{C2}
There are $L,\eps,Q>0$ and $p\colon \NN\to [0,1]$ such that $\sum_{t\geq 1} t p(t)<\infty$ and every $W\in \tilde\AAA$ with $g(W)\cap Z\neq \emptyset$ has an admissible decomposition satisfying
\begin{enumerate}[label=(\roman{*})]
\item\label{integrable}
$m_W( \{x\in W \mid \tau(x) = t\}) \leq p(t)m_W(W)$ for all $t\in \NN$;
\item\label{largereturns}
$\bar G(W_j)\in \tilde\AAA$ for every $j$; and
\item\label{distortion}
if $x,y\in W_j$ then $\log\frac{|D\bar G(x)|_{T_xW}|}{|D\bar G(y)|_{T_yW}|} \leq Qd(\bar G(x),\bar G(y))^\alpha$.
\end{enumerate}
\end{enumerate}

\begin{remark}
Condition \ref{C2} is analogous to the familiar construction of an inducing scheme or tower; we stress, however, that we do not demand any Markov property.  The role of inducing time is played by $t$, which is such that at time $t$, each $W_j$ returns to uniformly large scale (this is \ref{largereturns}) with bounded distortion (this is \ref{distortion}).  We think of the function $p$ as a ``probability envelope'' that controls the probability of encountering different return times.  The condition $\sum t p(t)<\infty$, together with \ref{integrable}, corresponds to the requirement that inducing time be integrable (expected inducing time is finite).  Condition \ref{C3} below will guarantee that there is a choice of inducing time at which we have uniform hyperbolicity -- see Lemma \ref{lem:local-growth}.
 
In our examples, $g$ is obtained by slowing down $f$ near a fixed point.  In this case there is a natural admissible decomposition such that each $\bar G(W_j)$ has length between $\eps$ and $2\eps$, and so the challenge will be to prove an expansion estimate to verify \ref{integrable}, an estimate on H\"older curvature to verify \ref{largereturns}, and a bounded distortion estimate to verify \ref{distortion}.
\end{remark}

%Finally, we need to control the defect from domination of trajectories passing through $Z$.  
Given $x\in Z$, let $\Delta(x)$ be the defect from domination given by \eqref{eqn:defect}.
%$ = \frac 1\alpha \max(0, \lambda^s(x) - \lambda^u(x))$.  
We need to control the expansion, contraction, and defect when the trajectory of $x$ passes through $Z$, so that overall expansion, contraction, and defect of a trajectory can be controlled in terms of how often it enters $Z$.  We suppose that
\begin{enumerate}[label=\textbf{(C\arabic{*})}]
\setcounter{enumi}{2}
\item\label{C3}
there is $C>0$ such that given $W$ as in \ref{C2} and $x\in W$, we have
\begin{equation}\label{eqn:bdd-defect}
\sum_{j=k}^{\tau(x)} \lambda^u(g^j(x)) - \Delta(g^j(x)) \geq -C,\qquad
\sum_{j=k}^{\tau(x)} \lambda^s(g^j(x)) \leq C
\end{equation}
for every $0\leq k\leq \tau(x)$.  Moreover, we suppose that every orbit of $f$ leaving $Z$ takes more than $C/\log\chi$ iterates to return to $Z$.
\end{enumerate}

We give examples of systems satisfying the above conditions in the next section.  These conditions let us apply the main results to obtain an SRB measure.

\begin{theorem}\label{thm:perturbed}
Let $g$ be a $C^{1+\alpha}$ perturbation of an Axiom A system, such that $g$ satisfies conditions \ref{C1}--\ref{C3}.  Then $g$ has an SRB measure.
\end{theorem}

The proof of Theorem \ref{thm:perturbed} is given in \S\ref{sec:general-conditions} and goes as follows.  Given a small admissible curve 
$W\in \tilde\AAA$, we study the sequence of escape times through $Z$ for a trajectory starting at $x\in W$.  This is a sequence of random variables with respect to $m_W$, and while this sequence is not independent or identically distributed, \ref{C2} lets us control the average value of this sequence.  This in turn gives good bounds on the sum of $\lambda(x)$ along a trajectory, and also controls the frequency with which the angle between stable and unstable cones degenerates.  Ultimately, we will conclude that $m_W$-a.e.\ point 
$x\in W$ satisfies a weak version of effective hyperbolicity, and deduce existence of an SRB measure using a version of Theorem \ref{thm:main2}.\footnote{To get effective hyperbolicity as in \ref{eqn:eh1}, we would need to strengthen condition \ref{C3} and require that $\sum_{j=k}^{\tau(x)} \lambda(g^jx) \geq -C$ for each $k$, which does not automatically follow from \eqref{eqn:bdd-defect}.   To avoid verifying this stronger condition, we will replace \ref{eqn:eh1} with \ref{eqn:eh3} from \S\ref{sec:eht} below and use Theorem \ref{thm:main3} in place of Theorem \ref{thm:main2}; see Lemmas \ref{lem:local-growth} and \ref{lem:Chyptimes}.}
% because the bounds there are only one-sided.  The stronger condition would allow us to use Theorem \ref{thm:main2} in place of Theorem \ref{thm:main3} to prove Theorem \ref{thm:perturbed}.  However, \eqref{eqn:bdd-defect} is easier to verify.}

\subsection{Maps on the boundary of Axiom A: neutral fixed points}\label{sec:katok}

We give a specific example of a map for which the conditions of Theorem \ref{thm:perturbed} can be verified.  Let $f\colon U\to M$ be a $C^{1+\alpha}$ Axiom A diffeomorphism onto its image with $\overline{f(U)}\subset U$, where $\alpha\in (0,1)$.  Suppose that $f$ has one-dimensional unstable bundle.

Let $p$ be a fixed point for $f$ (if no such fixed point exists, take a periodic point $p$ and replace $f$ by an iterate that fixes $p$).  We perturb $f$ to obtain a new map $g$ that has an indifferent fixed point at $p$.  The case when $M$ is two-dimensional and $f$ is volume-preserving was studied by Katok~\cite{aK79}.  We allow manifolds of arbitrary dimensions and (potentially) dissipative maps. For example, one can choose $f$ to be the Smale-Williams solenoid or its sufficiently small perturbation. 

For simplicity, we suppose that there exists a neighborhood $Z\ni p$ with local coordinates in which $f$ is the time-1 map of the flow generated by
\begin{equation}\label{eqn:hypflow}
\dot x = Ax
\end{equation}
for some $A\in GL(d,\RR)$.  Assume that the local coordinates identify the splitting $E^u\oplus E^s$ with $\RR \oplus \RR^{d-1}$, so that $A = A_u \oplus A_s$, where $A_u = \gamma\Id_u$ and $A_s = -\beta\Id_s$ for some $\gamma,\beta>0$.  (This assumption of conformality in the stable direction is made primarily for technical convenience and should not be essential.)  Note that in the Katok example we have $d=2$ and $\gamma = \beta$ since the map is area-preserving.  In the more general setting when $\gamma\neq \beta$, many estimates from the original Katok example no longer hold.

Now we use local coordinates on $Z$ and identify $p$ with $0$.  Fix $0<r_0<r_1$ such that $B(0,r_1)\subset Z$, and let $\psi\colon Z\to [0,1]$ be a $C^{1+\alpha}$ function such that
\begin{enumerate}
\item $\psi(x)=\|x\|^\alpha$ for $\|x\|\leq r_0$;
\item $\psi(x)$ is an increasing function of $\|x\|$ for $r_0 \leq \|x\| \leq r_1$;
\item $\psi(x)=1$ for $\|x\|\geq r_1$.
\end{enumerate}
Let $\XX\colon Z\to \RR^d$ be the vector field given by $\XX(x) = \psi(x) Ax$.  Let $g\colon U\to M$ be given by the time-1 map of this vector field on $Z$ and by $f$ on $U\setminus Z$.  Note that $g$ is $C^{1+\alpha}$ because $\XX$ is $C^{1+\alpha}$.
The following is proved in \S\ref{sec:katok-pf}.

\begin{theorem}\label{thm:katok}
The map $g$ satisfies conditions \ref{C1}--\ref{C3}, hence $g$ has an SRB measure by Theorem \ref{thm:perturbed}.
\end{theorem}

\begin{remark}
Note that $g$ does not have a dominated splitting because of the indifferent fixed point, and hence this example is not covered by \cite{ABV00}.
We also observe that if $\psi$ is taken to be $C^\infty$ away from $0$, then $g$ is also $C^\infty$ away from the point $p$.  
The condition 
$\psi(x)=\|x\|^\alpha$ near $0$ for $\alpha<1$ takes the place of the condition in \cite{aK79} that $1/|\psi|$ be integrable, which ensured existence of a finite absolutely continuous invariant measure for the map $g$ in the case when $f$ is area preserving.
The verification of conditions \ref{C1}--\ref{C3} requires similar bounds as those proved for the original Katok map, but the computations are made more difficult by the fact that $\beta\neq\gamma$.  Moreover, in our case the attractor can intersect each stable manifold in a Cantor set, which differs from the behavior of the Katok map.
\end{remark}

\section{Overview of proofs of Theorems \ref{thm:main} and \ref{thm:main2}}\label{sec:overview}

\subsection{Description of geometric approach for uniformly hyperbolic attractors}\label{sec:uh-construction}

To motivate the approach that we will use later on, we first 
%construct an SRB measure for a uniformly hyperbolic attractor $\Lambda$ for $f$.  
consider the case when $\Lambda$ is a uniformly hyperbolic attractor for $f$.  In this case,
%Note that here 
the cones $K^u(x)$ and $K^s(x)$ are defined at every $x\in U$ and are continuous.
Let $W\subset U$ be an \emph{admissible manifold}; that is, a $d_u$-dimensional submanifold that is tangent to an unstable cone $K^u(x)$ at some point $x\in U$ and has a fixed size and uniformly bounded curvature.  
%Given an \emph{admissible manifold} $W \subset U$, 
%with leaf volume $m_W$ and a density function $\rho \in L^1(m_W)$, we refer to $(W,\rho)$ as a \emph{standard pair}.  
Consider leaf volume $m_W$ on $W$ and take the pushforwards $f_*^n m_W$ given by
\begin{equation}\label{eqn:pushLeb}
(f_*^nm_W)(E) = m_W(f^{-n}(E)).
\end{equation}
To obtain an invariant measure, we take C\'esaro averages:
\begin{equation}\label{eqn:mun-mW}
\mu_n := \frac 1n \sum_{k=0}^{n-1} f_*^k m_W.
\end{equation}
By weak* compactness there is a subsequence $\mu_{n_k}$ that converges to an invariant measure $\mu$ on $\Lambda$.  It is a classical result that $\mu$ is an SRB measure, and this can be proved in various ways.  We present an argument that can be adapted to our setting of effective hyperbolicity.

Consider the images $f^n(W)$ and observe that for each $n$, the measure $f_*^n m_W$ is absolutely continuous with respect to leaf volume on $f^n(W)$.  For every $n$, the image $f^n(W)$ can be covered with uniformly bounded multiplicity\footnote{This requires a version of the Besicovitch covering lemma; see Lemma \ref{lem:besicovitch}.} by a finite number of admissible manifolds $W_i$, so that 
\begin{equation}\label{eqn:cvx-comb}
f_*^n m_W \text{ is a convex combination of measures } \rho_i\, dm_{W_i},
\end{equation}
where $\rho_i$ are H\"older continuous positive densities on $W_i$.  We refer to each $(W_i,\rho_i)$ as a \emph{standard pair}; this idea of working with pairs of admissible manifolds and densities was introduced by Chernov and Dolgopyat in \cite{CD09} and is an important recent development in the study of SRB measures via geometric techniques.  

To proceed in a more formal way, fix constants $\gamma,\kappa,r>0$, and define a \emph{$(\gamma,\kappa)$-admissible manifold of size $r$} to be $V(x) := \exp_x \graph \psi$, where $\psi\colon B_{E^u(x)}(0,r) := B(0,r) \cap E^u(x) \to E^s(x)$ is $C^{1+\alpha}$ and satisfies
\begin{equation}\label{eqn:gkadm1}
\begin{aligned}
\psi(0) &= 0 \text{ and } D\psi(0) = 0,\\
\|D\psi\| &:= \sup_{\|v\| < r} \|D\psi(v)\| \leq \gamma, \\
|D\psi|_\alpha &:= \sup_{\|v_1\|,\|v_2\| < r} \frac{\|D\psi(v_1) - D\psi(v_2)\|}{\|v_1-v_2\|^\alpha} \leq \kappa.
\end{aligned}
\end{equation}

\begin{remark}
Our definition of admissible manifold is reminiscent of the notion of admissible manifolds in~\cite{BP07} and also of \emph{manifolds tangent to a cone field} used in~\cite{ABV00}.  There are several differences between those definitions and this one:  most importantly, in~\eqref{eqn:gkadm1} we require control not just of $\|D\psi\|$, but also of the H\"older constant of $D\psi$, so that we can bound the (H\"older) curvature of $V(x)$.  Furthermore, unlike~\cite{BP07}, we do not use Lyapunov coordinates, but rather work in the original Riemannian metric, and unlike~\cite{ABV00}, we look at the image of the manifold in a single tangent space $T_x M$, rather than in all the tangent spaces $T_y M$ for $y\in V(x)$.  While this is not important in the uniformly hyperbolic setting, where $K^u(x)$ is continuous in $x$, it will become crucial later on, when the cones $K^{u,s}$ are not even necessarily defined in all of the tangent spaces along $V(x)$.
\end{remark}

Now fix $L>0$ and write $\KK = (\gamma,\kappa,r,L)$ for convenience.  Then the space of admissible manifolds
\begin{multline*} %\label{eqn:adm-mfd-1}
\RRR_\KK := \{ \exp_x(\graph \psi) \mid x\in U, \psi \in  B_{E^u(x)}(r) \to E^s(x) \text{ satisfies } \eqref{eqn:gkadm1} \}
\end{multline*}
and the space of standard pairs
\[
\RRR'_\KK := \{ (W,\rho) \mid W\in \RRR_\KK, \rho \in C^\alpha(W,[\tfrac 1L,L]), |\rho|_\alpha \leq L \}
\]
can be shown to be compact in the natural product topology (see the next section for a more detailed description of this topology).
%To proceed  in a more formal way one should fix a collection $\KK$ of parameters, which control the geometry of the admissible manifolds (size, slope, and curvature) and the regularity of the density function (H\"older constant, lower and upper bounds); see \S\ref{sec:adm-srb} for further details.  Then the space $\RRR_\KK$ of admissible manifolds and the space $\RRR'_\KK$ of standard pairs are compact in a natural topology. 
A standard pair determines a measure $\Psi(W,\rho)$ on $\overline{U}$ in the obvious way:
\begin{equation}\label{eqn:Phi1}
\Psi(W,\rho)(E) := \int_{E\cap W} \rho \,dm_W.
\end{equation}
Moreover, each measure $\eta$ on  $\RRR'_\KK$  determines a measure $\Phi(\eta)$ on $\overline{U}$ by
\begin{equation}\label{eqn:Phi2}
\begin{aligned}
\Phi(\eta)(E) &:= \int_{\RRR'_\KK} \Psi(W,\rho)(E)\,d\eta(W,\rho) \\
&= \int_{\RRR'_\KK} \int_{E\cap W} \rho(x) \,d m_W(x) \,d\eta(W,\rho).
\end{aligned}
\end{equation}
Write $\MMM(\overline{U})$ and $\MMM(\RRR'_\KK)$ for the spaces of finite Borel measures on $\overline{U}$ and $\RRR'_\KK$, respectively.  It is not hard to show that $\Phi\colon \MMM(\RRR'_\KK)\to \MMM(\overline{U})$ is continuous; in particular, $\MMM_\KK:= \Phi(\MMM_{\leq 1}(\RRR_\KK'))$ is compact, where we write $\MMM_{\leq 1}$ for the space of measures with total weight at most 1.

On a uniformly hyperbolic attractor, an invariant probability measure is an SRB measure if and only if it is in $\MMM_\KK$ for some $\KK$.   We see from \eqref{eqn:cvx-comb} that $\MMM_\KK$ is invariant under the action of $f_*$, and thus $\mu_n \in \MMM_\KK$ for every $n$.  By compactness of $\MMM_\KK$ one can pass to a convergent subsequence $\mu_{n_k} \to \mu\in \MMM_\KK$, and this is the desired SRB measure.

%the action of $f_*$ on $\MMM(\overline{U})$ induces an action on $\MMM(\RRR'_\KK)$.  In particular, this means that for every $n$, there is $\eta_n \in \MMM(\RRR'_\KK)$ such that $\mu_n = \Phi(\eta_n)$.  By compactness of $\MMM(\RRR'_\KK)$ we may assume that $\eta_{n_k} \to \eta\in \MMM(\RRR'_\KK)$, giving
%\[
%\mu = \lim_k \mu_{n_k} = \lim_k \Phi(\eta_{n_k}) = \Phi(\lim_k \eta_{n_k}) = \Phi(\eta),
%\]
%so the limiting measure $\mu$ is in $\Phi(\MMM(\RRR'_\KK))$, and hence is an SRB measure.

\subsection{Constructing SRB measures with effective hyperbolicity}\label{sec:srb-eh}

Now we move to the setting of Theorems \ref{thm:main} and \ref{thm:main2}, so we assume that $A\subset U$ is a forward-invariant set such that \ref{H1} holds.  We will see in \S\ref{sec:reduction} that the hypotheses of Theorem \ref{thm:main} imply the hypotheses of Theorem \ref{thm:main2}, so here we consider a $d_u$-dimensional manifold $W\subset U$ for which $m_W(S)>0$, where we write $S$ for the set of effectively hyperbolic points $x\in W \cap A$  with the property that $T_x W \subset K^u(x)$.  In this setting, there are two major obstacles to overcome.
\begin{enumerate}
\item The action of $f$ along admissible manifolds is not necessarily uniformly expanding.
\item Given $n\in \NN$ it is no longer necessarily the case that $f^n(W)$ contains any admissible manifolds in $\RRR_\KK$, let alone that it can be covered by them.  When $f^n(W)$ contains some admissible manifolds, we will need to control how much of  it  can be covered.
\end{enumerate}
To address the first of these obstacles, we need to consider admissible manifolds for which we control not only the geometry but also the dynamics; thus we will replace the collection $\RRR_\KK$ from the previous section with a more carefully defined set (in particular, $\KK$ will include more parameters).  Since we do not have uniformly transverse invariant subspaces $E^{u,s}$, our definition of an admissible manifold also needs to specify which subspaces are used, and the geometric control requires an assumption about the angle between them.  

Given $\theta,\gamma,\kappa,r>0$, write $\II = (\theta,\gamma,\kappa,r)$ and consider the following set of ``$(\gamma,\kappa)$-admissible manifolds of size $r$ with transversals controlled by $\theta$'':
\begin{multline}\label{eqn:PPP}
\PPP_\II = \{ \exp_x(\graph \psi) \mid x\in A,\ T_x M = G \oplus F,\
G \subset \overline{K^u(x)}, 
\\ F\subset \overline{K^s(x)}, \measuredangle(G,F) \geq \theta,\  
\psi\in C^{1+\alpha}(B_G(r), F) \text{ satisfies~\eqref{eqn:gkadm1}}\}.
\end{multline}
Elements of $\PPP_\II$ are admissible manifolds with controlled geometry.  We also impose a condition on the dynamics of these manifolds.  Fixing $C, \bl>0$, write $\JJ = (C,\bl)$ and consider for each $N\in \NN$ the collection of sets
\begin{multline}\label{eqn:QQQ}
\QQQ_{\JJ,N} = \{f^N(V_0) \mid V_0 \subset U, \text{ and for every } y,z \in V_0, \text{ we have} \\
d(f^j(y), f^j(z)) \leq Ce^{-\bl(N- j)} d(f^N(y),f^N(z)) \text{ for all } 0\leq j\leq N\}.
\end{multline}
%(Note that in particular we have $W\subset f^N(U)$ for every $W\in \QQQ_{\JJ,N}$.)  
Elements of $\PPP_\II \cap \QQQ_{\JJ,N}$ are admissible manifolds with controlled geometry and dynamics in the unstable direction.  When we give the details of the proof, we will also introduce a parameter $\beta>0$ that controls the dynamics in the stable direction, and another parameter $L>0$ that controls densities in standard pairs (as before).  Then writing $\KK = \II \cup \JJ \cup \{\beta,L\}$, we will define in \eqref{eqn:RRR} a set $\RRR_{\KK,N} \subset \PPP_\II \cap \QQQ_{\JJ,N}$ for which we have the added restriction that we control the dynamics in the stable direction; the corresponding set of standard pairs will be written $\RRR_{\KK,N}'$.  
%the stable dynamics (see \eqref{eqn:RRR}), and $\RRR_{\KK,N}'$ for the corresponding set of standard pairs.

The set $\RRR'_{\KK,N}$ carries a natural product topology; an element of $\RRR'_{\KK,N}$ is specified by a quintuple $(x,G,F,\psi,\rho)$, and a small neighborhood $\Omega \ni x$  can be identified with $\RR^n$ via the exponential map.  Then the second coordinate can be identified with the set of all $k$-dimensional subspaces of 
$\RR^n$, the third with all $(n-k)$-dimensional subspaces, the fourth with $C^1$ functions $B_{\RR^k}(r)\to \RR^{n-k}$, and the fifth with $C^0$ functions $B_{\RR^k}(r)\to [\frac 1L,L]$.  This specifies a natural topology on each coordinate: the Grassmanian topology on the subspaces $G$ and $F$, and the $C^1$ and $C^0$ topologies on the functions $\psi$ and $\rho$, respectively.  Thus we may define a topology on $\RRR'_{\KK,N}$ as the product topology over each such Euclidean neighborhood in $U$.  In Proposition \ref{prop:R-cpt}, we prove that $\RRR_{\KK,N}'$ is compact in this topology and that the map $\Phi$ defined in \eqref{eqn:Phi2} is continuous.

As before, let $\MMM_{\leq 1}(\RRR'_{\KK,N})$ denote the space of measures on $\RRR'_{\KK,N}$ with total weight at most 1.  The resulting measures on $U$ will play a central role in our proof:
\begin{equation}\label{eqn:Mach}
\MMM_{\KK,N} = 
\Phi(\MMM_{\leq 1}(\RRR'_{\KK,N})).
\end{equation}
%Thus $\MMM_{\KK,N}$ comprises measures that are supported on admissible manifolds with uniformly bounded size and curvature, and that are absolutely continuous on these manifolds, with well-behaved density functions; furthermore, the admissible manifolds supporting these measures must display some uniform hyperbolicity for $N$ backwards iterates.
One should think of $\MMM_{\KK,N} \subset \MMM(\overline{U})$ as an analogue of the regular level sets that appear in Pesin theory.  Measures in $\MMM_{\KK,N}$ have uniformly controlled geometry, dynamics, and densities via the parameters in $\KK$, and Proposition \ref{prop:R-cpt} gives compactness of $\MMM_{\KK,N}$. However, at this point we encounter the second obstacle mentioned above: because $f(W)$ may not be covered by admissible manifolds in $\RRR_{\KK,N}$, the set $\MMM_{\KK,N}$ is not $f_*$-invariant.  

Thus we must establish good recurrence properties to $\MMM_{\KK,N}$ under the action of $f_*$ on $\MMM(\overline{U})$; this will be done via effective hyperbolicity.
%, and consider for each $N\in \NN$ the collection $\RRR_{\KK,N}$ of admissible manifolds $W$ for which the action of $f^{-k}$ along $W$ is controlled by these parameters for each $0\leq k\leq N$.  Let $\RRR'_{\KK,N}$ be the corresponding collection of standard pairs.  
Consider for $x\in A$ and $\bl>0$ the set of \emph{effective hyperbolic times}
\begin{equation}\label{eqn:ehtimes}
\Gamma^e_\bl(x) = \bigg\{n \mid \sum_{j=k}^{n-1} (\lambda^u - \Delta)(f^jx) \geq \bl(n-k) \text{ for all }0\leq k<n\bigg\}.
\end{equation}
Any effective hyperbolic time is a hyperbolic time as well, but not every hyperbolic time is effective.
In \S\ref{sec:eht} we use results from \cite{CP16} to show that the set $\Gamma^e_\bl(x)$ has positive lower asymptotic density for a positive volume set of $x$, and that for almost every effective hyperbolic time $n\in \Gamma^e_\bl(x)$, there is a neighborhood $W_n^x\subset W$ containing $x$ such that $f^n(W_n^x) \in \PPP_\II \cap \QQQ_{\JJ,N}$.  With a little more work (see Lemma \ref{lem:eht2} and Proposition \ref{prop:unif-large}), we will produce a `uniformly large' set of points $x$ and times $n$ such that $f^n(W_n^x) \in \RRR_{\KK,N}$, and in fact $f_*^n m_{W_n^x} \in \MMM_{\KK,N}$. % \Phi(\RRR'_{\KK,N})$.
%In \S\S\ref{sec:adm-srb}--\ref{sec:find-srb}, we use this fact to address a difficulty created by the second obstacle above: $\MMM(\RRR'_{\KK,N})$ is not invariant under the action induced by $f_*$, since $f(W)$ may not be covered by admissible manifolds in $\RRR_\KK$.  The way out is to use the fact that at effective hyperbolic times, $f_*^n m_{W_n^x}$ returns to $\Phi(\RRR'_{\KK,N})$, and 
We use this to obtain measures $\nu_n \in \MMM_{\KK,N}$ such that 
\begin{equation}\label{eqn:nun-ulim}
\nu_n \leq \mu_n = \tfrac 1n \textstyle\sum_{k=0}^{n-1} f_*^k m_W
\qquad\text{and}\qquad
\ulim_{n\to\infty} \|\nu_n\| > 0.
\end{equation}
Once this is achieved, we can use compactness of $\MMM_{\KK,N}$ to conclude that there is a non-trivial $\nu \in \bigcap_N \MMM_{\KK,N}$ such that $\nu \leq \mu = \lim_k \mu_{n_k}$.  In order to apply the absolute continuity properties of $\nu$ to the measure $\mu$, we define in \eqref{eqn:NR}--\eqref{eqn:Mac} a collection $\Mac$ of measures with good absolute continuity properties along admissible manifolds, for which we can prove a version 
of the Lebesgue decomposition theorem (Proposition \ref{prop:decomposition}) that gives $\mu = \mu^{(1)} + \mu^{(2)}$, where $\mu^{(1)}\in \Mac$ is invariant.    This measure is non-trivial since $0\neq \nu \leq \mu^{(1)}$, and the definition of $\RRR'_{\KK,N}$ will guarantee that the set of points with non-zero Lyapunov exponents has positive measure with respect to $\nu$, and hence also with respect to $\mu^{(1)}$.  Thus some ergodic component of $\mu^{(1)}$ is hyperbolic, and hence is an SRB measure.

\section{Proof of Theorems \ref{thm:main} and \ref{thm:main2}}\label{sec:main-pfs}

In this section we prove our main results, modulo some technical lemmas whose proofs we defer to \S\ref{sec:pf} so as not to disrupt the exposition here.  

\subsection{Reduction to a density condition}

We start by observing that Theorem \ref{thm:main} is a consequence of Theorem \ref{thm:main2}: the following is proved in \S\ref{sec:reduction}.

\begin{proposition}\label{prop:reduction}
Let $f$ be a $C^{1+\alpha}$ diffeomorphism of a compact manifold $M$, and $\Lambda$ a topological attractor for $f$.  Assume that $f$ admits measurable invariant cone families as in \ref{H1}, and that the set $S$ of effectively hyperbolic points satisfies $\Leb S > 0$.  Then there is a $d_u$-dimensional embedded submanifold $W\subset U$ such that $m_W(\{x\in S\cap W \mid T_xW \subset K^u(x)\}) > 0$.
\end{proposition}

Theorem \ref{thm:main2} will in turn follow from Theorem \ref{thm:main3} below, which is slightly more technical to state but will prove useful in our applications.

%The definition of effective hyperbolicity  uses \ref{eqn:eh1}, which in turn relies on controlling the quantity $\lambda(x)$ along a trajectory.  Positivity of $\lambda(x)$ requires simultaneous expansion in $K^u$ and contraction in $K^s$.  %In fact it is possible to obtain Theorem \ref{thm:main2} using a weaker condition (which is somewhat more technical), which we now present. 

Given 
$C,\bl>0$ and $q\in\NN$, write $\JJ=(C,\bl)$ as before and let
\begin{multline}\label{eqn:shtimes}
\Gamma^s_{\JJ,q}(x) = \big\{ n\in \NN \mid
\|Df^{-k}(f^nx)(v)\|\geq Ce^{\bl k}\|v\|\\
 \text{ for all } 0\leq k\leq q \text{ and } v\in K^s(f^nx)\big\}.
\end{multline}
This is similar to the condition in \eqref{eqn:ehtimes} on the dynamics in the unstable direction, but only requires control of the dynamics for iterates in $[n-q,n]$ instead of all iterates in $[0,n]$.  Pliss' lemma \cite[Lemma 11.2.6]{BP07} shows that if $x$ satisfies
\begin{equation}\label{eqn:growth}
\llim_{n\to\infty} \frac 1n \sum_{k=0}^{n-1} \lambda(f^k x) > \chi > \bl > 0,
\end{equation}
then for every $q\in \NN$, we have
$
\ld\left( \Gamma^e_\bl(x) \cap \Gamma^s_{(1,\bl),q}(x)\right) \geq \frac{\bar\lambda - \chi}{L - \chi},
$
where $L = \sup_x \lambda(x)$, the set $\Gamma^e_\bl(x)$ is defined in \eqref{eqn:ehtimes}, and $\Gamma^s_{(1,\bl),q}(x)$ is given by  \eqref{eqn:shtimes} with $C=1$.  In particular, every effectively hyperbolic $x$ has the property that there is $\JJ = (C,\bl)$ such that
\begin{enumerate}[label=\textup{\textbf{(EH\arabic{*}$'$)}}]
\setcounter{enumi}{0}
\item\label{eqn:eh3} $\lim_{q\to\infty} \ld\left(\Gamma^e_\bl(x) \cap \Gamma^s_{\JJ,q}(x)\right) > 0$.
\end{enumerate}
Consider the set $\hat S = \{x\in A \mid \text{\ref{eqn:eh3} and \ref{eqn:eh2} hold} $ for some $C,\bl>0\}$.  Then $S\subset \hat S$: every effectively hyperbolic point is contained in $\hat S$.  On the other hand, \ref{eqn:eh3} is weaker than \ref{eqn:eh1}, so $\hat S$ may be strictly larger than $S$.
%there may be points that are not effectively hyperbolic in the sense of \ref{eqn:eh1}, but nevetheless satisfy \ref{eqn:eh3} and are in $\hat S$. %contained in $\hat S$.  As in \eqref{s_w}, given a $d_u$-dimensional embedded submanifold $W\subset U$, let 
%$\hat S_W = \{x\in \hat S\cap W \mid T_xW \subset K^u(x)\}$.
We devote the rest of \S\ref{sec:main-pfs} to proving the following result.

\begin{thm}\label{thm:main3}
Let $f$ be a $C^{1+\alpha}$ diffeomorphism of a compact manifold $M$, and $\Lambda$ a topological attractor for $f$.  Assume that 
\begin{enumerate}
\item $f$ admits measurable invariant cone families as in \ref{H1}, and
\item there is a $d_u$-dimensional embedded submanifold $W\subset U$ such that $m_W( \{x\in \hat S\cap W \mid T_xW \subset K^u(x)\}) > 0$.
\end{enumerate}
Then $f$ has an SRB measure supported on 
$\Lambda$.
%Theorems \ref{thm:main} and \ref{thm:main2} continue to hold if $S$ is replaced by $\hat S$.\foot{Restate hypotheses here}
\end{thm}

Because $S\subset \hat S$, Theorem \ref{thm:main3} implies Theorem \ref{thm:main2}.  
When we apply our main results to the applications described in \S\ref{sec:app}, we will find it easier to check the weaker condition \ref{eqn:eh3} rather than the more restrictive \ref{eqn:eh1}.

\subsection{A Hadamard--Perron theorem for effective hyperbolic times}\label{sec:eht}

Our first major step in the proof of Theorem \ref{thm:main3} is a version of the Hadamard--Perron theorem that works at effective hyperbolic times; this was proved in \cite{CP16}.  Here we give a statement of this result that is adapted to the notation and terminology of Theorem \ref{thm:main3}, and includes an elementary integration bound that follows from the assumption on $m_W(S)$.  This lemma is proved in \S\ref{sec:real-eht}, where we recall the  precise statement of the result from \cite{CP16}.

%From \cite{CP16}, we have a Hadamard--Perron theorem for effective hyperbolic times, which we recall in \S\ref{sec:real-eht} and combine with some elementary integration estimates to prove the following.

\begin{lemma}\label{lem:real-eht}
Under the conditions of Theorem \ref{thm:main3}, there are  $\delta,\bar\theta,\bar\gamma,\bar\kappa,\bar{r}>0$  such that writing $\II = (\bar\theta,\bar\gamma,\bar\kappa,\bar{r})$ and letting $\JJ$ be as in \ref{eqn:eh3}, the following is true.  For every $x\in A$ there is $\Gamma_{\bl}^{e'}(x) \subset \Gamma_{\bl}^e(x)$ such that writing $\Gamma_q^\JJ(x) := \Gamma_{\bl}^{e'}(x) \cap \Gamma_{\JJ,q}^s(x)$, we have
\begin{enumerate}
\item given $x\in W$ and $n\in \Gamma_{\bl}^{e'}(x)$, there is 
$
W_n^x \subset W \cap B(x,\bar{r}e^{-\bar\lambda n})
$
with
\begin{equation}\label{eqn:eh-admissible}
f^nW_n^x \in \PPP_{\II}
\cap \QQQ_{\JJ,n};
\end{equation}
\item for every $q\in \NN$ we have
\begin{equation}\label{eqn:often-Gamma}
\llim_{n\to\infty}
\int_{W} \frac 1n \#([0,n) %\cap \Gamma^{e'}_\bl(x) \cap \Gamma^s_{C,\bl,q}(x)) 
\cap \Gamma_q^\JJ(x)) \,dm_W(x) > \delta.
\end{equation}
\end{enumerate}
\end{lemma}

To prove Theorem \ref{thm:main3}, we will use Lemma \ref{lem:real-eht} to show that on average, a large part of the measures $f_*^n m_W$ returns to the space of measures corresponding to standard pairs over $\PPP_\II \cap \QQQ_{\JJ,q}$.  We will also need to control the densities of these pairs, and to guarantee that for the limiting measure, transverse directions have negative Lyapunov exponents (the definition of $\QQQ_{\JJ,q}$ will guarantee positive Lyapunov exponents along the admissible manifolds).  We describe the conditions here, and then in \S\ref{sec:find-srb} we formulate (Theorem \ref{thm:buildSRB}) a general set of criteria for a sequence of measures to have a limit point that has an SRB measure as an ergodic component.

Given $W\in \PPP_{\II} \cap \QQQ_{\JJ,q}$ (in particular, $W\subset f^q(U)$), consider the set
\begin{multline}\label{eqn:HCLN}
H_{\JJ,q}(W) = \{y \mid T_y M = (T_yW) \oplus G \text{ for some } G\\ 
 \text{ with }
\|Df^{-j}(y)|_G^{-1}\| \leq Ce^{-\bl j}
 \text{ for every } 0\leq j\leq q\},
\end{multline}
of all points in $W$ whose backwards trajectories of length $q$ have expansion controlled by $(C,\bl)$ in a direction transverse to $TW$.
Given $\beta,L>0$, we write $\KK=\II \cup \JJ \cup (\beta,L)$ and consider
\begin{equation}\label{eqn:RRR}
\begin{aligned}
\RRR_{\KK,q} &= \bigg\{W\in \PPP_{\II} \cap \QQQ_{\JJ,q}
\mid \frac{m_W\big(H_{\JJ,q}(W)\big)}{m_W(W)} \geq \beta\bigg\}, \\
\RRR'_{\KK,q} &= \{(W,\rho) \mid W\in \RRR_{\KK,q}, \rho \in C^\alpha(W,[\tfrac 1L,L]), |\rho|_\alpha \leq L \}.
\end{aligned}
\end{equation}

%Recall that $\MMM_{\KK,q} = \Phi(\MMM_{\leq 1}(\RRR'_{\KK,q}))$.

\begin{proposition}\label{prop:R-cpt}
In the natural product topology defined in \S\ref{sec:srb-eh}, the set $\RRR_{\KK,q}'$ is compact  and the map $\Phi\colon \MMM(\RRR_{\KK,q}')\to \MMM(\overline{U})$ defined in \eqref{eqn:Phi2} is continuous.  In particular, $\MMM_{\KK,q} = \Phi(\MMM_{\leq 1}(\RRR'_{\KK,q}))$ is weak*-compact.
\end{proposition}

Now we give a condition under which a part of $m_W$ returns to $\MMM_{\KK,q}$.  Given $0\leq q\leq n$, consider the set
\begin{equation}\label{eqn:SJqn}
S_{\JJ,q,n} := \{x\in W \mid n\in \Gamma_q^\JJ(x)\}.
\end{equation}

\begin{lemma}\label{lem:eht2}
For every $\beta'>0$ there is $\beta>0$  such that the following is true.  If $0\leq q\leq n$ and $x\in S_{\JJ,q,n}$, and moreover the set $W_n^x \subset W$ from Lemma \ref{lem:real-eht} satisfies
\begin{equation}\label{eqn:mWnx}
m_W(W_n^x \cap S_{\JJ,q,n}) \geq \beta' m_W(W_n^x),
\end{equation}
then we have $f_*^n m_{W_n^x} \in \MMM_{\KK,q}$.
\end{lemma}

% we recall that each standard pair $(W,\rho)$ defines a measure on $U$ as in \eqref{eqn:Phi1}, and that writing $\MMM(\RRR'_{\KK,N})$ for the space of finite Borel measures on $\RRR'_{\KK,N}$,  there is a natural map $\Phi\colon \MMM(\RRR'_{\KK,N}) \to \MMM(U)$ as in \eqref{eqn:Phi2}: associate to each $\eta \in \MMM(\RRR'_{\KK,N})$ the measure $\Phi(\eta)\in \MMM(U)$ given by
%\begin{equation}\label{eqn:Phi*}
%\begin{aligned}
%\int_M a \,d\Phi(\eta) &:= \int_{\RRR'_{\KK,N}} \Phi(W,\rho)(a) \,d\eta(W,\rho) \\
%&= \int_{\RRR'_{\KK,N}}\int_W a(x) \rho(x) \,dm_W(x)\,d\eta(W,\rho).
%\end{aligned}
%\end{equation}
%(It is worth emphasising that $\Phi$ is not one-to-one.)  

We will use \eqref{eqn:often-Gamma} from Lemma \ref{lem:real-eht} to verify \eqref{eqn:mWnx} for `many' points $x$; then Lemma \ref{lem:eht2} will give a large part of $\mu_n = \frac 1n \sum_{k=0}^{n-1} f_*^k m_W$ sitting in $\MMM_{\KK,q}$.  
We will need
%To obtain the set $S'_{\JJ,q,n}$, we need 
the following version of the Besicovitch covering lemma for the `balls' $W_n^x \subset W$ with centres $x\in S_{\JJ,q,n}$.  We give a proof in \S\ref{sec:besicovitch}.

\begin{lemma}\label{lem:besicovitch}
There is $p\in \NN$, depending only on $d_u, \bar\theta, \bar\gamma, \bar\kappa$, and $\bar r$, such that for every $n\in \NN$ and every $0\leq q\leq n$, there are subsets $A_1,\dots, A_p \subset S_{\JJ,q,n}$ such that
\begin{itemize}
\item $S_{\JJ,q,n} \subset \bigcup_{i=1}^p \bigcup_{x\in A_i} W_n^x$;
\item for each $1\leq i\leq p$ and $x,y\in A_i$, we have $x=y$ or $W_n^x \cap W_n^y = \emptyset$.
\end{itemize}
\end{lemma}

In \S\ref{sec:pf-unif-large}, we use Lemma \ref{lem:besicovitch} to produce for each $0\leq q\leq n$ a finite set $S'_{\JJ,q,n} \subset S_{\JJ,q,n}$ such that
\begin{enumerate}
\item the sets $W_n^x$ associated to each $x\in S_{\JJ,q,n}'$ are disjoint;
\item the union of these sets is `large': writing
\[
W_{\JJ,q,n} := \bigcup_{x\in S_{\JJ,q,n}'} W_n^x,
\] 
for each  $\beta'>0$, one can choose $S_{\JJ,q,n}'$ such that for all $q,n$, we have
\[
m_W(W_{\JJ,q,n}) \geq \tfrac 1p m_W(S_{\JJ,q,n}) - \beta' m_W(W).
%\text{  for every $q$ and $n$.}
\]
\end{enumerate} 
Then we will use Lemma \ref{lem:eht2} to conclude that $f_*^n m_{W_{\JJ,q,n}} \in \MMM_{\KK,q}$; averaging the lower bound over $n$ and using \eqref{eqn:often-Gamma} will give a sequence of measures with the following property.
 
%More precisely, given %measures $\nu,\mu\in\MMM(U)$, we write $\nu\le\mu$ if $\nu(E)\le\mu(E)$ for every measurable set $E\subset M$. Given 
\begin{definition}
Given a sequence of measures $\mu_n \in \MMM (U)$ and a sequence of numbers $q_n\to\infty$, we say that $\mu_n$ have \emph{uniformly large projections onto} $\MMM_{\KK,q_n}$ if there exist $\delta>0$ and a sequence of measures 
$\eta_n\in\MMM(\RRR_{\KK,q_n}')$ such that $\Phi(\eta_n)\le\mu_n$ and 
$\|\Phi(\eta_n)\|\ge\delta$ for every $n$, where $\nu\leq \mu$ means that $\nu(E) \leq \mu(E)$ for every measurable set $E$.
\end{definition}

The above discussion outlines the proof of the following result, whose details are given in \S\ref{sec:pf-unif-large}.

\begin{proposition}\label{prop:unif-large}
Under the hypotheses of Theorem \ref{thm:main3}, there is a sequence of numbers $q_n\to\infty$ such that the measures $\mu_n = \frac 1n \sum_{k=0}^{n-1} f_*^k m_W$ have uniformly large projections onto $\MMM_{\KK,q_n}$.
\end{proposition}

 %the uniformly large projection required by Proposition \ref{prop:unif-large}, as long as the sequence $q_n$ grows slowly enough.

%The proof of Proposition \ref{prop:unif-large}  is given in \S\ref{sec:pf-unif-large} and is essentially a careful application of (a version of) the Besicovitch covering lemma.  %Lemma \ref{lem:often-Gamma} is proved in \S\ref{sec:pf-often-Gamma}, and 
%Lemma \ref{lem:eht2} is proved in \S\ref{sec:eht2}, and then to prove Theorem \ref{thm:main3} it 

Once Proposition \ref{prop:unif-large} is proved, it only remains to show that the ``uniformly large projections'' condition  gives an SRB measure; we discuss this next.

\subsection{General criteria for existence of an SRB measure}\label{sec:find-srb}

\begin{thm}\label{thm:buildSRB}
Let $M$ be a Riemannian manifold, $U\subset M$ an open set such that $\overline{U}$ is compact, and 
$f\colon U\to M$ a $C^{1+\alpha}$ diffeomorphism onto its image such that 
$\overline{f(U)}\subset U$. Let also $\mu_n$ be a sequence of measures on 
$M$ that converges in the weak* topology to an invariant measure $\mu$.
Suppose that there exists $\KK=(\theta, \gamma, \kappa,r,C,\bl,\beta,L)$ and a sequence $q_n\to\infty$ such that the measures $\mu_n$ have uniformly large projections onto $\MMM_{\KK,q_n}$. Then $\mu$ has an ergodic component that is an SRB measure.
\end{thm}

In this section we prove Theorem \ref{thm:buildSRB}, modulo two further technical results that we prove in \S\ref{sec:pf}.  The main idea is to define asymptotic versions of $\RRR_{\KK,N}$, $\RRR_{\KK,N}'$, $\MMM_{\KK,N}$ to help characterize the limiting measure $\mu = \lim \mu_n$ in relation to the set of SRB measures.  To this end, let
\begin{equation}\label{eqn:backcontr}
\RRR_\KK = \bigcap_{N\in\NN} \RRR_{\KK,N},\qquad
\RRR = \bigcup_\KK \RRR_\KK,
\end{equation}
where the union in the final definition is taken over all $\theta,\gamma,\kappa,r,C,\bl,\beta>0$, and we define $\RRR'_\KK$ and $\RRR'$ similarly.%\foot{This means that $(W,\rho) \in \RRR'$ has $\rho\in C^\alpha$, whereas before it was $L^1$.  Does it matter?}

Note that elements of $\RRR_\KK$ and $\RRR$ are genuine unstable manifolds (not just admissible), because we control the entire backwards trajectory.  We write $\MMM_\KK = \Phi(\MMM(\RRR_\KK'))$ for the collection of all measures on $\overline{U}$ that can be given in terms of $\KK$-uniform standard pairs.
%Having established the manifolds with which we work, we consider the following collection of standard pairs:
%\begin{equation}
%\RRR' = \{(W,\rho) \mid W\in \RRR, \,\rho\in L^1(W,m_W)\}.
%\end{equation}
%We will also need to consider a collection of standard pairs where the densities $\rho$ are controlled uniformly.  Fixing $L>0$, we abuse notation slightly by implicitly including $L$ in $\KK$ and putting
%\begin{equation}\label{eqn:RK}
%\RRR'_{\KK} = \{(W,\rho) \mid W\in \RRR_\KK, \,\rho\in C^\alpha(W,[1/L,L]), \,|\rho|_\alpha\leq L\},
%\end{equation}\\
% We make a similar definition for $\RRR'_{\KK,N}$.
%Now we can define collections of measures on $U$ with certain absolute continuity properties that we will need for an SRB measure.  First we need to consider the collection of sets that are null sets for some cover by admissible manifolds.  More precisely, let
In the proof, we will need another version of the absolute continuity properties of measures in $\MMM_\KK$.  Let
\begin{multline}\label{eqn:NR}
\NNN(\RRR) = \Big\{ E\subset U \mid \text{there is $\WWW \subset \RRR$ such that $E \subset \textstyle\bigcup_{W\in \WWW} W$} \\
\text{and } m_W(E) = 0 \text{ for all } W\in \WWW \Big\}.
\end{multline}
Writing $\MMM_f(\Lambda)$ for the set of all invariant measures on $U$ (which must all be supported on the attractor $\Lambda$), we consider
\begin{equation}\label{eqn:Mac}
\begin{aligned}
\Mac &= \{\mu\in \MMM(U) \mid \mu(E)=0 \text{ for all } E\in \NNN(\RRR)\}, \\
\Mh &= \{ \mu \in \MMM_f(\Lambda) \mid \text{all Lyapunov exponents of $\mu$ are non-zero}\}.
\end{aligned}
\end{equation}
Observe that elements of $\Mh$ are $f$-invariant, while elements of $\Mac$ may not be.  However, the collection $\Mac$ is invariant; if $\mu\in \Mac$, then $f_*\mu \in \Mac$.
The following proposition is proved in \S\ref{sec:MacSRB}.
\begin{proposition}\label{prop:MacSRB}
The intersection $\Mac\cap \Mh$ is precisely the set of SRB measures for $f$. %, and for every $\KK$ we have $\MMM_\KK \subset \Mac$.
\end{proposition}

Let $(\Mac)^\perp$ be the set of measures $\nu\in \MMM(U)$ such that $\nu \perp \mu$ for every $\mu\in \Mac$.  The following is proved in \S\ref{sec:decomposition}.

\begin{proposition}\label{prop:decomposition}
Given any measure $\mu \in \MMM(U)$, there are unique measures $\mu^{(1)}\in \Mac$ and $\mu^{(2)} \in (\Mac)^\perp$ such that $\mu = \mu^{(1)} + \mu^{(2)}$.
\end{proposition}

%\begin{remark}
%The dimension of the admissible manifolds $W$ is not specified -- indeed, $\RRR$ contains manifolds at points $x$ whose dimension is not maximal among admissible manifolds through $x$.  However, the definition of $\Mac$ forces $\mu$ to give zero weight to any set that can be covered by admissible manifolds of any dimension such that it is a null set for each, which in particular forces $\mu$ to be absolutely continuous in all unstable directions.
%\end{remark}

\begin{remark}
The collections $\Mac$ and $\MMM_\KK$ are built using two different notions of absolute continuity.\footnote{This is comparable to the situation in \cite[\S8.6]{BP07}, where one can define absolute continuity using either null sets, or densities, or holonomy maps.}  The criterion for inclusion in $\Mac$ is that a measure gives zero weight to a certain collection of null sets defined using the reference measures $m_W$, while the criterion for inclusion in $\MMM_\KK$ is that a measure be defined in terms of integrating densities against measures $m_W$.  The second criterion immediately implies the first, so $\MMM_\KK \subset \Mac$. 
% If there was only a single admissible manifold $W$, with corresponding reference measure $m_W$, then the converse implication (up to weakening the regularity conditions set by $\KK$) would be the Radon--Nikodym theorem.  However, there are many different admissible manifolds $W$ in $\mathcal{R}$, and so in our setting it is not clear whether the converse holds.
\end{remark}

Now we prove Theorem \ref{thm:buildSRB}, taking Propositions \ref{prop:MacSRB} and \ref{prop:decomposition} as given.  Suppose that we have measures $\mu_n\to \mu\in \MMM(\Lambda,f)$ with uniformly large projections onto $\MMM_{\KK,q_n}$; that is, there are $q_n\to\infty$ and $\nu_n = \Phi(\eta_n)\in \MMM_{\KK,q_n}$ such that $\nu_n\leq \mu_n$ and $\|\nu_n\| \geq \delta > 0$ for every $n$.  

By Proposition \ref{prop:decomposition}, there is a unique decomposition $\mu = \mu^{(1)} + \mu^{(2)}$ such that $\mu^{(1)}\in \Mac$ and $\mu^{(2)} \in (\Mac)^\perp$.  Note that $\Mac$ and $(\Mac)^\perp$ are both $f_*$-invariant: thus $f$-invariance of $\mu$ gives $\mu = f_*\mu = f_*\mu^{(1)} + f_*\mu^{(2)}$, and uniqueness of the decomposition gives $f_*\mu^{(1)}=\mu^{(1)}$ and $f_*\mu^{(2)}=\mu^{(2)}$.

First we show that $\mu^{(1)}$ is non-trivial, using the measures $\nu_n$.    Note that
\begin{equation}\label{eqn:nun}
\nu_n \in \MMM_{\KK,q_n} \subset \MMM_{\KK,N} \text{ whenever } q_n \geq N,
\end{equation}
and also that $q_n\to\infty$.  Each $\nu_n$ is contained in $\MMM_{\KK,1}$, which is weak* compact by Proposition \ref{prop:R-cpt}, so there is a convergent subsequence $\nu_{n_k} \to \nu$.  For every $N$, we see from \eqref{eqn:nun} that $\nu_{n_k}\in \MMM_{\KK,N}$ for sufficiently large $k$, and thus $\nu\in \MMM_{\KK,N}$.  In particular, $\nu\in \MMM_{\KK} \subset \Mac$.
The relation $\nu_n\leq \mu_n$ passes to the limit, and we get $\nu\leq \mu$.  Because $\mu^{(2)}\in (\Mac)^\perp$ and $\nu\in  \Mac$, we conclude that $\nu\leq \mu^{(1)}$.  Now $\mu^{(1)}$ is non-trivial because $\|\nu\|\geq \delta > 0$.

Now we show that some ergodic component $\zeta$ of $\mu^{(1)}$ is an SRB measure.  By Proposition \ref{prop:MacSRB} it suffices to show that some $\zeta$ is in both $\Mh$ and $\Mac$.  In fact, since any ergodic component $\zeta$ of $\mu^{(1)}$ has $\zeta\ll \mu^{(1)}$, we see that \eqref{eqn:NR}--\eqref{eqn:Mac} give $\zeta \in \Mac$, and to complete the proof of Theorem \ref{thm:buildSRB}, it suffices to show that some ergodic component $\zeta$ of $\mu^{(1)}$ is hyperbolic.

To this end, 
%Using this inequality together with the fact that elements of $\MMM_\KK$ have some hyperbolicity, we show that $\mu^{(1)}$ has an ergodic component in $\Mh$.  
let $E$ be the set of Lyapunov regular points for which all Lyapunov exponents are non-zero; we claim that $\mu^{(1)}(E) > 0$.  Indeed, for every $(W,\rho)\in \RRR'_{\KK}$, \eqref{eqn:RRR} gives $m_W(E) \geq \beta \|m_W\|$, and the bounds on $\rho$ give
\[
\Psi(W,\rho)(E) = \int_{E\cap W} \rho(x)\,dm_W(x) \geq \frac\beta L \|m_W\|
\geq \frac\beta{L^2} \int_W \rho(x)\,dm_W(x),
\]
where $\Psi$ is as in \eqref{eqn:Phi1}.
Since $\nu = \Phi(\eta)$ for some $\eta\in \MMM(\RRR'_\KK)$, we have
\[
\nu(E) = \int_{\RRR'_\KK} \Psi(W,\rho)(E) \,d\eta \geq \beta L^{-2} \int_{\RRR'_\KK} \|\Psi(W,\rho)\| \,d\eta = \beta L^{-2} \|\nu\| > 0.
\]
Recalling that $\mu^{(1)} \geq \nu$, this implies that some ergodic component $\zeta$ of $\mu^{(1)}$ has $\zeta(E)>0$.  By ergodicity this gives $\zeta(E)=1$, hence $\zeta$ is hyperbolic.  In particular, we obtain $\zeta\in \Mh \cap \Mac$, and Proposition \ref{prop:MacSRB} shows that $\zeta$ is an SRB measure.

\section{Proof of intermediate technical results}\label{sec:pf}

\subsection{Proof of Proposition \ref{prop:reduction}}\label{sec:reduction}

The angle between the cones $K^s(x), K^u(x)$ is given by the measurable function $\theta(x)$.  Let $X_n = \{x\in A \mid \theta(x) > 1/n\}$, and observe that $S \subset \bigcup_{n\geq 1} X_n$, so there exists $n$ such that $\Leb(S \cap X_n) > 0$.  By restricting $S$ to include only points in $X_n$, we may assume without loss of generality that $\theta(x)\geq\bar\theta = \frac 1n$ for all $x\in S$, while still guaranteeing that $\Leb S > 0$.

By again decreasing $S$ to a smaller set that still has positive Lebesgue measure, we will obtain a foliation of a neighborhood of $S$ such that every leaf $V(x)$ of the foliation is tangent to $K^u(x)$ -- that is, $T_x V(x) \subset K^u(x)$ for every $x\in S$.

To this end, for every $z\in M$ we fix a neighborhood $0\in B_z \subset T_z M$ such that $\exp_x^{-1}$ is well-defined on $\exp_z(B_z)$ for every $x\in \exp_z(B_z)$.  Let $\pi_{z,x} := \exp_x^{-1} \circ \exp_z\colon B_z \to T_x M$, and observe that $\pi_{z,x} = I_{z,x} + g_{z,x}$, where $I_{z,x} = D\pi_{z,x}$ is an isometry and where $g_{z,x}$ is a smooth map (as smooth as the manifold) with $D g_{z,x}(0) = 0$.  %(The magnitude of the function $g$ is related to the curvature of $M$ near $z$; if $M$ is Euclidean then $g_{z,x} \equiv 0$.)

As $d(z,x)$ goes to zero, the quantity $\|g_{z,x}\|_{C^2}$ goes to zero as well.  It follows that we can choose for every $z\in U$ a neighborhood $\Omega_z$ such that
\begin{enumerate}
\item $\exp_x^{-1}$ is well-defined on $\Omega_z$ for all $x\in \Omega_z$;
\item $\|g_{x,z}\|_{C^2} \leq \bar\theta/2$ for all $x\in \Omega_z$.%, where $g_{x,z} = \pi_{x,z} - D_0 \pi_{x,z}$.
\end{enumerate}
Now we choose a finite set $E \subset U$ such that the neighborhoods $\{\Omega_z \mid z\in E\}$ cover the entire trapping region $U$.  Thus we can fix $z\in E$ such that $\Leb(S \cap \Omega_z) > 0$.

Let $\GGG_z^{d_u}$ denote the collection of $d_u$-dimensional subspaces of $T_z M$, endowed with the metric given by angle (or equivalently the Hausdorff metric on the intersections of subspaces with the unit sphere).  Given $P\in \GGG_z^{d_u}$, denote by $E_x^P$ the affine subspace of $T_z M$ passing through $\exp_z^{-1}(x)$ parallel to $P$.  Given $x\in \Omega_z \cap S$, let
\[
\Delta_x = \{P\in \GGG_z^{d_u} \mid T_x \exp_z (E_x^P \cap B_z) \subset K^u(x) \};
\]
that is, $\Delta_x$ comprises those $d_u$-dimensional subspaces $P\subset T_zM$ for which the manifold $\pi_{z,x} (E_x^P\cap B) \subset T_x M$ has a tangent space at $0$ that lies in the unstable cone $K^u(x)$.

Each such $\Delta_x$ contains a ball of radius $\bar{\theta}/2$ in $\GGG_z^{d_u}$.  To see this, consider the cone $I_{z,x}^{-1} K^u(x) \subset T_z M$, which contains a ball of radius $\bar{\theta}$ in $\GGG_z^{d_u}$, and observe that if $P\in \GGG_z^{d_u}$ lies in this cone and is at least a distance of $\bar{\theta}/2$ from its boundary, then $P\in \Delta_x$ since $\|g_{z,x}\|_{C^2} \leq \bar{\theta}/2$.

Now let $F\subset \GGG_z^{d_u}$ be a finite $\bar{\theta}/2$-dense set, and given $P\in F$, let $X_P = \{x\in \Omega_z \cap S \mid P\in \Delta_x\}$.  Then $\bigcup_{P\in F} X_P = S \cap \Omega_z$, and by finiteness, there exists $P\in F$ such that $\Leb X_P > 0$.  

It follows that by restricting $S$ to include only points in $X_P$, we may assume that $T_x \exp_z (E_x^P\cap B_z) \subset K^u(x)$ for every $x\in S$.  Denote by $\xi=\xi(P)$ the foliation of $\Omega_z$ whose leaves are $V(x) := \exp_z(E_x^P\cap B_z)$.

To summarize, we assume without loss of generality that
\begin{enumerate}
\item $S \subset \Omega_z \ni z$, where $\exp_z^{-1}$ is well-defined on the open set $\Omega_z$;
\item there is a $d_u$-dimensional subspace $P\subset T_z M$ so that $\exp_z(E_x^P\cap B_z)$ is $(\gamma,1)$-admissible at $x$ for all $x\in S$.  In particular, the foliation of $T_z M$ by $d_u$-dimensional affine subspaces parallel to $P$ projects under $\exp_z$ to a foliation $\xi$ of $\Omega_z$ such that for every $x\in S$, the leaf $V(x)$ of $\xi$ passing through $x$ is $(\gamma,1)$-admissible and $T_x V(x) \subset K^u(x)$.
\end{enumerate}
Now we have a smooth foliation $\xi$ such that every leaf $V$ of $\xi$ is in $\PPP_{(\bar\theta, \bar\gamma, 1, \bar{r})}$.  There are density functions $\rho_V \in L^1(V,m_V)$ for every $V\in \xi$ and a measure $\eta \in \MMM(\PPP_{(\bar\theta,\bar\gamma,1,\bar r)})$ such that
\[
\Leb(S) = \int_\PPP \int_{V \cap S} \rho_V(x) \,dm_V(x) \,d\eta(V).
\]
Since $\Leb S>0$, there exists $W\in \xi \subset \PPP_{(\bar\theta,\bar\gamma,1,\bar r)}$ such that $m_{W}(S)>0$, where we restrict $S$ to include only those points $x$ for which $T_x W \in K^u(x)$. This proves Proposition \ref{prop:reduction}.

\subsection{Proof of Lemma \ref{lem:real-eht}}\label{sec:real-eht}

For Lemma \ref{lem:real-eht} we need to use a version of the Hadamard--Perron theorem from 
 \cite{CP16}, which we now describe.
 Let $\Omega\subset \RR^d$ be a neighborhood of the origin, and $f_n \colon \Omega \to \RR^d$ a sequence of $C^{1+\alpha}$ diffeomorphisms onto their images; ultimately we will take $f_n$ to be the representation of $f$ in local coordinates around $f^n(x), f^{n+1}(x)$ along a trajectory.  For now we just assume that $\|Df_n\|$, $\|Df_n^{-1}\|$, and 
$|Df_n|_\alpha$ are uniformly bounded by some constant $e^L$,\footnote{This is distinct from the $L$ that appears in the collection of constants $\KK$ and in \eqref{eqn:RRR}, which plays no role in the proof of Lemma \ref{lem:real-eht}.  We use $L$ here for consistency with \cite{CP16}.}
and that there is a sequence of splittings $\RR^d = E_n^u \oplus E_n^s$ that determine a $Df_n(0)$-invariant sequence of cones $K_n^u,K_n^s$; that is,
$\overline{Df_n(0)(K_n^u)}\subset K_{n+1}^u$ and 
$\overline{Df_n(0)^{-1}(K_{n+1}^s)} \subset K_n^s$ for each $n$.  We stress that there is no invariance condition on the subspaces $E_n^{u,s}$ themselves; see \cite[Remark 2.2]{CP16}.
We also assume that there is
$L'>0$ such that the angle $\theta_n$ between $K_n^u$ and $K_n^s$ satisfies $\theta_{n+1}\geq e^{-L'}\theta_n$ for all $n$.
As in \eqref{eqn:luls} and \eqref{eqn:defect}, consider
\begin{equation}\label{eqn:luls2defect}
\begin{aligned}
\lambda_n^u &= \inf \{\|Df_n(0)(v)\| \mid v\in K_n^u, \|v\|=1\},\\
\lambda_n^s &= \sup \{\|Df_n(0)(v)\| \mid v\in K_n^s, \|v\|=1\},\\
\Delta_n &= \tfrac 1\alpha \max(0,\ \lambda_n^s - \lambda_n^u).
\end{aligned}
\end{equation}
Let $L'' = \max\left( \frac{L'}\alpha,L(1+\frac2\alpha)\right)$.  Fix $\bar\theta>0$ and let
\begin{equation}\label{eqn:lne}
\lambda_n^e = \begin{cases}
\lambda_n^u - \Delta_n & \theta_n \geq \bar\theta,\\
-L'' & \theta_n<\bar\theta.
\end{cases}
\end{equation}
The following theorem is a consequence of \cite[Theorem A]{CP16}. 

\begin{thm}\label{thm:HP2}
Given $L,L',\bar\theta,\bl>0$ there are $\bar{\gamma},\bar{\kappa},\bar{r},\delta>0$ such that the following is true.  If $\theta_0\geq \bar\theta$ and $n\in\NN$ is such that
\begin{equation}\label{eqn:ehHP2}
\sum_{j=k}^{n-1}\lambda_k^e \geq (n-k)\bl \text{ for all }0\leq k<n,
\end{equation}
then $\theta_n\ge\ba$; moreover, if $\psi_0\colon B_{E_0^u}(\eta)\to E_0^s$ is an arbitrary $C^{1+\alpha}$ function with $\psi_0(0)=0$, $D\psi_0(0)=0$, and $|D\psi_0|_\alpha\leq\bar\kappa$, then for each $n$ satisfying \eqref{eqn:ehHP2} there exists a $C^{1+\alpha}$ function 
$\psi_n\colon B_{E_n^u}(\bar{r}) \to E_n^s$ satisfying the following conditions:
\begin{enumerate}
\item $\psi_n(0)=0$, $D\psi_n(0)=0$;
\item $\|D\psi_n\| \leq \bar\gamma$ and $|D\psi_n|_\alpha\le\bar\kappa$;
\item the graph of $\psi_n$ is the connected component of $F_n(\graph(\psi_0))\cap B_{E_n^u}(\bar{r}) \times B_{E_n^s}(\bar\gamma \bar{r})$ containing the origin;
\item writing $F_n = f_{n-1} \circ\cdots\circ f_1 \circ f_0$, if 
$F_n(x),F_n(y) \in\graph(\psi_n)$, then for every $0\leq k\le n$, we have
\begin{equation}\label{eqn:hyptimeHP2}
\|F_n(x) - F_n(y)\| \geq e^{(n-k)\bl} \|F_k(x) - F_k(y)\|.
\end{equation}
\end{enumerate}
\end{thm}

\begin{remark}\label{rmk:HP2}
The key property of times satisfying the hypotheses of Theorem \ref{thm:HP2} is that not only are the infinitesimal dynamics along the first $n$ iterates uniform, but so are the (finite-scale) geometry and dynamics of manifolds close to the unstable direction.  The inclusion of the quantity $\Delta_n$ in the expression for $\lambda_n^e$ is crucial for this; when $\lambda_n^s > \lambda_n^u$, the curvature of the image $F_n(\graph(\psi_0))$ (that is, the quantity $|D\psi_n|_\alpha$) can increase, and we are forced to `cut off' part of the image in order to control $\|D\psi_n\|$; the amount that we cut off is controlled by $\Delta_n$, and in order to return to size $\bar{r}$ we must wait until the expansion given by $\lambda_n^u$ overcomes the defect introduced by $\Delta_n$.
\end{remark}

Given $W$ and $\hat{S}$ as in the hypothesis of Theorem \ref{thm:main3}, for each $x\in \hat {S} \cap W$ with $T_x W \subset K^u(x)$ we want to apply Theorem~\ref{thm:HP2} to the sequence of maps $T_{f^n(x)}M \to T_{f^{n+1}(x)}M$ given by writing $f$ in local coordinates, choosing $E_n^u = T_{f^n(x)}(f^n(W))$, and taking $E_n^s$ to be any subspace in $K^s(f^n(x))$.  Note that the existence of $L,L'$ satisfying the uniformity bounds above comes immediately from compactness of $\overline{U}$ and smoothness of $f$.

In order to apply Theorem \ref{thm:HP2}, we need to consider points $x\in W$ at which $T_xW \subset K^u(x)$ and $\theta(x)$ is sufficiently large.  
To do this, consider for each $M\in \NN$ the set
\begin{multline*}
S_M = \{x\in \hat S \cap W \mid T_x W \subset K^u(x), \theta(x)\geq \tfrac 1M, 
\text{ and } \\
\ld(\Gamma_\bl^e(x) \cap \Gamma_{\JJ,q}^s(x)) \geq \tfrac 1M  \text{ for all } q\in \NN.\}
\end{multline*}
By the hypothesis of Theorem \ref{thm:main3}, we have $m_W(\bigcup_M S_M) > 0$, and so there is $M$ with $m_W(S_M) > 0$.

If $x\in W$ is such that $\theta(x) < \frac 1M$ or $T_x W \not\subset K^u(x)$, then we put $\Gamma^{e'}_\bl(x) = \emptyset$.  For every other $x\in W$ -- that is, whenever $\theta(x) \geq \frac 1M$ and $T_x W \subset K^u(x)$ -- we put
\begin{equation}\label{eqn:Gammane'}
\Gamma_{\bar\lambda}^{e'}(x) = \bigg\{ n \mid \sum_{j=k}^{n-1} \lambda_j^e \geq \bl(n-k) \text{  for all }0\leq k<n \bigg\}.
\end{equation}
This is exactly the set of times $n\in \NN$ at which \eqref{eqn:ehHP2} holds and Theorem \ref{thm:HP2} can be applied.  Note that $\lambda^u - \Delta \leq \lambda$ and so $\Gamma_\bl^{e'}(x) \subset \Gamma_\bl^e(x)$, but the containment may be proper since whenever $\theta(f^j(x)) < \bar\theta$ we have $\lambda_j^e = -L''$, which may be less than $\lambda_j^u - \Delta_j$.  Nevertheless, 
by \cite[Proposition 9.3]{CP16}, there is $\bar\theta \in (0, \tfrac 1M]$ such that 
\begin{equation}\label{eqn:eh4}
\ld\left(\Gamma^{e'}_\bl(x) \cap \Gamma^s_{\JJ,q}(x)\right) > \tfrac 1{2M} \text{ for every } q\in \NN \text{ and } x\in S_M.
\end{equation}
For this value of $\bar\theta$, let $\bar\gamma,\bar\kappa,\bar{r},\delta$ be as in Theorem \ref{thm:HP2}.  Writing $\II = (\bar\theta,\bar\gamma,\bar\kappa,\bar{r})$, it follows from Theorem \ref{thm:HP2} that for every $n\in \Gamma^{e'}_\bl(x)$ there is $W_n^x \subset W$ such that $f^nW_n^x \in \PPP_\II \cap \QQQ_{\JJ,n}$, where the geometric bounds for $\PPP_\II$ come from the first three conclusions of Theorem \ref{thm:HP2}, and the dynamical bound for $\QQQ_{\JJ,n}$ comes from \eqref{eqn:hyptimeHP2}.  This bound also shows that $W_n^x \subset B(x,\bar{r}e^{-\bl n})$, which proves the first part of Lemma \ref{lem:real-eht}.

For the second part of Lemma \ref{lem:real-eht}, fix $q,m\in \NN$ and let
\[
X_m^q = \{x \mid \tfrac 1n \#([0,n) \cap \Gamma_q^\JJ(x)) \geq \tfrac 1{2M} \text{ for every } n\geq m\}.
\]
By \eqref{eqn:eh4}, we have $\bigcup_m X_m^q = S_M$ for every $q$, and in particular, there is $N(q)$ such that $m_W(X_{N(q)}^q) \geq \frac 12 m_W(S_M)$.  Thus for every $n\geq N(q)$ we have
\begin{multline*}
\int_W \frac 1n \#([0,n) \cap \Gamma_q^\JJ(x)) \,d m_W(x)
\geq \int_{X_{N(q)}^q} \frac 1n \#([0,n) \cap \Gamma_q^\JJ(x))\,dm_W(x)
\\
\geq \tfrac 1{2M} m_W(X_{N(q)}^q) \geq \tfrac 1{4M} m_W(S_M) > 0.
\end{multline*}
Taking $\delta = \frac 1{4M} m_W(S_M)$ completes the proof of Lemma \ref{lem:real-eht}.

\subsection{Proof of Proposition \ref{prop:R-cpt}}\label{sec:R-cpt}

Continuity of $\Psi\colon\RRR_{\KK,N}'\to \MMM(U)$ is immediate, 
and also implies continuity of $\Phi\colon \MMM(\RRR_{\KK,N}')\to\MMM(U)$.  We show that $\RRR_{\KK,N}'$ is compact.

First we note that each of the geometric conditions for inclusion in $\PPP_{\II}$ is compact -- that is, given $(x_k,G_k,F_k,\psi_k)$ such that \eqref{eqn:PPP} is satisfied, compactness of the Grassmanian for $M$ guarantees existence of a subsequence for which $(x_k,G_k,F_k)$ converges to $(x,G,F)$; the uniform bound on the angle guarantees that we still have $T_x M=G\oplus F$ in the limit; and the Arzel\`a--Ascoli theorem guarantees that we can get $C^1$ convergence of $\psi_k$ by passing to a further subsequence.  Moreover, the limiting function $\psi$ is $C^{1+\alpha}$ and satisfies the same bounds in terms of $\gamma$ and $\kappa$.

Using the Arzel\`a--Ascoli theorem again gets a subsequence for which the densities $\rho_k$ converge (in $C^0$) to $\rho\in C^\alpha(W,[1/L,L])$ satisfying $|\rho|_\alpha\leq L$.  Thus to verify compactness of $\RRR_{\KK,N}'$, it remains only to verify that when $W_k\in \RRR_{\KK,N}$ is a sequence with $W_k \to W \in \PPP_\II$ in the sense of the previous paragraph, then $W\in \RRR_{\KK,N}$ as well.  This requires us to 
check the dynamical conditions \eqref{eqn:QQQ}, \eqref{eqn:HCLN}, and \eqref{eqn:RRR} controlled by $C,\bl,N,\beta$.

It is straightforward that \eqref{eqn:QQQ} passes to the limit, so if $W_k\to W$ and $W_k\in \QQQ_{\JJ,N}$, then $W\in \QQQ_{\JJ,N}$ as well.  For inclusion in $\RRR_{\KK,N}$, we observe that if $W_k\to W$ and $x_k\in H_{\JJ,N}(W_k)$ are such that $x_k\to x$, then  $x\in H_{\JJ,N}(W)$.  In other words,
\[
H_{\JJ,N}(W) \supset \bigcap_{m\in\NN} \overline{\bigcup_{k\geq m} H_{\JJ,N}(W_k)}.
\]
Write $Y_m = \overline{\bigcup_{k\geq m} H_{\JJ,N}(W_k)}$, so $Y_m$ is a nested sequence of compact sets.  

Let $\nu_W = m_W/\|m_W\|$, and similarly for $\nu_{W_k}$.  Compactness of $Y_m$ and the fact that $\nu_{W_k}\to \nu_W$ (in the weak* topology) guarantees that $\nu_W(Y_m)\geq \ulim_{k\to\infty} \nu_{W_k}(Y_m) \geq\beta$, where we have used the fact that $Y_m\supset H_{\JJ,N}(W_k)$ for all $k\geq m$ and the condition in \eqref{eqn:RRR}.  Because the $Y_m$ are nested and we have $\nu_W(Y_m)\geq \beta$ for every $m$, we conclude that $\nu_W(H_{\JJ,N}(W)) \geq \beta$, which completes the proof that $(W,\rho)\in \RRR_{\KK,N}'$.

Now we have shown that $\MMM_{\leq 1}(\RRR_{\KK,N}')$ is weak* compact, and continuity of $\Phi$ completes the proof that $\MMM_{\KK,N}$ is weak* compact.

\subsection{Proof of Lemma \ref{lem:eht2}}\label{sec:pf-eht2}

The first part of the proof of Lemma \ref{lem:eht2} is to find and control the density function for $f_*^n m_{W_n^x}$.  Then we will use this to estimate $m_{f^nW_n^x}(H_{\JJ,q}(f^nW_n^x))$ using \eqref{eqn:mWnx}.

Given $y\in W_n^x$, let
\begin{equation}\label{eqn:phiny}
\phi_n(y) = \det (Df^n)(y)|_{T_y W}
\end{equation}
and define $\rho_n^x\in C(f^nW_n^x,\RR^+)$ by
\begin{equation}\label{eqn:rhonx}
\rho_n^x(z) = \frac{\phi_n(x)}{\phi_n(f^{-n}(z))}.
\end{equation}
%Writing $J_n^x = \det (Df^n)(x)|_{T_x W}$, 
It follows immediately that
\begin{equation}\label{eqn:densities}
%\frac{d \left( \phi_n(x)  f_*^n m_{W}\right)}{d m_{f^n(W)}} =
\frac{d (f_*^n m_{W_n^x})}{d m_{f^nW_n^x}} =
\frac{\rho_n^x}{\phi_n(x)}.
\end{equation}
We will show that $\rho_n^x$ is a well-behaved function.

\begin{lemma}\label{lem:bdd-distortion}
There exists $L>0$ such that if $x\in S_{\JJ,q,n}$ and $W_n^x$ is as above, then $\rho_n^x(z) \in [\frac 1L,L]$ for all $z\in f^n(W_n^x)$, and $|\rho_n^x|_\alpha \leq L$.
\end{lemma}
\begin{proof}
By \cite[Theorem D]{CP16}, there is $\gamma>0$ such that the manifolds $f^kW_n^x$ for $0\leq k<n$ have the property that for every $y,z\in f^kW_n^x$, the Grassmanian distance between $T_y(f^kW_n^x)$ and $T_z(f^kW_n^x)$ is smaller than some fixed constant $\gamma$ (independent of $k,n$).  In particular, because $f$ is $C^{1+\alpha}$, there exists $K>0$ such that
\[
|\det Df(y)|_{T_{y} (f^kW_n^x)} - \det Df(z)|_{T_{z} (f^kW_n^x)} | \leq K d(y,z)^{\alpha}
\]
for every $0\leq k<n$ and $y,z\in f^kW_n^x$.  Using the backwards contraction property \eqref{eqn:QQQ} of $f^nW_n^x$, we see that for every $z_1,z_2\in f^nW_n^x$, we have
\begin{multline*}
|\phi_n(f^{-n}(z_1)) - \phi_n(f^{-n}(z_2))| \leq \sum_{k=1}^n K d(f^{-k}(z_1),f^{-k}(z_2))^\alpha \\ 
\leq K\sum_{k=1}^n (e^{-\bl k} d(z_1,z_2))^\alpha 
= Ke^{-\bl\alpha}(1-e^{-\bl\alpha})^{-1} d(z_1,z_2)^\alpha.
\end{multline*}
%where the second inequality uses the fact from the previous section that~\eqref{eqn:backcontr} holds.  
Write $K' = Ke^{-\bl\alpha}(1-e^{-\bl\alpha})^{-1}$, so that
\begin{equation}\label{eqn:phinHold}
|\phi_n(f^{-n}(z_1)) - \phi_n(f^{-n}(z_2))| \leq K'd(z_1,z_2)^\alpha.
\end{equation}
Applying this with $z_1=z$ and $z_2=f^n(x)$ yields
\[
\left| \frac{\phi_n(f^{-n}(z))}{\phi_n(x)} - 1 \right| \leq K'\bar r^\alpha e^{-\bl d_u n}
\]
for every $z\in f^n(W)$, where we use the fact that $\phi_n(x) \geq e^{\bl d_u n}$.  Writing $K'' = K' \bar r^\alpha + 1$, we see that
\[
\frac 1{K''} \leq \frac{\phi_n(f^{-n}(z))}{\phi_n(x)} \leq K'',
\]
which proves the first inequality for $\rho_n^x$.  (Note that $K''$ depends only on $K$, $\bar r$, $\bar \chi$, and $\alpha$.)

To show that the functions $\rho_n^x$ are uniformly H\"older continuous, we fix $z_1,z_2 \in f^nW_n^x$, write $y_i = f^{-n}(z_i)$, and observe that
\begin{multline*}
|\rho_n^x(z_1) - \rho_n^x(z_2)| =
\left| \frac{\phi_n(x)}{\phi_n(y_1)} - \frac{\phi_n(x)}{\phi_n(y_2)} \right| \\
\leq \frac {\phi_n(x) |\phi_n(y_1) - \phi_n(y_2)|}{\phi_n(y_1) \phi_n(y_2)} 
\leq K' K''d(z_1,z_2)^\alpha e^{-\bl d_u n}.
\end{multline*}
This completes the proof of Lemma \ref{lem:bdd-distortion}.
\end{proof}

To conclude the proof of Lemma \ref{lem:eht2}, we use Lemma \ref{lem:bdd-distortion} to estimate $m_{f^nW_n^x}(H_{\JJ,q}(f^nW_n^x))$ and get $f^nW_n^x \in \RRR_{\KK,q}$.
Fix $\beta'>0$ and let $0\leq q\leq n$ and $x\in S_{\JJ,q,n}$ be such that
\begin{equation}\label{eqn:many-Sn}
m_W (W_n^x \cap S_{\JJ,q,n}) \geq \beta' m_W(W_n^x).
\end{equation}
By the definition of $\Gamma^s_{\JJ,q} \supset \Gamma_q^\JJ$ in \eqref{eqn:shtimes}, we have
$f^n(W_n^x \cap S_{\JJ,q,n}) \subset H_{\JJ,q}(f^nW_n^x)$.
Then Lemma \ref{lem:bdd-distortion} gives
\begin{multline*}
m_{f^n(W)} (f^n(W_n^x \cap S_{\JJ,q,n})) = \int_{W_n^x \cap S_{\JJ,q,n}} \frac{\phi_n(x)}{\rho_n^x(z)} \,dm_W(z) \\
\geq \tfrac 1L \phi_n(x)m_W(W_n^x \cap S_{\JJ,q,n})) \geq \phi_n(x) \beta' L^{-1} m_W(W_n^x),
\end{multline*}
and similarly,
\[
m_{f^n(W)}(f^nW_n^x) \leq L\phi_n(x) m_W(W_n^x),
\]
so we conclude that
\[
m_{f^n(W)} (H_{\JJ,q}(f^nW_n^x)) \geq \beta' L^{-2}m_{f^n(W)}(f^n(W_x^n)).
\]
In particular, taking $\beta=\beta'L^{-2}$ gives $f^nW_n^x\in \RRR_{\KK,q}$.  Moreover, Lemma \ref{lem:bdd-distortion} shows that $(f^nW_n^x,\rho_n^x)\in \RRR_{\KK,q}'$, and since
\[
f_*^nm_{W_n^x}(E) = \frac 1{\phi_n(x)} \int_{f^n(E)} \rho_n^x(z)\,dm_{f^nW_n^x}(z),
\]
this shows that $f_*^nm_{W_n^x}\in \MMM_{\KK,q}$, as desired.

\subsection{Proof of Lemma \ref{lem:besicovitch}}\label{sec:besicovitch}

Fix $d_u$ and $\II = (\bar\theta,\bar\gamma,\bar\kappa,\bar{r})$.  We must produce $p\in \NN$ such that for every $n\in \NN$ and $0\leq q\leq n$, the set $S_{\JJ,q,n}$ admits $p$ subsets $A_1,\dots, A_p$ such that the sets $\{W_n^x \mid x\in A_i, 1\leq i\leq p\}$ cover $S_{\JJ,q,n}$, and are disjoint within each fixed value of $i$.

\begin{remark}
The sets $W_n^x$  are not necessarily close to being balls in the metric $d_W$, but their images $f^n W_n^x$ are almost balls in the metric $d_{f^nW}$.  Thus we could obtain Lemma \ref{lem:besicovitch} from the Besicovitch covering lemma in \cite[2.8.14]{hF69} if we could prove that $\bigcup_{x\in S_{\JJ,q,n}} f^n W_n^x \subset f^nW$ is \emph{directionally limited}.  It is shown in \cite[2.8.9]{hF69} that $C^2$ Riemannian manifolds are directionally limited, but we only know that $f^n W$ is $C^{1+\alpha}$.  Thus we give a direct proof of Lemma \ref{lem:besicovitch} taking advantage of the fact that the sets $f^n W_n^x$ have uniformly controlled radii.  (In the general Besicovitch covering lemma the radii are allowed to vary.)
\end{remark}

%Recall from \cite[2.8.9]{hF69} that a metric space $(X,d)$ is \emph{directionally ($\xi,\eta,\zeta$)-limited at $A\subset X$}, where $\xi>0$, $\eta \in (0,\frac 13]$, and $\zeta\in \NN$, if $\# B \leq \zeta$ for any $a\in A$ and $B\subset A \cap B(a,\xi) \setminus \{a\}$ that is ``$\eta$-separated in angle'' in the following sense: for every $b,c\in B$ and $x\in X$ with $d(a,x)=d(a,c)$ and $d(a,b) = d(a,x) + d(x,b)$, we have $d(x,c)/d(x,a) \geq \eta$.

%In \cite[2.8.14]{hF69} it is shown that if $(X,d)$ is directionally $(\xi,\eta,\zeta)$-limited at $A$ and $F$ is a family of closed balls with radii less than $\xi/2$ such that every point of $A$ is the centre of some member of $F$, then $A$ is contained in the union of $2\zeta + 1$ disjointed subfamilies of $F$.  It is also shown that compact $C^2$ Riemannian manifolds are directionally limited.  We mimic the proof of the latter result, applied to $f^n S_{\JJ,q,n} \subset f^n W$, to obtain the following lemma.

%We want to apply this to the submanifold $f^n W$, with $E=S_{\JJ,q,n}$ and $\AAA = \{f^nW_n^x \mid x\in S_{\JJ,q,n}\}$.

Given $n\in \NN$ and $x\in S_{\JJ,q,n}$, we will write $\hat x = f^n x$, $\hat W_n^x = f^n W_n^x$, etc., in order to simplify notation.  The uniform expansion of $f^{n-k}\colon f^kW_n^x\to \hat W_n^x$ guaranteed by Lemma \ref{lem:real-eht} shows that $\hat W_n^x$ is the graph of a function $\psi \colon B_{E^u(\hat x)}(\bar{r})\to E^s(\hat x)$; moreover, we have $\|D\psi\|\leq \bar\gamma$ and $|D\psi|_\alpha\leq \bar\kappa$, and the angle between $E^u(\hat x)$ and $E^s(\hat x)$ is at least $\bar\theta$.  Thus the map $\Psi\colon E^u(\hat x)\to T_{\hat x}M$ given by $\Psi(v) = v+\psi(v)$ is Lipschitz with a constant that depends only on $\II$. %$\bar\gamma,\bar\kappa,\bar\theta,\bar r$.

We recall some terminology and notation from \S\ref{sec:reduction}.  Given $a\in M$, let $\GGG_a^{d_u}$ be the Grassmanian collection of $d_u$-dimensional subspaces of $T_a M$, with metric $\rho$ given by angle.  Given $z,x\in M$ nearby, we write $\pi_{z,x} = \exp_x^{-1} \circ \exp_z \colon B_z \to T_x M$, observing that $\pi_{z,x} = I_{z,x} + g_{z,x}$, where $I_{z,x}$ is an isometry and $g_{z,x}$ is smooth with $Dg_{z,x}(0)=0$.

From the observations in the first paragraph above,  there is $\eps>0$ such that for every $x\in S_{\JJ,q,n}$ and $z,z'\in \exp_{\hat x}^{-1} \hat W_n^x \subset T_{\hat x} M$, we have
\begin{equation}\label{eqn:rho}
\rho(T_z \exp_{\hat x}^{-1} \hat W_n^x, T_{z'} \exp_{\hat x}^{-1} \hat W_n^x) < \eps.
\end{equation}
(We commit a slight abuse of notation by conflating $T_{\hat x} M$ with $T_z T_{\hat x} M$ for each $z\in T_{\hat x} M$ so that we can compare the angles.)
Let $E_1(\hat x) = T_{\hat x}\hat W_n^x$ and $E_2(\hat x) = E_1(\hat x)^\perp$, and let $P\colon T_{\hat x} M \to E_1(\hat x)$ be orthogonal projection along $E_2(\hat x)$.  Then $P(\exp_{\hat x}^{-1} \hat W_n^x) \subset B_{E_1(\hat x)}(0,r')$, where $r'>r$ depends only on $\II$.

We conclude that if $d_{\hat W}$ denotes distance on $\hat W$ and $B_{\hat W}(y,r)$ denotes the $d_{\hat W}$-ball of radius $r$ centred at $y$, then there are $0<r<\bar{r}<r'$ such that
\begin{equation}\label{eqn:B_W}
B_{\hat W}(\hat x,r) \subset \hat W_n^x \subset B_{\hat W}(\hat x,r') \text{ for all } x\in S_{\JJ,q,n}.
\end{equation}
Let $A\subset S_{\JJ,q,n}$ be such that $\hat A = f^n A$ is a maximal $r$-separated subset of $\hat S_{\JJ,q,n} := f^n S_{\JJ,q,n}$.  Then we have
\begin{equation}\label{eqn:fnSnin}
\hat S_{\JJ,q,n} \subset \bigcup_{x\in A} B_{\hat W}(\hat x,r) \subset \bigcup_{x\in A} \hat  W_n^x,
\end{equation}
and in particular, $S_{\JJ,q,n} \subset \bigcup_{x\in A} W_n^x$.  Let $G$ be the graph whose vertex set is $A$, with an edge between $x,y\in A$ if and only if $W_n^x \cap W_n^y \neq \emptyset$.  Write $x\leftrightarrow y$ when this occurs.

To complete the proof of Lemma \ref{lem:besicovitch} it suffices to show that there is $p\in \NN$, depending only on $d_u$ and $\II$, such that the chromatic number of $G$ is $\leq p$.  It suffices to show that every vertex of $G$ has degree $\leq p$.

Let $W_n = \bigcup_{x\in S_{\JJ,q,n}} W_n^x$.  Fix $x\in A$, and consider the set
\[
\hat V_n^x := \bigcup\{ B_{\hat W_n} (\hat y, r') \mid y\in A,
%S_{\JJ,q,n}, 
y\leftrightarrow x\}.
\]
%together with $V_n^x = f^{-n} \hat V_n^x$.  
Note that $\hat V_n^x \subset B_{\hat W_n}(\hat x, 3r')$, and that for every $y\leftrightarrow x$ we have 
%If $y\in A$ is such that there is an edge from $x$ to $y$ in $G$, then we have $W_n^x\cap W_n^y\neq\emptyset$, and so $d_{\hat W_n}(\hat x,\hat y) \leq 2r'$ by \eqref{eqn:B_W}.  It follows that 
$B_{\hat W_n}(\hat y,\tfrac r2) \subset \hat V_n^x$.  Because $\hat A$ is $r$-separated, the sets $\{B_{\hat W_n}(\hat y,\tfrac r2) \mid y\in A\}$ are pairwise disjoint.  In particular, we have
\begin{equation}\label{eqn:degx}
\begin{aligned}
m_{\hat W_n} (\hat V_n^x) &\geq \sum \{ m_{\hat W_n}(B_{\hat W_n}(\hat y,\tfrac r2)) \mid y\in A, y\leftrightarrow x \} \\
&\geq (\deg x) \inf_{y\in S_{\JJ,q,n}} m_{\hat W_n}(B_{\hat W_n}(\hat y,\tfrac r2)).
\end{aligned}
\end{equation}

For a lower bound on $m_{\hat W_n}(B_{\hat W_n}(\hat y,\tfrac r2))$, we write $V(r,d_u)$ for the volume of the Euclidean ball of radius $r$ in dimension $d_u$, and let $Q$ be such that the exponential map $\exp_a\colon T_a M \to M$ is $Q$-Lipschitz on the ball of radius $3r'$ for every $a\in \overline{U}$.  (Here we use compactness of $\overline{U}$.)  Using \eqref{eqn:rho}, we see that for every $y\in S_{\JJ,q,n}$ we have
\begin{equation}\label{eqn:Myr2}
m_{\hat W_n}(B_{\hat W_n}(\hat y,\tfrac r2)) \geq Q^{-d_u}(1+\eps)^{-d_u} V(\tfrac r2, d_u).
\end{equation}

For the upper bound on $m_{\hat W_n} (\hat V_n^x)$, we first observe that $\exp_{\hat x}^{-1} \hat V_n^x$ is not contained in $\exp_{\hat x}^{-1} \hat W_n^x$.  However, for every $y\leftrightarrow x$ there is $u\in W_n^x \cap W_n^y$, and thus for any $z\in \exp_{\hat x}^{-1} \hat W_n^y$ we have
\begin{align*}
\rho(T_z \exp_{\hat x}^{-1} \hat V_n^x, E^1(\hat x))
&\leq \rho(T_z \exp_{\hat x}^{-1} \hat W_n^y, T_{\exp_{\hat x}^{-1}\hat u} \exp_{\hat x}^{-1} \hat W_n^y) \\
&\qquad\qquad + \rho(T_{\exp_{\hat x}^{-1}\hat u} \exp_{\hat x}^{-1} \hat W_n^y, E^1(\hat x)) \\
&\leq \|g_{\hat y,\hat x}\|_{C^2} \rho(T_{\pi_{y,x}^{-1} z} \exp_{\hat y}^{-1} \hat W_n^y, T_{\exp_{\hat y}^{-1}\hat u} \exp_{\hat y}^{-1} \hat W_n^y) + \eps.
\end{align*}
There is $C>0$ such that for every $x,y \in S_{\JJ,q,n}$ with $d_{\hat W}(\hat x,\hat y) < 2r'$ we have $\|g_{\hat y,\hat x}\|_{C^2} \leq C$, and so for every $z,z' \in \exp_{\hat x}^{-1} \hat V_n^x \subset T_{\hat x}M$, we have
\begin{equation}\label{eqn:rhoVnx}
\rho(T_z \exp_{\hat x}^{-1} \hat V_n^x, T_{z'} \exp_{\hat x}^{-1} \hat V_n^x) < 2\eps(1+C).
\end{equation}
Then as in \eqref{eqn:Myr2} we have
\begin{equation}\label{eqn:Myr3}
m_{\hat W_n}(\hat V_x^n) \leq Q^{d_u} (1+2\eps(1+C))^{d_u} V(d_u,3r').
\end{equation}
Combining \eqref{eqn:degx}, \eqref{eqn:Myr2}, and \eqref{eqn:Myr3}, we see that there is $p\in \NN$, depending only on $d_u,r,r',\eps,Q,C$, such that $\deg x \leq p$ for every $n\in \NN$ and $x\in S_{\JJ,q,n}$.  This completes the proof of Lemma \ref{lem:besicovitch}.

\subsection{Proof of Proposition \ref{prop:unif-large}}\label{sec:pf-unif-large}

First note that Lemma \ref{lem:real-eht} gives $\delta>0$ such that for each $q\in \NN$ there is $N(q)$ with the property that for every $n\geq N(q)$ we have
\[
\int_W \frac 1n \#([0,n)
\cap \Gamma_q^\JJ(x)) \,dm_W(x) > \delta,
\]
or equivalently,
\begin{equation}\label{eqn:average-delta}
\frac 1n \sum_{k=1}^n m_W(S_{\JJ,q,k}) > \delta.
\end{equation}
Let $p\in \NN$ be as in Lemma  \ref{lem:besicovitch} and fix $\beta' > 0$ such that
\begin{equation}\label{eqn:beta''}
\beta'' := p^{-1} \delta - \beta' m_W(W) > 0.
\end{equation}
Given $0\leq q\leq n$, decompose $S_{\JJ,q,n}$ into the following two sets:
\begin{align*}
S_{\JJ,q,n}^g &= \{x\in S_{\JJ,q,n} \mid m_W(W_n^x \cap S_{\JJ,q,n}) \geq \beta' m_W(W_n^x) \}, \\
S_{\JJ,q,n}^b &= \{x\in S_{\JJ,q,n} \mid m_W(W_n^x \cap S_{\JJ,q,n}) < \beta' m_W(W_n^x) \},
\end{align*}
where $W_n^x$ is as in Lemma \ref{lem:real-eht}.  Note that $S_{\JJ,q,n}^g$ is exactly the set of points to which Lemma \ref{lem:eht2} applies (for the given value of $\beta'$).
Now let $A_1,\dots, A_p \subset S_{\JJ,q,n}$ be given by Lemma \ref{lem:besicovitch}.  We have
\begin{align*}
m_W&(S_{\JJ,q,n}) \leq \sum_{i=1}^p \sum_{x\in A_i} m_W(S_{\JJ,q,n} \cap W_n^x) \\
&= \sum_{i=1}^p \bigg(\sum_{x\in A_i \cap S_{\JJ,q,n}^b} m_W(S_{\JJ,q,n} \cap W_n^x) +
%\sum_{i=1}^p 
\sum_{x\in A_i \cap S_{\JJ,q,n}^g} m_W(S_{\JJ,q,n} \cap W_n^x)\bigg) \\
&\leq p\beta' m_W(W) + \sum_{i=1}^p \sum_{x\in A_i \cap S_{\JJ,q,n}^g} m_W(W_n^x),
\end{align*}
where the last inequality follows from the definition of $S_{\JJ,q,n}^b$, which gives $m_W(S_{\JJ,q,n} \cap W_n^x) \leq \beta' m_W(W_n^x)$ for every $x\in A_i \cap S_{\JJ,q,n}^b$, together with the fact that for each value of $i$, the set $\{W_n^x \mid x\in A_i \cap S_{\JJ,q,n}^b\}$ are disjoint.  We conclude that there exists $i$ such that
\begin{equation}\label{eqn:1pbeta'}
\sum_{x\in A_i\cap S_{\JJ,q,n}^g} m_W(W_n^x) \geq \tfrac 1p m_W(S_{\JJ,q,n}) - \beta' m_W(W).
%\big(\tfrac 1p - \beta'\big) m_W(S_{\JJ,q,n}).
\end{equation}
(The right-hand side of \eqref{eqn:1pbeta'} may be negative for some values of $n$, but we will see that it is positive on average.)
Write %$\beta'' = \frac 1p - \beta'>0$ and 
$S_{\JJ,q,n}' = S_{\JJ,q,n}^g \cap A_i$ for this choice of $i$.  Then given any $x\neq y\in S_{\JJ,q,n}'$, we have $W_n^x \cap W_n^y=\emptyset$, which is the disjointness condition we wanted.  Writing $W_{\JJ,q,n} = \bigsqcup_{x\in S_{\JJ,q,n}'} W_n^x$, Lemma \ref{lem:eht2} gives
\begin{equation}\label{eqn:in-cpt}
f_*^n m_{W_{\JJ,q,n}} \in \MMM_{\KK,q}
\end{equation}
whenever $n\geq q$.  Moreover, \eqref{eqn:average-delta}, \eqref{eqn:beta''}, and \eqref{eqn:1pbeta'} give
\begin{equation}\label{eqn:Wn}
\frac 1n \sum_{k=1}^n m_W(W_{\JJ,q,k})
\geq \frac 1p\delta - \beta' m_W(W) = \beta'' > 0
%m_W(W_{\JJ,q,n}) \geq \beta'' m_W(S_{\JJ,q,n}).
\end{equation}
whenever $n \geq N(q)$.
To complete the proof of Proposition \ref{prop:unif-large}, we find $q_n\to\infty$ such that the measures $
\mu_n = \frac 1n \sum_{k=0}^{n-1} f_*^k (m_{W})$ have uniformly large projection to $\MMM_{\KK,q_n}$; that is, we need $\nu_n \in \MMM_{\KK,q_n}$ such that $\nu_n \leq \mu_n$ and $\llim_n \|\nu_n\| > 0$.
%which together with \eqref{eqn:Wn} gives
%\begin{equation}\label{eqn:Wk2}
%\frac 1n \sum_{k=1}^n m_W(W_{\JJ,q,k}) \geq \beta'' \delta \text{ for all } n\geq N(q).
%\end{equation}
Choose $q_n \to \infty$ such that
\[
n\geq N(q_n) \text{ and } q_n \leq \tfrac 12 \beta'' m_W(W) n.
\]
It follows from \eqref{eqn:in-cpt} that for every $k\geq q_n$ we have
\[
f_*^k m_W \geq f_*^k m_{W_{\JJ,q,k}} \in \MMM_{\KK,q_n}.
\]
Averaging over $k$ from $q_n$ to $n$ gives
\[
\mu_n \geq \nu_n := \frac 1n \sum_{k=q_n}^{n-1} f_*^k m_{W_{\JJ,q,k}} \in \MMM_{\KK,q_n}
\]
and so it only remains to estimate $\|\nu_n\|$.  Using \eqref{eqn:Wn} and our choice of $q_n$, we see that
\begin{align*}
\|\nu_n\| &= \frac 1n\sum_{k=q_n}^{n-1} \|m_{W_{\JJ,q,k}}\| 
= \frac 1n \bigg(\sum_{k=0}^{n-1} m_W(W_{\JJ,q,k})\bigg) - \frac 1n \bigg(\sum_{k=0}^{q_n-1} m_W(W_{\JJ,q,k}) \bigg) \\
&\geq \beta'' - \tfrac{q_n}n m_W(W) \geq \tfrac 12 \beta''.
\end{align*}

%\section{Proof of Theorem~\ref{thm:buildSRB}}\label{sec:pfthma}

\subsection{Proof of Proposition \ref{prop:MacSRB}}\label{sec:MacSRB}

%Now we move to the setting of Theorem~\ref{thm:buildSRB}, so $M$ is a compact Riemannian manifold, $U\subset M$ an open set, and $f\colon U\to M$ a $C^{1+\alpha}$ local diffeomorphism such that $\overline{f(U)} \subset U$.  Let $\Lambda = \bigcap_{n\geq 1} f^n(U)$ be the attractor for $f$.  Let $\Mac$ and $\Mh$ be as defined in Section \ref{sec:general}.  Note that the set $\NNN = \NNN(\RRR)$ defined in \eqref{eqn:NR} is a candidate collection of null sets, and \eqref{eqn:Mac} gives $\Mac = \SSS(\NNN)$, so Lemma \ref{lem:decomposition} shows that $\MMM(M) = \Mac \oplus (\Mac)^\perp$.

We prove that $\Mac\cap \Mh$ is precisely the set of SRB measures for $f$.

First we show that every SRB measure $\mu$ is in $\Mac\cap\Mh$.  Every SRB measure is hyperbolic, so $\mu\in\Mh$.  To show that $\mu\in \Mac$, first observe that $\mu$ is invariant and supported on $\Lambda$.  As discussed before Definition \ref{def:SRB}, $\mu$ can be expressed in terms of conditional measures on local unstable manifolds.  More precisely, if we write $\tilde\RRR$ for the set of all local unstable manifolds (so that in particular $\tilde\RRR \subset \RRR$), then for each $\ell$ one can take a measurable partition of the regular set $Y_\ell$ into sets of the form $W\cap Y_\ell$, where $W\in \tilde\RRR$, and let $\{ \mu_W \mid W\in\tilde\RRR\}$ be the conditional measures of $\mu$ relative to each element of this partition.  This means that there is a measure $\eta$ on $\tilde\RRR$ such that 
\begin{equation}\label{eqn:conditional}
\mu(E) = \int_{\tilde\RRR} \mu_W(E) \,d\eta(W)
\end{equation}
for every measurable set $E$.  By Definition \ref{def:SRB} we have $\mu_W \ll m_W$ for  $\eta$-a.e.\ $W$, and since every local unstable manifold is contained in $\RRR$, this shows that $\mu(E)=0$ for all $E\in \NNN(\RRR)$, so $\mu\in \Mac$.

Conversely, if $\mu\in \Mac\cap\Mh$, we show that $\mu$ is an SRB measure.  Since $\mu$ is invariant, it is supported on $\Lambda$. Because hyperbolicity is given, this amounts to showing that the  conditional measures generated by $\mu$ on local unstable manifolds are absolutely continuous with respect to the leaf volume. Once again using the decomposition of $\mu$ in \eqref{eqn:conditional}, we show that $\mu_W\ll m_W$ for $\eta$-a.e.\ $W$.  Indeed, if there is a positive $\eta$-measure set of $W$ such that $\mu_W \not\ll m_W$, then we may write $\WWW$ for this set and take for each $W\in \WWW$ a set $E_W\subset W$ with $m_W(E_W)=0$ and $\mu_W(E_W)>0$.  Taking $E$ to be the union of these $E_W$ yields a set with $E \subset \bigcup_{W\in \WWW} W$ and  $m_W(E)=0$ for every $W\in \WWW$, so $E\in \NNN(\RRR)$, and moreover $\mu(E) >0$ since $\mu_W(E_W)>0$, which contradicts the assumption that $\mu\in\Mac$.

\subsection{Proof of Proposition \ref{prop:decomposition}}\label{sec:decomposition}

We deduce
Proposition \ref{prop:decomposition} from a  generalisation of the Lebesgue decomposition theorem, which follows the proof given in~\cite{jB71}.

Let $(X,\Omega)$ be a measurable space, and let $\MMM$ denote the collection of all finite measures on $X$.  We say that a collection of subsets $\NNN \subset \Omega$ is a \emph{candidate collection of null sets} if it is closed under passing to subsets and countable unions:
\begin{enumerate}
\item if $E\in \NNN$ and $F\in \Omega$, $F\subset E$, then $F\in \NNN$;
\item if $\{E_n\} \subset \NNN$ is a countable collection, then $\bigcup_n E_n \in \NNN$.
\end{enumerate}
Given a candidate collection of null sets, let $\SSS = \SSS(\NNN)$ be the subspace of $\MMM$ defined by
\begin{equation}\label{eqn:SN}
\SSS = \{ \mu\in \MMM \mid \mu(E) = 0 \text{ for all } E\in \NNN \}.
\end{equation}
Given a subspace $\SSS \subset \MMM$, define the space of singular measures by
\[
\SSS^\perp = \{ \nu\in\MMM \mid \nu\perp \mu \text{ for all } \mu\in \SSS \}.
\]
If $\SSS$ is given by~\eqref{eqn:SN}, then the subspaces $\SSS$ and $\SSS^\perp$ give a decomposition of $\MMM$.

\begin{lemma}\label{lem:decomposition}
Let $\NNN$ be a candidate collection of null sets, and let $\SSS = \SSS(\NNN)$ be given by~\eqref{eqn:SN}.  Then $\MMM = \SSS \oplus \SSS^\perp$.
\end{lemma}
\begin{proof}
Fix $\nu\in \MMM$; we need to show that there is a unique decomposition $\nu = \nu_1 + \nu_2$, where $\nu_1 \in \SSS$ and $\nu_2 \in \SSS^\perp$.  To this end, consider the following collection of subsets:
\[
\NNN' = \{ E \in \NNN \mid \nu(E) > 0 \}.
\]
Let $\theta = \sup \{\nu(E) \mid E \in \NNN' \}$, and let $E_n\in \NNN'$ be a countable collection such that $\nu(E_n) \to \theta$.  Consider the union $A = \bigcup_n E_n$, and observe that $\nu(A) = \theta$ and $A\in \NNN'$.

We claim that $\nu_1 = \nu|_{X\setminus A}$ and $\nu_2 = \nu|_A$ gives the desired decomposition.  Indeed, $\nu_2\perp \mu$ for all $\mu\in \SSS$ since $\mu(A)=0$ and $\nu_2(X\setminus A) = 0$, so $\nu_2 \in \SSS^\perp$.  Furthermore, given $E\in \NNN$, we may write $E' = E \setminus A$ and $F = E \cap A$; then $\nu_1(F) = 0$ by definition, and if $\nu(E') = \nu_1(E')>0$, we would have $\nu(E' \cup A) = \nu(E') + \nu(A) > \theta$, contradicting the definition of $\theta$.  It follows that $\nu_1(E) = 0$, and since this holds for all $E\in \NNN$, we have $\nu_1 \in \SSS$.

Finally, uniqueness of the decomposition follows from the fact that $\SSS \cap \SSS^\perp = \{0\}$.
\end{proof}

Proposition \ref{prop:decomposition} follows
upon observing that 
the set $\NNN = \NNN(\RRR)$ defined in \eqref{eqn:NR} is a candidate collection of null sets, and \eqref{eqn:Mac} gives $\Mac = \SSS(\NNN)$, so Lemma \ref{lem:decomposition} shows that $\MMM(M) = \Mac \oplus (\Mac)^\perp$.

\section{Proof of Theorem \ref{thm:perturbed}}\label{sec:general-conditions}

\subsection{``Good'' iterates}

Recall that $\tilde\AAA$ denotes the set of $C^{1+\alpha}$ curves $W\subset U\setminus Z$ such that
\begin{itemize}
\item $T_x W \subset K^u(x)$ for all $x\in W$;
\item %the unit tangent vector in $T_x W$ can be chosen to be an $\alpha$-H\"older continuous function of $x$, with constant $L$;
$W$ has H\"older curvature bounded by $L$;
\item the length of $W$ is between $\eps$ and $2\eps$.
\end{itemize}
The first observation we need is that by uniform hyperbolicity on $U\setminus Z$, \ref{C2} can be extended to all $W\in \tilde\AAA$, not just those that enter $Z$.  In fact, it can be strengthened slightly.  
Recall also that $\lambda(x) =\min(\lambda^u(x) - \Delta(x), \lambda^s(x))$ as in \eqref{eqn:le}.

\begin{lemma}\label{lem:local-growth}
There are constants $L,\eps,Q,\nu>0$ and a function $p\colon \NN\to [0,1]$ such that $\sum_{t\geq 1} tp(t)<\infty$ and every $W\in \tilde\AAA$ has an admissible decomposition $\{W_j, \tau_j\}$ satisfying
\begin{enumerate}[label=(\arabic{*})]
\item\label{t-integrable}
$m_W(\{x\in W \mid \tau(x) = t\}) \leq p(t)m_W(W)$ for every $t\in \NN$;
\item\label{t-largereturns}
$\bar G(W_j)\in \tilde\AAA$ for every $j$;
\item\label{t-distortion}
if $x,y\in W_j$ then $\log\frac{|D\bar G(x)|_{T_xW}|}{|D\bar G(y)|_{T_yW}|} \leq Qd(\bar Gx,\bar Gy)^{\alpha^2}$;
\item\label{t-growth}
$\sum_{j=k}^{\tau(x)} (\lambda^u-\Delta)(g^j(x)) \geq \nu$ and $\sum_{j=k}^{\tau(x)} \lambda^s(g^j(x)) \leq -\nu$ for all $x\in W$ and $0\leq k\leq \tau(x)$.
\end{enumerate}
\end{lemma}
\begin{proof}
Recall that an admissible decomposition of $W\in \tilde{A}$ is a partition (mod zero with respect to $m_W$) of $W$ into subsets $W_j$, together with an assignment of an `inducing time' $\tau_j$ to each $W_j$ such that $g^{\tau_j}(W_j) \subset U\setminus Z$; thus we must define both the partition $W_j$ and the inducing function $\tau$ (which is constant on each $W_j$).
If $g(W)$ does not intersect $Z$ then it suffices to take $\tau\equiv 1$ and partition $W$ into either 1 or 2 pieces, depending on whether $\bar G(W) = g(W)$ has length greater or less than $2\eps$.

If $g(W)$ does intersect $Z$, then start with the decomposition $W = \bigsqcup_j W_j$ from \ref{C2}.  Let $\hat\tau\colon W\to \NN$ be the inducing function given there, and $\hat p\colon \NN\to [0,2\eps)$ be the probability envelope.  Let $n$ be the minimum time it takes for an $f$-orbit leaving $Z$ to return to $Z$.  Note that by \ref{C3} we have $n>C/\log \chi$.  Let $\tau = \hat \tau + n$ and let $p(t)= \hat p(t-n)$.  Convergence of $\sum tp(t)$ follows from convergence of $\sum t \hat p(t)$.

To get the desired decomposition, note that for all $j$ we have $\hat G(W_j) = g^n(g^{\hat\tau_j}(W_j))\subset U\setminus Z$.  By invariance of $K^u$ we see that all the tangent vectors to $\hat G(W_j)$ lie in $K^u$, and thus $\hat G(W_j)$ can be decomposed into a disjoint union $\bigsqcup_\ell \hat G(W_{j,\ell})$, where each $\hat G(W_{j,\ell})$ is in $\tilde\AAA$.  Thus $W = \bigsqcup_{j,\ell} W_{j,\ell}$ is the desired decomposition and \ref{t-largereturns} is verified.

Condition \ref{t-growth} follows from \ref{C3}, putting $\nu = n\log \chi - C$.  So it only remains to prove the bounded distortion condition \ref{t-distortion}.

Recall that the H\"older curvature of every $W\in \AAA$ is bounded by $L$.  Thus given $W\in \AAA$ and two nearby points $x,y\in W$, the Grassmanian distance between $T_x W$ and $T_y W$ is bounded above by $L' d(x,y)^\alpha$.  Because $f$ is $C^{1+\alpha}$ and $M$ is compact, this gives
\begin{equation}\label{eqn:Dgxy}
|Dg(x)|_{T_xW} - Dg(y)|_{T_yW}| \leq K d(x,y)^{\alpha^2}
\end{equation}
for some uniform constant $K$.  As long as $W$ lies outside of $Z$, we can use uniform expansion together with the observation that $\log a - \log b \leq (a-b) \frac 1b$ whenever $a>b$ to get
\begin{equation}\label{eqn:logDgxy}
\log \frac{|Dg(x)|_{T_xW}|}{|Dg(y)|_{T_yW}|} \leq K d(x,y)^{\alpha^2}.
\end{equation}
We also observe that there is $\hat\chi<1$ such that for every $t\in \NN$, every $x,y\in W_j\subset W(t)$, and every $k\geq 0$ such that $g^{t-k}(x) \in U\setminus Z$, we have
\begin{equation}\label{eqn:dfkx}
d(g^{t-k}(x), g^{t-k}(y)) \leq \hat\chi^k d(g^t(x),g^t(y)) = \hat\chi^k d(\bar G(x), \bar G(y)).
\end{equation}
For convenience of notation, given $j\in \NN$ we write
\[
D_j g(x) = Dg(g^j(x))|_{T_{g^j(x)}(g^j(W))}.
\]
Then we can use \eqref{eqn:dfkx} together with \ref{C2}\ref{distortion} and \eqref{eqn:logDgxy} to get
\begin{align*}
\log\frac{|D\bar G(x)|_{T_xW}|}{|D\bar G(y)|_{T_yW}|} &\leq
\log\frac{|Dg^{\hat\tau}(x)|_{T_xW}|}{|Dg^{\hat\tau}(y)|_{T_yW}|}
+ \sum_{k=1}^n \log \frac{|D_{\hat\tau + n-k}g(x)|}{|D_{\hat\tau + n-k}g(y)|} \\
&\leq Q \hat\chi^{\alpha n} d(\bar G(x), \bar G(y))^\alpha + \sum_{k=1}^n K \hat\chi^{k\alpha^2} d(\bar G(x), \bar G (y))^{\alpha^2},
\end{align*}
which suffices to complete the proof of Lemma \ref{lem:local-growth}.
\end{proof}

Now we can iterate Lemma \ref{lem:local-growth} and follow a procedure similar to the one in \cite{BV00}.  Given $W\in \tilde\AAA$, let $W = \bigsqcup_{j_1=1}^{k_1} W(j_1)$ be the partition given by Lemma \ref{lem:local-growth}.  Then for every $j_1$, the curve $\bar G(W(j_1))$ is in $\tilde\AAA$, and so the lemma can be applied to this curve as well, giving a decomposition $W(j_1) = \bigsqcup_{j_2=1}^{k_2} W(j_1,j_2)$, where each $\bar G^2(W(j_1,j_2))$ is in $\tilde\AAA$.  (Note that $k_2$ may depend on $j_1$.)

To simplify notation we write $\jj=(j_1,\dots,j_n)$ and $|\jj|=n$.  Iterating the above procedure yields a partition $W = \bigsqcup_{\jj} W(\jj)$ such that $\bar G^n (W(\jj))\in \tilde\AAA$.  Moreover, writing $T(\jj) = \sum_{i=1}^{|\jj|} \tau_{j_i}$, Lemma \ref{lem:local-growth}\ref{t-growth} yields
\begin{equation}\label{eqn:nth-return}
\sum_{i=0}^{T(\jj)} (\lambda^u(g^ix) - \Delta(g^ix)) \geq \nu |\jj|,\qquad
\sum_{i=0}^{T(\jj)} \lambda^s(g^ix) \leq -\nu |\jj|.
\end{equation}
Given $\jj$ with $|\jj| = n$ and any $x,y\in W(\jj)$, let $\gamma$ be a path on $\bar G^n(W(\jj))$ that connects $\bar G^n(x)$ to $\bar G^n(y)$.  Then there is a path $\eta$ on $W(\jj)$ such that $\bar G^n(\eta) = \gamma$, and by \eqref{eqn:nth-return}, the lengths of $\eta$ and $\gamma$ are related by
\[
|\eta| \leq e^{-\nu n} |\gamma|.
\]
Writing $d_W$ for distance on $W$, this implies that
\[
d_{W(\jj)}(x,y) \leq e^{-\nu n} d_{\bar G^nW(\jj)}(\bar G^n(x), \bar G^n(y)).
\]
Because $W(\jj)$ and $\bar G^nW(\jj)$ both have H\"older curvature bounded by $L$, there is a constant $L'>0$ such that
\begin{equation}\label{eqn:induced-contract}
d(x,y) \leq L'e^{-\nu |\jj|} d(\bar G^nx, \bar G^ny)
\end{equation}
whenever $x,y\in W(\jj)$, where $d$ is the usual metric on $M$.  We can use this to get the  following bounded distortion control.

\begin{lemma}\label{lem:distortion}
There exists $K\in \RR$ such that given $W\in \tilde\AAA$ and $x,y \in W(\jj)$, we have
\begin{equation}\label{eqn:distortion}
K^{-1} \leq \frac{|Dg^{T(\jj)}(x)|_{T_xW}|}{|Dg^{T(\jj)}(y)|_{T_yW}|} \leq K.
\end{equation}
\end{lemma}
\begin{proof}
Given $x,y\in W(\jj)$, and $0\leq i < n$, we adopt the shorthand notation
\[
D_i \bar G(x) = D\bar G(\bar G^i (x))|_{T_{\bar G^i (x)}(\bar G^i (W))},
\]
so that
\begin{equation}\label{eqn:DiGx}
Dg^{T(\jj)}(x)|_{T_x W} = \prod_{i=0}^{n-1} D_i\bar G(x).
\end{equation}
We see from \eqref{eqn:induced-contract} that 
\begin{equation}\label{eqn:dGix}
d(\bar G^i(x),\bar G^i(y)) \leq L'e^{-\nu(n-i)} d(g^{T(\jj)}(x), g^{T(\jj)}(y)). 
\end{equation}
Now Lemma \ref{lem:local-growth}\ref{t-distortion}, together with \eqref{eqn:DiGx} and \eqref{eqn:dGix}, yields
\begin{multline*}
\log\frac{|Dg^{T(\jj)}(x)|_{T_xW}|}{|Dg^{T(\jj)}(y)|_{T_yW}|}
=
\sum_{i=0}^{n-1} \log\frac{|D_i\bar G(x)|}{|D_i\bar G(y)|} \\
\leq \sum_{i=0}^{n-1} Q d(\bar G^{n-i}(x), \bar G^{n-i} (y))^{\alpha^2} 
\leq \sum_{i=0}^{n-1} Q e^{-\nu \alpha^2 i} (L')^{\alpha^2} d(g^{T(\jj)}(x),g^{T(\jj)}(y))^{\alpha^2}.
\end{multline*}
This completes the proof since $\sum e^{-\nu\alpha^2 i}$ converges and the roles of $x,y$ are symmetric.
\end{proof}

Now given $W\in \tilde\AAA$ and $x\in W$, let $j_1(x),j_2(x),\dots$ be such that $x\in W(j_1(x),\dots,j_n(x))$ for all $n$.  Define a sequence of $\NN$-valued random variables $t_1,t_2,\dots$ on $(W,m_W)$ by $t_n(x) = \tau_{j_n(x)}$.  

\begin{proposition}\label{prop:pre-martingale}
With $p,\eps$ as in Lemma \ref{lem:local-growth} and $K$ as in Lemma \ref{lem:distortion}, we have $\PP[t_n = T \mid t_1,\dots, t_{n-1}] \leq Kp(T)$.
\end{proposition}
\begin{proof}
Observe that $t_n$ is constant on $W(\jj)$ whenever $|\jj|\geq n$, and so
\begin{equation}\label{eqn:expectation}
%\EE[t_n\mid t_1,\dots,t_{n-1}] \leq \sup_{|\jj|=n-1} \frac 1{m_W(W(\jj))}\int_{W(\jj)} t_n(x)\,dm_W(x).
\PP[t_n=T\mid t_1,\dots, t_{n-1}] \leq \sup_{|\jj|=n-1} \frac{m_W(W(\jj,T))}{m_W(W(\jj))}.
\end{equation}
By Lemma \ref{lem:distortion}, for every $\jj$, the map $g^{T(\jj)}$ carries $W(\jj)$ to $V_\jj := g^{T(\jj)}W(\jj) \in \tilde\AAA$ with distortion bounded by $K$, and so in particular for $|\jj|=n-1$ we have
\[
%\int_{W(\jj)} \frac {t_n(x)}{m_W(W(\jj))}\,dm_W(x) = \sum_{t_{j_n}\in \NN} j_n \frac{m_W(W(\jj,j_n))}{m_W(W(\jj)} \\
%\leq K \sum_{j_n\in \NN} t_{j_n} \frac{m_{V_\jj}(V_\jj(j_n))}{m_{V_\jj}(V_\jj)}
%\leq K\eps^{-1} \int_{V_\jj} t_n(x)\,dm_{V_\jj}(x)
%\leq K\eps^{-1} Q,
\frac{m_W(W(\jj,T))}{m_W(W(\jj))}
\leq
K \frac{m_{V_\jj}(V_\jj(T))}{m_{V_\jj}(V_\jj)}
\leq Kp(T),
\]
where the last inequality uses \ref{C2}.
\end{proof}

\subsection{Asymptotic averages of return times}

Now we are in a position to prove that the asymptotic average of the return times $t_n$ is bounded for $m_W$-a.e.\ initial condition.  The arguments used here are well-known, but we give full details as our setting differs from that in which the strong law of large numbers is usually proved.  We follow \cite[Theorem 22.1]{pB79}, which gives an argument that goes back to Etemadi.

Consider the probability space $(W,m_W)$, where $W\in \tilde\AAA$ is as in the previous section and we take $m_W$ to be normalized.  Let $\FFF_n$ be the increasing sequence of $\sigma$-algebras generated by the sets $W(j_1,\dots, j_n)$.  By Condition \ref{C2} and Proposition \ref{prop:pre-martingale}, we can choose $p\colon \NN \to [0,1]$ such that 
\begin{align}
\label{eqn:Ptn}
\PP[t_n=T \mid \FFF_{n-1}] &\leq p(T), \\
\label{eqn:R}
R := \sum_{T=1}^\infty Tp(T) &< \infty.
\end{align}
Note that the ``new'' $p(t)$ is obtained by multiplying the ``old'' one from Condition \ref{C2} by the distortion constant $K$ coming from Lemma \ref{lem:distortion}.

\begin{proposition}\label{prop:slln}
If $t_n$ is any sequence of random variables satisfying \eqref{eqn:Ptn} and \eqref{eqn:R}, then $\ulim_{n\to\infty} \frac 1n \sum_{k=1}^n t_k \leq R$ almost surely.
\end{proposition}
\begin{proof}
Consider the truncated random variables $s_n = t_n \one_{[t_n\leq n]}$, and note that by \eqref{eqn:Ptn} and \eqref{eqn:R}, we have
\[
0\leq s_n \leq t_n \Rightarrow \EE[s_n \mid \FFF_{n-1}] \leq R.
\]
Now consider the random variables $r_n = s_n + R - \EE[s_n \mid \FFF_{n-1}]$.  Note that $r_n$ is $\FFF_n$-measurable, and moreover
\begin{equation}\label{eqn:rnmartingale}
\EE[r_n \mid \FFF_{n-1}] = R \text{ for all } n.
\end{equation}
Let $X_N = \frac 1N (r_1 + \cdots + r_N) - R$, so that $\EE[X_N] = 0$.  The idea is to use \eqref{eqn:rnmartingale} to obtain an efficient estimate on $\EE[X_N^2]$, which via Chebyshev's inequality gives a bound on $\PP[X_N \geq \eps]$.  A careful use of Borel-Cantelli will lead to the result.

We begin with the observation that
\begin{align*}
\EE[X_N^2] &= \frac 1{N^2} \EE\left[ \left(\sum_{k=1}^N (r_k - R)\right)^2 \right] \\
&= \frac 1{N^2} \sum_{k=1}^N \EE[(r_k - R)^2] + \frac 2{N^2}\sum_{i<j} \EE[(r_i - R)(r_j - R)].
\end{align*}
Given $i<j$, the fact that $r_i$ is $\FFF_{n-1}$-measurable together with \eqref{eqn:rnmartingale} gives $\EE[(r_i - R)(r_j-R)] = 0$, hence
\begin{equation}\label{eqn:EXN}
\EE[X_N^2] = \frac 1{N^2} \sum_{k=1}^N \EE[(r_k-R)^2] 
\leq \frac 1{N^2} \sum_{k=1}^N (\EE[r_k^2] + 2R\EE[r_k] + R^2).
\end{equation}
Let $T_0$ be such that $\sum_{T\geq T_0} p(T) \leq 1$, and let $Y$ be a random variable taking the value $T$ with probability $p(T)$ for $T\geq T_0$.  Note that by the definition of $s_n$ and $r_n$, we have $r_n \leq s_n + R \leq n+R$, thus
\[
\EE[r_k^2] = \sum_{T=1}^{k+R} T^2 \PP[r_k = T] \leq \sum_{T=1}^{k+R} T^2 p(T)
\leq C + \EE[Y^2 \one_{[Y\leq k+R]}]
\]
for some fixed constant $C$.  Note that the final expression is non-decreasing in $k$, and so together with \eqref{eqn:EXN} we have
\begin{equation}\label{eqn:EXN2}
\EE[X_N^2] \leq \frac 1N \left( C' + \EE[Y^2 \one_{[Y\leq N+R]}] \right),
\end{equation}
where again $C'$ is a fixed constant.  Given $\eps>0$, we can use \eqref{eqn:EXN2} in Chebyshev's inequality to get
\begin{equation}\label{eqn:PXN}
\PP[|X_N| \geq \eps] \leq \frac 1{\eps^2 N} \left( C' + \EE[Y^2 \one_{[Y\leq N+R]}] \right).
\end{equation}
Fix $\alpha>1$ and let $u_n = \lfloor \alpha^n \rfloor$.  Putting $K = \frac{2\alpha}{\alpha-1}$, choosing $y\in \RR$, and letting $m=m(y)$ be the smallest number such that $u_m\geq y$, we have
\[
\sum_{u_n\geq y} u_n^{-1} \leq 2\sum_{n\geq m} \alpha^{-n} = K\alpha^{-m} \leq Ky^{-1},
\]
and thus \eqref{eqn:PXN} yields
\begin{align*}
\sum_{n=1}^\infty \PP[|X_{u_n}| \geq \eps] 
&\leq 
\frac 1{\eps^2} \sum_{n=1}^\infty u_n^{-1}\left( C' + \EE[Y^2 \one_{[Y\leq u_n+R]}]\right) \\
&\leq \frac{KC'}{\eps^2} + \frac 1{\eps^2} \EE \left[ \sum_{n=1}^\infty Y^2 u_n^{-1} \one_{[Y\leq u_n+R]}\right] \\
&\leq \frac{KC'}{\eps^2} + \frac 1{\eps^2}\EE[Y^2 K(Y-R)^{-1}] < \infty,
\end{align*}
where the last inequality uses \eqref{eqn:R} and the definition of $Y$.  Since the probabilities of the events $|X_{u_n}|\geq \eps$ are summable, it follows from the first Borel--Cantelli lemma that with probability 1, only finitely many of these events occur.  In particular, we have $\ulim_{n\to\infty} |X_{u_n}| \leq \eps$ almost surely; in terms of the random variables $r_n$, this means that
\[
\ulim_{n\to\infty} \left|\left(\frac 1{u_n} \sum_{k=1}^{u_n} r_k\right) - R \right| \leq \eps \text{ almost surely}.
\]
Taking an intersection over all rational $\eps>0$ gives
\begin{equation}\label{eqn:alongun}
\lim_{n\to\infty} \frac 1{u_n} \sum_{k=1}^{u_n} r_k = R \text{  a.s.}
\end{equation}
Let $Z_k = \sum_{i=1}^k r_k$.  Because $r_k\geq 0$ we have $Z_{u_n} \leq Z_k \leq Z_{u_{n+1}}$ for all $u_n\leq k\leq u_{n+1}$, and in particular
\[
\frac {u_n}{u_{n+1}} \frac{Z_{u_n}}{u_n} \leq \frac{Z_k}{k} \leq \frac{u_{n+1}}{u_n} \frac{Z_{u_{n+1}}}{u_{n+1}}.
\]
Taking the limit and using \eqref{eqn:alongun} gives
\[
\frac 1\alpha R \leq \llim_{k\to\infty} \frac 1k Z_k \leq \ulim_{k\to\infty} \frac 1k Z_k \leq \alpha R \text{ a.s.}
\]
Taking an intersection over all rational $\alpha>1$ gives
\begin{equation}\label{eqn:rnaverages}
\lim_{n\to\infty} \frac 1n \sum_{k=1}^n r_k = R \text{ a.s.}
\end{equation}
Finally, we recall from the definition of $s_n$ and $r_n$ that $t_n\leq r_n$ whenever $t_n\leq n$.  In particular, we may observe that
\[
\sum_{n=1}^\infty \PP[t_n > r_n] \leq \sum_{n=1}^\infty \PP[t_n > n] \leq \EE[t_n] \leq R
\]
and apply Borel-Cantelli again to deduce that with probability one, $t_n > r_n$ for at most finitely many values of $n$.  In particular, \eqref{eqn:rnaverages} implies that
\begin{equation}\label{eqn:tnaverages}
\ulim_{n\to\infty} \frac 1n \sum_{k=1}^n t_k \leq R \text{ a.s.}
\end{equation}
which completes the proof of Proposition \ref{prop:slln}.
\end{proof}

\subsection{Positive rate of effective hyperbolicity}

It remains to verify the conditions of Theorem \ref{thm:main3}.\footnote{The argument that follows could also be given using the simpler Theorem \ref{thm:main2} -- that is, using \ref{eqn:eh1} instead of \ref{eqn:eh3} -- but it would require a more restrictive Condition \ref{C3}.}  Recall that $\hat S$ is the set of points $x\in W$ with $T_x W \subset K^u(x)$ for which $\Gamma_\bl^e(x) \cap \Gamma_{\JJ,q}^s(x)$ has uniformly (in $q$) positive lower asymptotic density (for some $\JJ = (C,\bl)$), and for which \ref{eqn:eh2} holds.  The sets $\Gamma_\bl^e(x)$ and $\Gamma_{\JJ,q}^s(x)$ are defined in \eqref{eqn:ehtimes} and \eqref{eqn:shtimes}, respectively, and in our setting can be controlled using the sequence $t_n$.

\begin{lemma}\label{lem:Chyptimes}
Let $\nu>0$ be as in Lemma \ref{lem:local-growth}\ref{t-growth} and fix $0<\bl <\nu'<\nu/R$.  If $\ulim_{n\to\infty} \frac 1n \sum_{k=1}^n t_k(x)\leq R$ then we have 
$\ld(\Gamma_\bl^e(x) \cap \Gamma_{\JJ,q}^s(x)) \geq \frac{\nu' - \bl}{L - \bl}$ for every $q$, where $L$ is a uniform bound for $\lambda^u - \Delta$ and $\lambda^s$.
\end{lemma}
\begin{proof}
Let $T_n(x) = \sum_{k=1}^n t_k(x)$.  By Lemma \ref{lem:local-growth}\ref{t-growth} and \ref{C3}, there are sequences $a_k,b_k\in [-L,L]$ such that the following hold:
\begin{enumerate}
\item $(\lambda^u - \Delta)(g^kx) \geq a_k$;
\item $\lambda^s(g^kx)\leq b_k$;
\item $\sum_{k=T_n(x)}^{T_{n+1}(x)-1} a_k = \nu$;
\item $\sum_{k=T_n(x)}^{T_{n+1}(x)-1} b_k = -\nu$;
\item $\sum_{k=r}^s a_k \geq -C$ for all $T_n(x) \leq r \leq s < T_{n+1}(x)$;
\item $\sum_{k=r}^s b_k \leq -C$ for all $T_n(x) \leq r \leq s < T_{n+1}(x)$.
\end{enumerate}
Now if $\ulim \frac 1n T_n(x)\leq R$, then the above imply $\llim \frac 1k \sum_{j=0}^k a_j \geq \nu/R > \nu'$.  Thus by Pliss' lemma \cite[Lemma 11.2.6]{BP07} there is $\Gamma_a\subset \NN$ with $\ld(\Gamma_a) \geq \frac{\nu' - \bl}{L-\bl}$ such that for all $n\in \Gamma_a$ and $0\leq k<n$, we have
\[
\sum_{j=k}^{n-1} (\lambda^u - \Delta)(g^jx) \geq \sum_{j=k}^{n-1} a_j \geq \bl (n-k).
\]
It follows that $\Gamma_a \subset \Gamma_\bl^e(x)$.  Moreover, the properties listed above guarantee that every $n\in \Gamma_a$ is also in $\Gamma_{\JJ,q}^s(x)$ for $q\leq n$.  This yields the desired lower bound on the asymptotic density.
\end{proof}

Combining Proposition \ref{prop:slln} and Lemma \ref{lem:Chyptimes}, we see that $m_W$-a.e.\ $x\in W$ satisfies \ref{eqn:eh3} (and has $T_x W \subset K^u(x)$), so
the only remaining problem is to estimate the angle between the stable and unstable cones.

\subsection{Angle between stable and unstable cones}

Because $g$ is a diffeomorphism and the angle between $K^s$ and $K^u$ is uniformly positive on $U\setminus Z$, we see that for every $\bar\theta>0$ there is $T$ such that if $x\in W\in \tilde \AAA$ and $\tau(x)\leq T$, then $\theta(g^kx)\geq \bar\theta$ for all $0\leq k\leq \tau(x)$.

%$t \in\NN$ there is $\hat\theta(t)>0$ such that $\theta(g^kx)\geq \hat\theta(\tau(x))$ for all $x\in Z^i$ and $0\leq k\leq \tau(x)$.  In particular, for every $T\in\NN$ there is $\bar\theta>0$ such that $\theta(g^kx)\leq \bar\theta$ for all $x\in Z^i$ with $\tau(x)\geq T$, and all $0\leq k\leq \tau(x)$.

Fix $W\in \tilde \AAA$ and let $t_n(x)$ be as in the previous section.  We conclude from the above observations that in order to bound the density of the set $\{k\in\NN \mid \theta(g^k(x)) < \bar\theta \}$, it suffices to get an upper bound for
\begin{equation}\label{eqn:tjtj}
\frac{\sum_{j=1}^m t_j(x) \one_{[t_j(x)\geq T]}}{\sum_{j=1}^m t_j(x)}
\leq \frac 1m \sum_{j=1}^m t_j(x) \one_{[t_j(x)\geq T]}.
\end{equation}
To see this, we fix $\bar\theta>0$, let $T=T(\bar\theta)$ be as above, and write $T_n(x) = \sum_{i=1}^n t_i(x)$.  For a fixed $x\in W$, let $m$ be the largest number such that $T_m(x) \leq n$, and put $d_j=1$ if $t_j(x)\geq T$, and $d_j=0$ otherwise. Then
\[
\#\{1\leq k\leq n \mid \theta(g^kx) < \bar\theta \}
\leq \bigg(\sum_{j=1}^m t_j(x) d_j(x) \bigg)
+ (n-T_m(x)) d_{m+1}(x).
\]
If $d_{m+1}=0$, we have
\[
\frac 1n \#\{1\leq k\leq n \mid \theta(g^kx) < \bar\theta \}
\leq \frac 1{T_m(x)} \sum_{j=1}^m t_j d_j,
\]
and if $d_m=1$ we have the same inequality with $m$ replaced by $m+1$ on the right-hand side.  Writing 
\[
\rho(x) = \ulim_{n\to\infty} \frac 1n \#\{1\leq k\leq n \mid \theta(g^k(x)) < \bar\theta\},
\]
this implies that
\begin{equation}\label{eqn:angle-density}
%\ulim_{n\to\infty} \rho_n(x)
\rho(x) \leq
\ulim_{m\to\infty} \frac{\sum_{j=1}^m t_j(x) \one_{[t_j(x)\geq T]}}{\sum_{j=1}^m t_j(x)}.
\end{equation}
Fixing $T$, let $\tilde t_j = t_j \one_{[t_j\geq T]}$.  With $p(t)$ as in \eqref{eqn:Ptn} and \eqref{eqn:R}, we have
\[
\PP[\tilde t_j = t \mid \FFF_{j-1}] \leq \tilde p(t),
\]
where $\tilde p(t) = p(t)$ when $t\geq T$ and $\tilde p(t)=0$ otherwise.  In particular, for every $\eps>0$ there is $T$ such that $\sum t\tilde p(t) < \eps$.  Applying Proposition \ref{prop:slln} to $\tilde t_j$ gives
\[
\ulim \frac 1m \sum_{j=1}^m \tilde t_j \leq \eps \text{ a.s.};
\]
together with \eqref{eqn:tjtj} and \eqref{eqn:angle-density} this implies that $\rho(x) \leq \eps$ for $m_W$-a.e.\ $x\in W$.
%which together with Proposition \ref{prop:slln} applied to $t_j$ yields $\rho(x) \leq \eps/R$.  
Because $\eps>0$ can be taken arbitrarily small by sending $\bar\theta\to 0$, we conclude that $m_W$-a.e.\ $x\in W$ satisfies \ref{eqn:eh2}.  The fact that $m_W$-a.e.\ $x\in W$ satisfies \ref{eqn:eh3} was proved in the previous section, and thus Theorem \ref{thm:main3} applies, establishing the existence of an SRB measure for $g$.

\section{Proof of Theorem \ref{thm:katok}}\label{sec:katok-pf}

First we outline how the conditions \ref{C1}--\ref{C3} will be verified.  The key will be to obtain estimates on the flow generated by $\XX$, in particular on solutions $x(t)$ of the flow itself and on tangent vectors $v(t)$ that evolve under the flow.  Many of these estimates mirror the ones in \cite{aK79} for the original Katok example, but we cannot directly follow the computations there, because in our setting it is possible for vectors in $K^u$ to contract (see Remark \ref{rmk:harder}).
Roughly speaking, our estimates go in three stages.
\begin{enumerate}
\item Estimate $\theta(t)$, the angle between $x(t)$ and the unstable direction, which starts near $\pi/2$ and decays towards $0$.  This is done in \eqref{eqn:tantheta}.
\item Estimate $\rho(t)$, the angle between $v(t)$ and the unstable direction, which we want to keep small to obtain good expansion estimates.  This is done in Lemma \ref{lem:tanrho2}, and a key result is Corollary \ref{cor:tanbdd}.
\item Given two trajectories starting on the same $W_j$, estimate the difference between the corresponding tangent vectors as $t$ varies.  This is done using Lemma \ref{lem:eta'} and \eqref{eqn:vLvX}.
\end{enumerate}
The conditions \ref{C1}--\ref{C3} will be verified as follows.  For \ref{C1}, we use Lemma \ref{lem:tanrho2} that shows decay of $\rho(t)$ for large values of $t$.  For \ref{C3}, we use \eqref{eqn:intexpansion2} and Corollary \ref{cor:tanbdd}, the key being that $\int_{t_1}^{t_2} \tan\rho(t)\,dt$ is uniformly bounded.

Verification of \ref{C2} is the most involved, and is carried out in \S\S\ref{sec:expansion}--\ref{sec:bounded-distortion}.  Given $W\in \tilde\AAA$ that is about to enter $Z$, the natural partition to use is $W = \bigsqcup_j W_j$ where $W_j$ is the set of points in $W$ for which $j$ is the first return time to $U\setminus Z$ under the action of $g$.  Indeed, there is $\eps>0$ such that if we partition $f^{-1}(Z)\setminus Z$ into level sets of the return time, then any curve in $\AAA$ that crosses one of these level sets completely is mapped by $G$ to a curve of length approximately $\eps$.  The issue is that there may be some level sets that $W$ does not cross completely, so that $G(W_j)$ is too short.  In this case we join $W_j$ with the neighboring $W_i$ and increase $\tau$ by 1, so that the image under $\bar G$ has length between $\eps$ and $2\eps$.

The above procedure describes an admissible decomposition such that $\bar G(W_j)$ has the right length for each $j$.  Now \ref{C2}\ref{integrable} will come from the expansion estimate \eqref{eqn:escaping} since we can take $p(t) = t^{-(1+\frac 1\alpha)}$ (up to a constant).  The rest of \ref{C2}\ref{largereturns} will come from \eqref{eqn:Deltavt}, which gives bounds on the H\"older curvature of $\bar G(W_j)$.  Finally, \ref{C2}\ref{distortion} will come from \eqref{eqn:bdd-distortion}.

\subsection{Trajectories near the fixed point}

All of the computations and estimates described in Section \ref{sec:katok} will actually be carried out on $Y := B(0,r_0) \subset Z$, where $\XX$ has the specific form $\|x\|^\alpha A x$.  As $g$ is uniformly hyperbolic on $Z\setminus Y$ and trajectories spend a uniformly bounded time in this region, the conditions continue to hold when we consider passages through the larger region $Z$.

Let $\ph_t$ be the flow on $Z$ generated by the vector field $\XX$.  When we consider a single trajectory of the flow, we will generally write $x(t) = \ph_t(x)$ to keep the notation more compact.

We start by verifying Condition \ref{C1}.  We first study how a trajectory moves through $Z$; in particular, how the relative distance from $x$ to the stable and unstable manifolds of the fixed point varies.

More precisely, let $Y = B(0,r_0)$ so that $\XX(x) = \|x\|^\alpha Ax$ on $Y$.  Let $x\colon [0,T]\to Y$ be a trajectory of the flow determined by $\XX$, and write $x =  x_u + x_s$ for $x_u \in E^u\times\{0\}$ and $x_s\in \{0\}\times E^s$.  (Recall that $E^u = \RR$ and $E^s = \RR^{d-1}$.)  Let $\theta(t)$ be the positive angle between $x$ and $E^u\times\{0\}$, so that $\tan\theta(t) = \|x_s(t)\|/\|x_u(t)\|$.  Let $T_0$ be the time at which $x$ leaves $Y$, so $\|x(T_0)\| = r_0$ and $x(t)\in Y$ for all $t\in [0,T_0]$.  Let $\lambda=\beta+\gamma$.

\begin{lemma}\label{lem:tantheta}
On $t\in [0,T_0]$, we have $(\tan\theta)' = -\lambda\|x\|^\alpha\tan\theta$.
\end{lemma}
\begin{proof}
Observe that $\dot x_u = \|x\|^\alpha \gamma x_u$ and $\dot x_s = -\|x\|^\alpha \beta x_s$, whence
\[
\|x_u\|' = \left\langle \frac{x_u}{\|x_u\|}, \dot x_u \right\rangle = \|x\|^\alpha \gamma \|x_u\|
= \|x\|^{1+\alpha} \gamma \cos \theta,
\]
and similarly $\|x_s\|' = -\|x\|^{1+\alpha}\beta \sin\theta$.  Thus
\begin{align*}
(\tan\theta)' &= \left(\frac{\|x_s\|}{\|x_u\|}\right)' = \frac{\|x_u\| \|x_s\|' - \|x_s\| \|x_u\|'}{\|x_u\|^2} \\
&= \frac{(\|x\|\cos\theta)(-\|x\|^{1+\alpha}\beta\sin\theta) - (\|x\|\sin\theta)(\|x\|^{1+\alpha}\gamma\cos\theta)}{\|x\|^2\cos^2\theta} \\
&= -(\beta+\gamma) \|x\|^\alpha\tan\theta.\qedhere
\end{align*}
\end{proof}

Defining a quantity $J(t_1,t_2)$ by
\begin{equation}\label{eqn:Jt1t2}
J(t_1, t_2) := \lambda \int_{t_1}^{t_2} \|x(\tau)\|^\alpha\,d\tau,
\end{equation}
it follows from Lemma~\ref{lem:tantheta} that
\begin{equation}\label{eqn:tantheta}
\tan\theta(t_2) = e^{-J(t_1,t_2)} \tan\theta(t_1).
\end{equation}

Given $s\in (0,\infty)$ and a trajectory $x\colon [0,T_0] \to Y$ that leaves $Y$ at time $T_0$, let $T_s$ be the time at which $\tan \theta = s$.  By Lemma \ref{lem:tantheta}, $\tan\theta$ is strictly decreasing, so $T_s$ is well-defined as long as 
\begin{equation}\label{eqn:s-ok}
\tan\theta(0) > s > \tan\theta(T_0).
\end{equation}

The following lemma controls how $\|x\|$ changes.

\begin{lemma}\label{lem:normx}
The norm of $x$ varies according to the ODE
\begin{equation}\label{eqn:xalph-ode}
\frac d{dt} \|x\|^{-\alpha} = \alpha(\beta \sin^2\theta - \gamma\cos^2\theta)
= \alpha \xi(\tan\theta),
\end{equation}
where $\xi(s) = \frac{\beta s^2 - \gamma}{s^2 + 1}$.
In particular, for every piece of trajectory $x\colon [0,T] \to Y$ (that need not enter/exit $Y$ at its endpoints), we have
\begin{align}
\label{eqn:xalph-gamma}
-\gamma\alpha &\leq 
\frac d{dt} \|x\|^{-\alpha} \leq \beta\alpha, \\
%&\geq -\gamma \alpha, \\
\label{eqn:xalph2}
\|x(t)\|^{-\alpha} &\leq \|x(T)\|^{-\alpha} + \gamma\alpha(T-t), \\
\label{eqn:xalph2b}
\|x(t)\|^{-\alpha} &\leq \|x(0)\|^{-\alpha} + \beta\alpha t.
\end{align}
Moreover, if $x\colon [0,T_0] \to Y$ is a trajectory that enters $Y$ at time 0 and leaves it at time $T_0$, then we have
%Finally, writing  for all $t \leq T_s$ (meaning that $\tan\theta \geq s$ on $[0,t]$) we have
\begin{alignat}{3}
\label{eqn:xalph3}
\|x(t)\|^{-\alpha} &\geq r_0^{-\alpha} + \alpha \xi (s) t \qquad
&&\text{ for all } s\in [0,T_s], \\
\label{eqn:xalph4}
\|x(t)\|^{-\alpha} &\geq r_0^{-\alpha} - \alpha \xi (s) (T_0 - t)
\qquad
&&\text{ for all }s\in [T_s,T_0].
\end{alignat}
\end{lemma}
\begin{proof}
For \eqref{eqn:xalph-ode}, we observe that
\[
(\|x\|^{-\alpha})' = -\alpha\|x\|^{-\alpha-1} \left\langle \frac{x}{\|x\|},\dot x\right\rangle
= -\alpha\|x\|^{-\alpha-2} \langle x, \|x\|^\alpha Ax\rangle
= -\alpha \langle \hat x, A\hat x\rangle,
\]
where we write $\hat x = x/\|x\|$.  Then the first equality in \eqref{eqn:xalph-ode} follows from the form of $A$, and the second from the identity
\[
\beta \sin^2\theta - \gamma\cos^2\theta = \frac{\beta \tan^2\theta - \gamma}{\tan^2\theta + 1} = \xi(\tan\theta).
\]
Now \eqref{eqn:xalph-gamma} follows directly from \eqref{eqn:xalph-ode}, and in turn implies \eqref{eqn:xalph2}--\eqref{eqn:xalph2b}.  For \eqref{eqn:xalph3}--\eqref{eqn:xalph4}, we first observe that for $t\in [0,T_s]$ we have $\tan\theta \geq 2$, and so $\xi(\tan\theta) \geq \xi(s)$, giving
\[
\frac d{dt} \|x\|^{-\alpha} \geq \alpha \xi(s),
\]
which proves \eqref{eqn:xalph3}.  A similar bound on $[T_s,T_0]$ gives \eqref{eqn:xalph4}.
\end{proof}

The following is our main technical result on estimates for $\|x\|$.

\begin{proposition}\label{prop:chevron}
For every $0<\kappa<1$, there are $\beta',\gamma'>0$ and $0<\kappa < 1$ such that for every trajectory that enters $Y$ at time $0$ and leaves it at time $T_0$, we have
\begin{alignat*}{3}
r_0^{-\alpha} + \alpha\beta't &\leq \|x(t)\|^{-\alpha} \leq r_0^{-\alpha} + \alpha\beta t
\quad&&\text{for all } t\in [0,\kappa T_0], \\
r_0^{-\alpha} + \alpha\gamma'(T_0-t) &\leq \|x(t)\|^{-\alpha} \leq r_0^{-\alpha} + \alpha\gamma (T_0-t)
\quad&&\text{for all } t\in [\kappa T_0,T_0].
\end{alignat*}
\end{proposition}
\begin{proof}
The upper bounds are immediate from Lemma \ref{lem:normx}.  For the lower bounds, we observe that by Lemma \ref{lem:tantheta} and \eqref{eqn:xalph-ode}, $\|x\|^{-\alpha}$ has negative second derivative with respect to  $t$; since the desired lower bounds are linear and hold at $0,T_0$, it suffices to prove them at $t=\kappa T_0$.

Given $0<s_1 < \sqrt{\gamma/\beta} < s_2$, let $\xi_j = \xi(s_j)$ so that $\xi_1 < 0 < \xi_2$, and given a trajectory as in the hypothesis, let $T_j$ be such that $\tan\theta(T_j)=s_j$.  From Lemma \ref{lem:tantheta} we have $\frac d{dt} (\log \tan\theta) = -(\beta+\gamma) \|x\|^\alpha$, and from \eqref{eqn:xalph2}--\eqref{eqn:xalph2b} we have $\|x\|^{-\alpha} \leq r_0^{-\alpha} + \alpha(\beta+\gamma)T_0$, so
\[
\frac d{dt} (\log\tan\theta) \leq \frac{-(\beta+\gamma)}{r_0^{-\alpha} + \alpha(\beta+\gamma)T_0}.
\]
We conclude that
\[
T_1 - T_2 \leq \log\big(\frac {s_2}{s_1}\big) \bigg(\frac{r_0^{-\alpha}}{\beta+\gamma} + \alpha T_0\bigg)
\]
for every such trajectory.  In particular, given $\kappa \in (0,1)$, we can choose $\eps>0$ such that $(\kappa - \eps, \kappa + \eps) \subset (0,1)$ and $\kappa-\eps > \kappa/2$; then we  choose $s_1,s_2$ as above close enough to $\sqrt{\gamma/\beta}$ that for every trajectory as in the hypothesis of the proposition, we have
\begin{equation}\label{eqn:T1T2}
T_1 - T_2 \leq \eps T_0.
\end{equation}
Now we consider the following three cases: either $T_2 \geq \kappa T_0$, or $T_1 \leq \kappa T_0$, or $T_2 < \kappa T_0 < T_1$.  

In the first case, we have $\kappa T_0 \in [0, T_{s_2}]$, and \eqref{eqn:xalph3} gives
\begin{equation}\label{eqn:one}
\|x(\kappa T_0)\|^{-\alpha} \geq r_0^{-\alpha} + \alpha \xi_2 \kappa T_0
= r_0^{-\alpha} + \alpha \xi_2 \frac{\kappa}{1-\kappa}(T_0 - \kappa T_0).
\end{equation}
In the second case, we have $\kappa T_0 \in [T_{s_1},0]$, and \eqref{eqn:xalph4} gives
\begin{equation}\label{eqn:two}
\|x(\kappa T_0)\|^{-\alpha} \geq r_0^{-\alpha} - \alpha \xi_1(T_0 - \kappa T_0)
= r_0^{-\alpha} -\alpha \xi_1\frac {1-\kappa}{\kappa} \kappa T_0.
\end{equation}
In the third case, 
 \eqref{eqn:T1T2} yields $T_2 \geq (\kappa - \eps)T_0$, and
we have $(\kappa-\eps)T_0 \in [0,T_{s_2}]$.  Thus \eqref{eqn:xalph3} gives
\[
\|x((\kappa-\eps) T_0)\|^{-\alpha} \geq r_0^{-\alpha} + \alpha\xi_2(\kappa-\eps)T_0.
\]
We can write $\kappa T_0$ as a convex combination of $(\kappa - \eps)T_0$ and $T_0$ by
\[
\kappa T_0 = \frac{1-\kappa}{1-\kappa + \eps}(\kappa-\eps)T_0 + \frac{\eps}{1-\kappa + \eps}T_0.
\]
Let $\xi_2' = \frac 12 \frac{1-\kappa}{1-\kappa+\eps} \xi_2$; 
then by concavity of $t\mapsto \|x(t)\|^{-\alpha}$ together with the fact that $\|x(T_0)\| = r_0$ and $\kappa - \eps \geq \frac 12 \kappa$, we get
\begin{equation}\label{eqn:three}
\|x(\kappa T_0)\|^{-\alpha} \geq r_0^{-\alpha} + \alpha \xi_2' \kappa T_0
= r_0^{-\alpha} + \alpha \xi_2' \frac{\kappa}{1-\kappa}(T_0 - \kappa T_0).
\end{equation}
Combining \eqref{eqn:one}--\eqref{eqn:three} and writing $\beta' = \min(\xi_2, |\xi_1| \frac{1-\kappa}{\kappa}, \xi_2')$, we conclude that $\|x(t)\|^{-\alpha} \geq r_0^{-\alpha} + \alpha \beta' t$ for $t = \kappa T_0$, and hence by concavity it holds for all $t\in [0,\kappa T_0]$.  The lower bound for $t\in [\kappa T_0, T_0]$ follows similarly by taking $\gamma' = \min(\xi_2 \frac{\kappa}{1-\kappa}, |\xi_1|, \xi_2' \frac{\kappa}{1-\kappa})$.
\end{proof}

We need a few more bounds regarding $T_s$.

\begin{lemma}\label{lem:int-tan-theta}
For every $s\in (0,\infty)$ and every trajectory satisfying \eqref{eqn:s-ok}, we have
\begin{equation}\label{eqn:int-tan-theta}
\int_{T_s}^{T_0} \|x\|^\alpha \tan\theta\,dt \leq \frac s{\gamma + \beta}.
\end{equation}
\end{lemma}
\begin{proof}
For \eqref{eqn:int-tan-theta} we observe that \eqref{eqn:Jt1t2} gives
\[
\frac d{dt} (e^{-J(T_s,t)}) = -\lambda \|x\|^\alpha e^{-J(T_s,t)},
\]
and so by \eqref{eqn:tantheta} we have
\[
\int_{T_s}^{T_0} \|x\|^\alpha \tan\theta\,dt =  \int_{T_s}^{T_0} \|x\|^\alpha s e^{-J(T_s,t)}\,dt 
= -\frac{s}{\lambda} \left[ e^{-J(T_s,t)} \right]_{T_s}^{T_0},
\]
which proves \eqref{eqn:int-tan-theta}.  
\end{proof}

Now we obtain an upper bound on $T_s$ when $s$ is large, and a lower bound on $T_s$ when $s$ is small.  Here `large' and `small' are related to the sign of $\xi(s)$, which is positive when $s^2>\gamma/\beta$ and negative when $s^2 < \gamma/\beta$.

\begin{lemma}\label{lem:bound-Ts}
Given $s\in (0,\infty)$ such that $s^2 > \gamma/\beta$, consider the quantity
\begin{equation}\label{eqn:chi-s}
\chi=\chi(s) = \frac \gamma{\gamma+\beta}\left(1 + \frac 1{s^2}\right) < 1;
\end{equation}
if the trajectory $x$ enters $Y$ at time $0$ and leaves it at time $T_0$, with $\tan\theta(T_0) < s < \tan\theta(0)$, then we have
\begin{equation}\label{eqn:Ts-small}
T_s \leq \chi T_0.
\end{equation}
If $s^2 < \gamma/\beta$, then consider
\begin{equation}\label{eqn:chi'}
\chi' = \chi'(s) = \frac{\gamma - \beta s^2}{\gamma + \beta} > 0;
\end{equation}
under the same assumptions on $x$, we have
\begin{equation}\label{eqn:Ts-large}
T_s \geq \chi' T_0.
\end{equation}
\end{lemma}
\begin{proof}
To get \eqref{eqn:Ts-small} we combine \eqref{eqn:xalph2} and \eqref{eqn:xalph3} from Lemma \ref{lem:normx}, yielding
\[
r_0^{-\alpha} + \xi \alpha T_s \leq r_0^{-\alpha} + \gamma\alpha(T_0 - T_s),
\]
where we write $\xi = \xi(s)$ and use the fact that $\|x(T_0)\| = r_0$.  This gives
\begin{equation}\label{eqn:TsT0}
(\xi + \gamma) T_s \leq \gamma T_0,
\end{equation}
and so observing that
\[
\xi + \gamma = \frac{\beta s^2 - \gamma}{s^2 + 1} + \gamma = (\beta+\gamma) \frac{s^2}{s^2 + 1},
\]
we get
\[
T_s \leq \frac \gamma{\xi + \gamma} T_0 \leq \frac{\gamma}{\beta + \gamma} \frac{s^2+1}{s^2} T_0,
\]
which proves \eqref{eqn:Ts-small}.  The claim regarding $\chi'$ is proved analogously: combining the estimates from \eqref{eqn:xalph2b} and \eqref{eqn:xalph4} gives
\[
-\xi(T_0 - T_s) \leq \beta T_s,
\]
and so
\[
T_s \geq \frac{-\xi}{\beta -\xi} T_0
= \bigg(\frac{\gamma - \beta s^2}{s^2+1} \slash \Big(\beta + \frac{\gamma - \beta s^2}{s^2+1} \Big)\bigg) T_0
= \frac{\gamma - \beta s^2}{\beta + \gamma} T_0.\qedhere
\]
%which completes the proof of Lemma \ref{lem:bound-Ts}.
\end{proof}

\subsection{Tangent vectors near the fixed point}

Now let $v(t)$ be a family of tangent vectors that is invariant under the flow -- that is, $D\ph_\tau(x(t))(v(t)) = v(t+\tau)$.   The following standard result governs how $v$ evolves in time; a proof (using different notation) is given in \cite[\S 15.2]{HS74}.

\begin{proposition}\label{prop:generalflows}
Let $\ph_t$ be the flow for a vector field $\XX$ on $\RR^n$.  Let $x(t)$ be a solution to $\dot x = \XX(x)$, and let $v(t)\subset T_{x(t)}M$ be a $D\ph$-invariant family of tangent vectors.  Then
\begin{equation}\label{eqn:vdot}
\dot v(t) = (\LLL_{v(t)} \XX)(x(t)),
\end{equation}
where $\LLL_v$ is the Lie derivative in the direction of $v$ -- that is, $\LLL_v \XX(x) = (D\XX(x))v$.
\end{proposition}

For the vector field $\XX(x) = \|x\|^\alpha Ax$, we have
\begin{equation}\label{eqn:LvX}
\begin{aligned}
\LLL_v \XX(x) &= \langle v, \grad \|x\|^\alpha \rangle Ax + \|x\|^\alpha Av \\
&= \alpha \|x\|^{\alpha-2} \langle v,x \rangle Ax + \|x\|^\alpha Av \\
&= \|x\|^\alpha \big( \alpha \langle v,\hat x\rangle A\hat x + Av \big),
\end{aligned}
\end{equation}
where $\hat x = x/\|x\|$.
Write $v=v_u + v_s$ and let $\rho(t)$ be the positive angle between $v(t)$ and $E^u$, so that $\tan\rho(t) = \|v_s(t)\|/\|v_u(t)\|$. 

\begin{lemma}\label{lem:tanrho}
For $t\in[0,T_0]$, if $\tan\rho\leq 1$, then
\begin{equation}\label{eqn:tanrho'}
(\tan\rho)' \leq -\lambda\|x\|^\alpha\tan\rho + \alpha\lambda\|x\|^\alpha \frac{\tan\theta}{1+\tan^2\theta}.
\end{equation}
\end{lemma}
\begin{proof}
It suffices to consider the case when $\|v\|=1$.  Observe that
\begin{align*}
\|v_u\|' &= \left\langle \frac{v_u}{\|v_u\|}, \dot v_u \right\rangle 
= \|v_u\|^{-1} \langle v_u, \alpha\|x\|^{\alpha-2}\langle v,x\rangle Ax_u + \|x\|^\alpha Av_u \rangle \\
&= \gamma \|x\|^\alpha \cos \rho + \gamma\alpha\|x\|^{\alpha} \langle v,\hat x\rangle \cos\theta,
\end{align*}
since $E^u$ is one-dimensional.  Writing $\phi_s$ for the angle between $v_s$ and $x_s$, a similar computation gives
\[
\|v_s\|' = -\beta\|x\|^\alpha \sin\rho - \beta\alpha\|x\|^{\alpha} \langle v,\hat x\rangle \cos \phi_s \sin \theta.
\]
Let $q_s = \cos \phi_s\in [-1,1]$.  Now
\begin{equation}\label{eqn:tanrho'2}
\begin{aligned}
(&\tan\rho)' = \left( \frac{\|v_s\|}{\|v_u\|}\right)' = \frac 1{\cos^2\rho} ( \cos\rho \|v_s\|' - \sin\rho \|v_u\|') \\
&= -(\beta+\gamma) \|x\|^\alpha\tan\rho - \frac{\cos\theta}{\cos\rho}\alpha\|x\|^{\alpha}\langle v,\hat x\rangle
(\beta q_s\tan\theta + \gamma \tan\rho) \\
&= -\lambda\|x\|^\alpha\tan\rho - 
\frac{\alpha\|x\|^\alpha (1 + q_s\tan\rho\tan\theta)(\beta q_s\tan\theta+\gamma \tan\rho)}
{1+\tan^2\theta}.
\end{aligned}
\end{equation}
The numerator in the final term is equal to
\[
\alpha\|x\|^\alpha \left(\tan \rho( \gamma + q_s^2 \beta\tan^2\theta) + q_s(\gamma\tan^2\rho+\beta)\tan\theta\right),
\]
and since $\tan\theta>0$, $\tan\rho\in[0,1]$, and $q_s \geq -1$, this is bounded below by $-\alpha\lambda\|x\|^\alpha\tan\theta$.  The result follows.
\end{proof}

Lemma \ref{lem:tanrho} establishes the existence of an invariant cone family $K^u$ by observing that since $r/(1+r^2)\leq 1/2$ for all $r\in \RR$, the cone defined by $\tan\rho\leq \alpha/2$ is invariant.
This proves the half of \ref{C1} involving $K^u(x)$.  For the condition on $K^s(x)$ in \ref{C1}, it suffices to observe that \eqref{eqn:tanrho'2} can be converted into an analogous expression for $(\cot\rho)'$, and then similar estimates to those above give
\begin{equation}\label{eqn:cotrho'}
(\cot\rho)' \geq \lambda\|x\|^\alpha \cot\rho - \alpha\lambda \|x\|^\alpha \frac{\tan\theta}{1+\tan^2\theta}
\end{equation}
whenever $\cot\rho \leq 1$, so that in particular the cone defined by $\cot\rho \leq \alpha/2$ is backwards-invariant.

%  Applying the same argument with time reversed establishes $\ref{C1}$.   
In what follows, we will carry out all of our estimates for $v\in K^u(x)$, so that in particular 
%In fact, although we will occasionally use the crude estimate
\begin{equation}\label{eqn:crude-rho}
\tan\rho \leq \frac \alpha 2 < \frac 12.
\end{equation}
The following estimates all have analogues for $v\in K^s(x)$, and we will mention these when necessary.

To establish the expansion properties in \ref{C2} and \ref{C3}, we will need more careful estimates of $\tan\rho$ than \eqref{eqn:crude-rho} gives.    As in the previous section, let $T_1\in [0,T_0]$ be such that $\tan\theta(x(T_1))=1$; our estimate will show that $\rho(t)$ is smaller the further away $t$ is from $0$ and $T_1$.

\begin{lemma}\label{lem:tanrho2}
With $x(t)$ and $v(t)$ as above, we have the following bounds:
\[
\tan\rho(t) \leq \begin{cases}
(\tan\rho_0) e^{-J(0,t)} + \frac \alpha2 e^{-J(t,T_1)} & t\in [0,T_1], \\
(\tan\rho(T_1) + \alpha J(T_1,t)) e^{-J(T_1,t)} & t\in [T_1,T_0].
\end{cases}
\]
\end{lemma}
\begin{proof}
Observe that if only the first half of the right hand side of~\eqref{eqn:tanrho'} was present, then $e^{J(t_0,t)}\tan\rho$ would be non-increasing for any $t_0$.  Thus to estimate $\tan\rho$, we differentiate this quantity and see how much it may increase.

Since $\tan\theta>0$ we have $\frac{\tan\theta}{1+\tan^2\theta} \leq \min(\frac 1{\tan\theta},\tan\theta)$ -- we will use the first bound on $[0,T_1]$ and the second on $[T_1,T_0]$.  On the first interval, we obtain
\begin{multline*}
(e^{J(0,t)} \tan\rho)' \leq \alpha\lambda e^{J(0,t)} \|x\|^\alpha (\tan\theta)^{-1} \\
= \alpha\lambda e^{J(0,T_1)} \|x\|^\alpha e^{-2J(t,T_1)} 
= \frac\alpha2 e^{J(0,T_1)} \left(e^{-2J(t,T_1)}\right)',
\end{multline*}
and hence,
\[
e^{ J(0,t)} \tan\rho(t) \leq \tan\rho_0 + \frac \alpha2 e^{J(0,T_1)} \left( e^{-2 J(t,T_1)} - e^{-2 J(0,T_1)} \right).
\]
This yields the estimate
\begin{equation}\label{eqn:tanrho1}
\begin{aligned}
\tan\rho(t) &\leq (\tan\rho_0) e^{- J(0,t)} + \frac \alpha2 e^{J(t,T_1)} \left( e^{-2 J(t,T_1)} - e^{-2 J(0,T_1)} \right) \\
&= (\tan\rho_0) e^{- J(0,t)} + \frac \alpha2 \left( e^{- J(t,T_1)} - e^{-(J(0,T_1) + J(0,t))}\right),
\end{aligned}
\end{equation}
which proves the first half of the lemma.

On $[T_1,T_0]$, we use the bound $\frac{\tan\theta}{1+\tan^2\theta} \leq \tan\theta$, and so
\[
(e^{ J(T_1,t)} \tan\rho)' \leq \alpha \lambda e^{ J(T_1,t)} \|x\|^\alpha \tan\theta 
= \alpha \lambda \|x\|^\alpha,
\]
which gives
\[
e^{ J(T_1,t)} \tan\rho(t) \leq \tan\rho(T_1) + \alpha J(T_1,t)
\]
and completes the proof of the lemma.
\end{proof}

The following consequence of Lemma \ref{lem:tanrho2} is crucial to many of our later estimates.  Here and in the remainder of the proof we use a number of constants denoted $Q_i$, which will not always be explicitly introduced.  The first appearance of such a constant should be understood to mean that there is some value of this constant, independent of the choice of trajectory or of the time $t$, for which the next statement is true.

\begin{corollary}\label{cor:tanbdd}
There is $Q_0\in \RR$ such that $\int_{t_1}^{t_2} \|x\|^\alpha \tan \rho \,dt \leq Q_0$ for every choice of $x$, $v$, and $t_1,t_2\in [0,T_0]$.
\end{corollary}
\begin{proof}
Using Lemma~\ref{lem:tanrho2}, observe that on $[0,T_1]$ we have
\begin{align*}
\lambda\|x\|^\alpha\tan\rho &\leq \lambda \tan\rho_0\|x\|^\alpha e^{-J(0,t)} + \frac\alpha2 \lambda\|x\|^\alpha e^{-J(t,T_1)} \\
&= -\tan\rho_0 (e^{-J(0,t)})' + \frac \alpha2 (e^{-J(t,T_1)})',
\end{align*}
whence
\begin{equation}\label{eqn:int0T1}
\int_0^{T_1} \|x\|^\alpha\tan\rho\,dt \leq \frac 1\lambda \left(\tan\rho_0 + \frac\alpha2\right) (1 - e^{-J(0,T_1)}).
\end{equation}
Similarly, on $[T_1,T_0]$ we see that since $\tan\rho(T_1) \leq \alpha/2$, we have
\begin{align*}
\lambda\|x\|^\alpha\tan\rho &\leq \alpha \lambda \|x\|^\alpha \left(\frac 12+J(T_1,t)\right)e^{-J(T_1,t)} \\
&= \alpha \left( -\left(\frac 32+J(T_1,t)\right)e^{-J(T_1,t)} \right)',
\end{align*}
and so
\begin{equation}\label{eqn:intT1T0}
\int_{T_1}^{T_0} \|x\|^\alpha\tan\rho\,dt \leq \frac {3\alpha}{2\lambda}.
\end{equation}
The result follows since the integrand is non-negative.
\end{proof}

\subsection{Expansion near the fixed point}\label{sec:expansion}

To estimate the expansion in the unstable cone along a trajectory, we observe that by Proposition \ref{prop:generalflows}, we have
\begin{equation}\label{eqn:intexpansion}
\log\left(\frac{\|D\ph_t(x)(v)\|}{\|v\|}\right) = \int_0^t \langle \hat v,\LLL_{\hat v} \XX \rangle \,d\tau,
\end{equation}
where $\hat v(\tau) = v(\tau)/\|v(\tau)\|$.  To make the calculations simpler we drop the hat and just assume $\|v\|=1$.  Then recalling~\eqref{eqn:LvX}, 
\begin{equation}\label{eqn:normv'}
\langle v, \LLL_v \XX \rangle
= \|x\|^\alpha \big(\alpha\langle v,\hat x\rangle \langle v,A\hat x \rangle + \langle v,Av\rangle\big).
\end{equation}
We estimate the first term by observing that
\begin{align*}
\langle v,\hat x\rangle \langle v,A\hat x\rangle &= (\cos\rho\cos\theta + q_s \sin\rho\sin\theta)
(\gamma \cos\rho\cos\theta - \beta q_s\sin\rho\sin\theta) \\
&= \gamma \cos^2\rho\cos^2\theta - \ell(x,v),
\end{align*}
where the final term is
\begin{equation}\label{eqn:elllxrho}
\ell(x,v) = q_s(\beta-\gamma)\sin\rho\sin\theta\cos\rho\cos\theta + q_s^2\beta\sin^2\rho\sin^2\theta.
\end{equation}

\begin{remark}\label{rmk:harder}
In the case $\gamma=\beta$, the above yields
\begin{align*}
\langle v,\LLL_v\XX\rangle &= \|x\|^\alpha \big(\alpha\gamma(\cos^2\rho\cos^2\theta - q_s^2\sin^2\rho\sin^2\theta) + \gamma(\cos^2\rho-\sin^2\rho)\big) \\
&\geq \gamma \|x\|^\alpha (1-(2+\alpha)\sin^2\rho) \geq 0,
\end{align*}
using the fact that $\tan\rho\leq \frac 12$ and hence $\sin^2\rho\leq \frac 15$.  Thus the unstable cone $K^u(x)$ never contains any contracting vectors.  This is the case for the original Katok map. In our case contraction may nevertheless occur when $\beta\neq \gamma$; however, we argue below that the total contraction along a trajectory making a single trip through $Y$ is uniformly bounded.
\end{remark}

We see from~\eqref{eqn:elllxrho} that there exists $Q_1>0$ such that $|\ell(x,\rho)|\leq Q_1\tan\rho$. Using this in~\eqref{eqn:normv'} together with the observation that
\[
\langle v,Av\rangle = \gamma\cos^2\rho - \beta \sin^2\rho
= \gamma - \lambda\sin^2\rho,
\]
we obtain %\foot{VC: Edited \eqref{eqn:vLvX} and \eqref{eqn:intexpansion2} so that they give two-sided bounds} 
\begin{equation}\label{eqn:vLvX}
\langle v, \LLL_v\XX\rangle \geq \|x\|^\alpha \left( \gamma(1+\alpha \cos^2\rho\cos^2\theta) - Q_2 \tan\rho\right)
%\big|\langle v,\LLL_v\XX\rangle - \gamma \|x\|^\alpha (1+\alpha \cos^2\rho\cos^2\theta) \big| \leq Q_2 \|x\|^\alpha \tan\rho
\end{equation}
for some constant $Q_2>0$.  By Corollary \ref{cor:tanbdd}, the contribution of the final term 
%right-hand side 
is uniformly bounded over all trajectories.  Together with  \eqref{eqn:intexpansion}, this shows that there is $Q_3>0$ such that
\begin{equation}\label{eqn:intexpansion2}
\log\left(\frac{\|D\ph_t(x)(v)\|}{\|v\|}\right) \geq -Q_3 + \gamma \int_0^t \|x\|^\alpha(1+\alpha\cos^2\rho\cos^2\theta) \,d\tau.
%\bigg| \log\left(\frac{\|D\ph_t(x)(v)\|}{\|v\|}\right) - \gamma \int_0^t \|x\|^\alpha(1+\alpha\cos^2\rho\cos^2\theta) \,d\tau \bigg| \leq Q_3
\end{equation}
for every $x$ and $t$ such that the trajectory of $x$ stays in $Z$ from time $0$ to time $t$.
%We note that this establishes \ref{C3} because the first half of the right-hand side of \eqref{eqn:vLvX} is non-negative.
We can use \eqref{eqn:intexpansion2} to verify \ref{C3}.  Write $\lambda^u_-(x) = \min(\lambda^u(x),0)$ and $\lambda^s_+(x) = \max(\lambda^s(x),0)$.  Then \eqref{eqn:intexpansion2} implies that
\begin{equation}\label{eqn:lu-}
\sum_{j=k}^{\tau(x)} \lambda^u_-(g^j(x)) \geq -Q_3
\end{equation}
for every $x$ and every $0\leq k\leq \tau(x)$.  Just as in \eqref{eqn:cotrho'}, an analogous bound to \eqref{eqn:intexpansion2} holds for vectors $v$ in the stable cone, and we get
\begin{equation}\label{eqn:ls+}
\sum_{j=k}^{\tau(x)} \lambda^s_+(g^j(x)) \leq Q_3.
\end{equation}
This immediately implies the second inequality in \ref{C3}, and for the first we observe that $\Delta(x) \leq \lambda^s_+(x) - \lambda^u_-(x)$, so \eqref{eqn:lu-} and \eqref{eqn:ls+} give
\[
\sum_{j=k}^{\tau(x)} (\lambda^u-\Delta)(g^j(x)) \geq \sum_{j=k}^{\tau(x)} (2\lambda^u_-(g^j(x)) - \lambda^s_+(g^j(x))) \geq -Q_3,
\]
which is enough to establish \ref{C3}.

Now we turn our attention to \ref{C2}, which once again requires us to control expansion along $W$, and by \eqref{eqn:intexpansion2} we see that it is important to control $\int \gamma \|x\|^\alpha\,d\tau$.  This can be done using Lemma \ref{lem:normx}.  Given a trajectory that escapes $Y$ at time $T_0$, we have
\begin{multline}\label{eqn:escaping-integral}
\int_{t_1}^{t_2} \gamma \|x(t)\|^\alpha\,dt
\geq \int_{t_1}^{t_2} \gamma (r_0^{-\alpha} + \gamma\alpha(T_0-t))^{-1}\,dt \\
=-\frac 1\alpha \log(r_0^{-\alpha} + \gamma\alpha(T_0-t))\Big|_{t_1}^{t_2} 
=\frac 1\alpha \log\left( \frac{1+r_0^\alpha \gamma\alpha(T_0-t_1)}{1+r_0^\alpha\gamma\alpha(T_0-t_2)}\right)
\end{multline}
In the next section we will frequently use this in the form
\begin{equation}\label{eqn:t1t2}
\int_{t_1}^{t_2} \langle v,\LLL_v\XX\rangle\,dt \geq
\int_{t_1}^{t_2} \gamma \|x\|^\alpha\,dt \geq  -Q_4 + \frac 1\alpha \log\left(\frac{T-t_1}{T-t_2}\right),
\end{equation}
where $T=T_0+1$ and $0\leq t_1<t_2\leq T-1$, and $Q_4$ is a constant independent of the trajectory.  For the time being, we establish a stronger expansion bound that can be used when $x(t_2)$ has escaped $Y$.

Let $s > \sqrt{\gamma/\beta}$, so $\chi=\chi(s)$ from \eqref{eqn:chi-s} in Lemma \ref{lem:bound-Ts} is less than $1$.
Given $x\in Y$ whose trajectory escapes $Y$ at time $T_0$, we put $t_0 = \max(0,T_s)$ and observe that by Corollary \ref{cor:tanbdd} and Lemma \ref{lem:int-tan-theta}, we have
\begin{equation}\label{eqn:Q0}
\int_{t_0}^{T_0} \|x\|^\alpha \tan\rho \,dt \leq Q_0, \qquad \int_{t_0}^{T_0} \|x\|^\alpha \tan\theta \,dt \leq \frac s\lambda.
\end{equation}
Moreover, Lemma \ref{lem:bound-Ts} gives
\begin{equation}\label{eqn:geq-chi}
T_0 - t_0 \geq (1-\chi) T_0.
\end{equation}
Now we can use \eqref{eqn:intexpansion2} to write
\begin{multline}\label{eqn:J-exp}
\log\left(\frac{\|D\ph_{T_0}(x)(v)\|}{\|v\|}\right) \geq -Q_3 + \gamma (1+ \alpha) \int_{t_0}^{T_0} \|x\|^\alpha \,d\tau \\
- \gamma \int_{t_0}^{T_0} \|x\|^\alpha (\sin^2\theta + \sin^2\rho + \sin^2\theta\sin^2\rho)\,d\tau
\end{multline}
for every $v\in K^u(x)$.  By \eqref{eqn:crude-rho} and the fact that $t\geq T_s$, the integrand on the second line is bounded above by $\|x\|^\alpha(\tan\theta + 2\tan\rho)$.  Thus by \eqref{eqn:Q0}, there is a constant $Q_5$ such that
\begin{equation}\label{eqn:escaping}
\log\left(\frac{\|D\ph_{T_0}(x)(v)\|}{\|v\|}\right)
\geq -Q_5 + \gamma(1+\alpha)\int_{t_0}^{T_0} \|x\|^\alpha \,d\tau.
\end{equation}
Using \eqref{eqn:escaping-integral} gives
\begin{equation}\label{eqn:escaping2}
\log\left(\frac{\|D\ph_{T_0}(x)(v)\|}{\|v\|}\right)
\geq -Q_5 + \left( 1 + \frac 1\alpha \right) \log(1+ r_0^\alpha \gamma\alpha (T_0 - t_0)).
\end{equation}
Together with \eqref{eqn:geq-chi} and the fact that $1-\chi>0$, this shows that $m_W(W(t))$ is bounded above by a constant multiple of $t^{-(1 + \frac 1\alpha)}$, establishing \ref{C2}\ref{integrable}.

\subsection{Bounded distortion near the fixed point}\label{sec:bounded-distortion}

To show \ref{C2}\ref{largereturns} and \ref{distortion}, we need to study two nearby trajectories on the same admissible manifold.  Fix $W\in \tilde\AAA$ such that $g(W)\cap Y\neq 0$, and let $x,y\in W_j$ for some $j$, where we fix an admissible decomposition as before.  Let $T_0 = \tau_j$ and let $T=T_0+1$.  Write $\bar x = \bar G(x)$ and $\bar y = \bar G(y)$, and note that by \eqref{eqn:escaping} we have
\begin{equation}\label{eqn:dxy}
d(x(t),y(t)) \leq Q_6 (T-t)^{-(1 + \frac 1\alpha)} d(\bar x,\bar y).
\end{equation}
%a bounded distortion estimate and an estimate on how the (H\"older) curvature of an admissible manifold varies under the flow generated by $\XX$.
%Let $x(t),y(t)$ be two trajectories of $\XX$ that have both left $Y$ by time $T_0$.  Let $\bar x = x(T_0)$ and $\bar y = y(T_0)$.  We will assume that $d(\bar x, \bar y)$ is small, and obtain our bounded distortion estimates in terms of this distance.
%Let $v(t)\in K^u$ be a $D\ph$-invariant family of tangent vectors along $x(t)$, and similarly for $w(t)$ along $y(t)$.  Recall that we write $\hat v = v/\|v\|$ and similarly for $\hat w$.
Let $v(0)$ and $w(0)$ be unit tangent vectors to $W$ at $x(0)$ and $y(0)$, respectively. Set $v(t) = D\ph_t(x(0))(v(0))$, and similarly for $w(t)$.  We will write $\Delta v = v-w$, and similarly for other quantities.  Recall that we write $\hat v = v/\|v\|$, and similarly for $w,x,y$, etc.

Let $\eta(t) = \|\Delta \hat v(t)\| = \|\hat v(t) - \hat w(t)\|$.  Most of our work in this section will be to estimate this quantity.

\begin{lemma}\label{lem:eta'}
Writing $z = \widehat{\Delta v} = \Delta v/\|\Delta v\|$, the quantity $\eta$ satisfies 
\begin{equation}\label{eqn:eta'}
\begin{aligned}
\eta' &= -(\langle v, \LLL_v \XX\rangle + \langle w,\LLL_w \XX\rangle)\eta +
\langle z, \Delta\LLL_v\XX\rangle \\
&\leq -a(t)\eta(t) + c(t),
\end{aligned}
\end{equation}
where 
\begin{equation}\label{eqn:ac}
\begin{aligned}
a(t) &= (\gamma - Q_7\tan\rho)\|x\|^\alpha, \\
c(t) &= Q_8 (T-t)^{-2}d(\bar x, \bar y).
\end{aligned}
\end{equation}
\end{lemma}
\begin{proof}
Note that
\[
\eta^2 = \langle \hat v-\hat w,\hat v-\hat w\rangle = 2(1-\langle \hat v,\hat w\rangle),
\]
so writing $\zeta = 1 - \langle\hat v, \hat w\rangle$, we have 
\begin{equation}\label{eqn:zeta'}
\eta\eta' = \zeta' = -\langle (\hat v)', \hat w\rangle - \langle\hat v,(\hat w)'\rangle.
\end{equation}
Differentiating the unit tangent vector $\hat v$ gives
\[
(\hat v)' = \left(\frac v{\|v\|}\right)'
= \frac{\|v\| v' - v \langle v',\hat v\rangle}{\|v\|^2}
= \frac{v' - \hat v\langle v', \hat v\rangle}{\|v\|}
= \LLL_{\hat v}\XX - \langle \LLL_{\hat v}\XX, \hat v\rangle \hat v,
\]
where the last equality uses Proposition \ref{prop:generalflows} together with linearity of the Lie derivative in $v$.  Thus \eqref{eqn:zeta'} yields
\[
\zeta' = -\big\langle \LLL_{\hat v}\XX - \langle \LLL_{\hat v}\XX, \hat v\rangle \hat v, \hat w\big\rangle
-\big\langle \hat v, \LLL_{\hat w}\XX - \langle \LLL_{\hat w}\XX, \hat w\rangle \hat w \big\rangle.
\]
Since the right-hand side involves no derivatives we may safely simplify the notation by considering a fixed time $t$ and assuming that $v(t),w(t)$ are normalized so that $\hat v = v$ and $\hat w=w$.    Then we have
\begin{equation}\label{eqn:zeta'2}
\begin{aligned}
\zeta' &= \langle \LLL_v\XX, v\rangle \langle v,w\rangle - \langle \LLL_v\XX, w\rangle
+ \langle \LLL_w\XX,w\rangle \langle v,w\rangle - \langle v,\LLL_w\XX\rangle \\
&= (\langle v, \LLL_v \XX\rangle + \langle w,\LLL_w \XX\rangle)(1-\zeta)
- \langle w,\LLL_v \XX\rangle - \langle v,\LLL_v \XX \rangle \\
&= -(\langle v, \LLL_v \XX\rangle + \langle w,\LLL_w \XX\rangle)\zeta +
\langle \Delta v, \Delta\LLL_v\XX\rangle.
\end{aligned}
\end{equation}
Dividing by $\eta$ yields the first half of \eqref{eqn:eta'}.  Now we observe that \eqref{eqn:LvX} gives
\begin{equation}\label{eqn:Deltav'}
\begin{aligned}
\Delta &\LLL_v\XX = \Delta\Big(\|x\|^\alpha \big(\alpha \langle v,\hat x\rangle A\hat x + Av \big)\Big) \\
&= \Delta(\|x\|^\alpha) \big(\alpha \langle v,\hat x\rangle A\hat x + Av \big) \\
&\quad+ \|y\|^\alpha \Big( \alpha \langle \Delta v, \hat x\rangle A\hat x
+ \alpha \langle w, \Delta\hat x\rangle A\hat x
+ \alpha\langle w,\hat y\rangle A(\Delta \hat x)
+ A\Delta v
\Big).
\end{aligned}
\end{equation}
We start by bounding the first and the last terms in the last line.  Let $\rho(t)$ be as in the previous section, so that $v,w$ are both within $\rho$ of $E^u$.  It follows that $\measuredangle(\Delta v,E^s) \leq \rho$.  Write $z = \widehat{\Delta v}= z_u + z_s$, where $z_{u,s}\in E^{u,s}$.  Decomposing $\hat x$ as $\hat x = x_u + x_s$ and using $A = -\beta \Id_u + \gamma \Id_s$, we have
\begin{align*}
\langle z,\hat x\rangle \langle \Delta v, A\hat x\rangle
&= \eta (z_u x_u + \langle z_s, x_s\rangle )(\gamma z_u x_u -\beta \langle z_s, x_s\rangle) \\
&\leq Q_9 \eta |z_u| \leq Q_9 \eta \tan\rho,
\end{align*}
and similarly,
\[
\langle z, A\Delta v\rangle = \eta(\gamma z_u^2 - \beta \|z_s\|^2) \leq \gamma \eta \tan \rho,
\]
so that using the first half of \eqref{eqn:eta'} together with \eqref{eqn:Deltav'} and the estimate \eqref{eqn:vLvX} on $\langle v,\LLL_v\XX\rangle$, we have
\begin{equation}\label{eqn:eta'2}
\eta' \leq a(t) \eta + b(t),
\end{equation}
where
\begin{equation}\label{eqn:bt}
\begin{aligned}
b(t) &= \Delta(\|x\|^\alpha)\big(\alpha \langle v,\hat x\rangle \langle z, A\hat x\rangle  + \langle z, Av\rangle \big) \\
&\qquad
+ \alpha \|y\|^\alpha\big( \langle w,\Delta \hat x\rangle \langle z, A\hat x\rangle + \langle w,\hat y\rangle \langle z, A(\Delta\hat x)\rangle\big).
\end{aligned}
\end{equation}

\begin{lemma}\label{lem:Deltanorms}
There are constants $Q_{10},Q_8$ such that
\begin{align*}
\|\Delta\hat x\| &\leq Q_{10}(T-t)^{-1} d(\bar x,\bar y), \\
\Delta(\|x\|^\alpha) &\leq Q_8(T-t)^{-2} d(\bar x,\bar y).
\end{align*}
\end{lemma}
\begin{proof}
Using \eqref{eqn:xalph2} gives
\begin{equation}\label{eqn:normxgeq}
\|x(t)\| \geq r_0(1+r_0^\alpha\gamma\alpha(T_0-t))^{-\frac 1\alpha},
\end{equation}
and \eqref{eqn:escaping} gives
\begin{equation}\label{eqn:Deltax}
\|\Delta x\| \leq Q_{12} \big(1+r_0^\alpha \gamma\alpha \chi (T_0-t)\big)^{-(1 + \frac 1\alpha)} d(\bar x, \bar y),
\end{equation}
Thus
\begin{align*}
\|\Delta \hat x\| &\leq \frac{\|\Delta x\|}{\|x\|}
\leq Q_{12} \big(1+r_0^\alpha \gamma\alpha \chi (T_0-t)\big)^{-(1 + \frac 1\alpha)}
\big(1+r_0^\alpha \gamma\alpha (T_0-t)\big)^{\frac 1\alpha}d(\bar x, \bar y) \\
&\leq Q_{10}\big(1+(T_0-t)\big)^{-(1 + \frac 1\alpha)}
\big(1+(T_0-t)\big)^{\frac 1\alpha}d(\bar x, \bar y),
\end{align*}
which proves the first half.  Similarly,
\begin{align*}
\Delta(\|x\|^\alpha) &\leq \alpha \|\Delta x\| \min(\|x\|,\|y\|)^{\alpha-1}
%\leq Q_8 (T-t)^{-2}d(\bar x, \bar y).
\end{align*}
establishes the second half and completes the proof of Lemma \ref{lem:Deltanorms}.
\end{proof}

We can now complete the proof of Lemma \ref{lem:eta'}.  It suffices to apply Lemma \ref{lem:Deltanorms} to \eqref{eqn:bt} and use this estimate in the first half of \eqref{eqn:eta'}, which yields the estimate in the second half of \eqref{eqn:eta'} and \eqref{eqn:ac}.
\end{proof}

To estimate the value of $\eta(t)$ using Lemma \ref{lem:eta'}, we let $I(t_1,t_2) = \int_{t_1}^{t_2} a(t)\,dt$, where $a(t)$ is as in \eqref{eqn:ac}.  Then Lemma \ref{lem:eta'} gives
\begin{align*}
(e^{I(0,t)}\eta)' &= e^{I(0,t)}  (a(t) \eta + \eta') \leq e^{I(0,t)} c(t);
\end{align*}
integrating, we obtain
\[
e^{I(0,t)}\eta(t) \leq \eta(0) + \int_0^t c(s) e^{I(0,s)}\,ds.
\]
Solving for $\eta(t)$, we have
\begin{equation}\label{eqn:etat}
\eta(t) \leq \eta(0) e^{-I(0,t)} + \int_0^t c(s) e^{-I(s,t)}\,ds.
\end{equation}
Now we use \eqref{eqn:t1t2} and Corollary \ref{cor:tanbdd} to get
\begin{align*}
\eta(t) &\leq \eta(0) e^{-Q_4} \left( \frac{T-t}{T}\right)^{\frac 1\alpha}
+ \int_0^t Q_8 (T-s)^{-2} d(\bar x, \bar y) e^{-Q_4} \left(\frac{T-t}{T-s}\right)^{\frac 1\alpha}\,ds
\end{align*}
and note that the integral is bounded above by
\[
Q_{13}(T-t)^{\frac 1\alpha} d(\bar x, \bar y) (T-s)^{-(1 + \frac 1\alpha)}\Big|_0^t
\leq
Q_{13}(T-t)^{-1} d(\bar x,\bar y).
\]
Thus we have
\begin{equation}\label{eqn:etatleq}
\eta(t) \leq Q_{13} \left( \eta(0) \left(1 - \frac tT\right)^{\frac 1\alpha}
+ (T-t)^{-1} d(\bar x,\bar y)\right).
\end{equation}
Now since $x,y\in W$ for some $W\in \AAA$, the H\"older property of $TW$ guarantees that 
\[
\eta(0) = \|\Delta v(0)\| \leq L d(x,y)^\alpha
\leq L Q_6^\alpha T^{-(\alpha + 1)} d(\bar x,\bar y)^\alpha,
\]
where the second inequality uses \eqref{eqn:dxy}.
We conclude that
\begin{equation}\label{eqn:Deltavt}
\|\Delta v(t)\| \leq Q_{14} \left( T^{-(\alpha+1)} \left(1-\frac tT\right)^{\frac 1\alpha} + (T-t)^{-1} \right) d(\bar x,\bar y)^\alpha.
\end{equation}
In particular, by choosing $r_1/r_0$ sufficiently large, we guarantee that $\bar G(W_j)$ has H\"older curvature bounded by $L$, which establishes \ref{C2}\ref{largereturns}.

It only remains to get the bounded distortion estimates.  For this we observe that
\begin{equation}\label{eqn:big-Delta}
\begin{aligned}
|\Delta\langle v,\LLL_v \XX\rangle|
&\leq |\langle \Delta v, \LLL_v \XX \rangle| + |\langle v, \Delta \LLL_v \XX \rangle| \\
&\leq Q_{15} \|\Delta v\| \|x\|^\alpha + \Delta(\|x\|^\alpha)\big(\alpha \langle v,\hat x\rangle \langle v, A\hat x\rangle  + \langle v, Av\rangle \big) \\
&\qquad
+ \alpha \|y\|^\alpha\big( \langle w,\Delta \hat x\rangle \langle v, A\hat x\rangle + \langle w,\hat y\rangle \langle v, A(\Delta\hat x)\rangle\big) \\
&\leq Q_{15} \|\Delta v\| \|x\|^\alpha + Q_{16} \big(\Delta(\|x\|^\alpha) + \|y\|^\alpha \|\Delta \hat x\| \big),
\end{aligned}
\end{equation}
where the second inequality uses \eqref{eqn:Deltav'}.  Fix $\kappa \in (0,1)$ and let $\beta',\gamma'$ be as in Proposition \ref{prop:chevron}, so that we have
\begin{equation}\label{eqn:norm-alpha}
\begin{alignedat}{3}
\|x\|^\alpha, \|y\|^\alpha
&\leq (r_0^{-\alpha} + \alpha\beta' t)^{-1} 
\quad&&\text{for all } t\in [0,\kappa T], \\
\|x\|^\alpha,\|y\|^\alpha
&\leq (r_0^{-\alpha} + \alpha\gamma'(T-t))^{-1}
\quad&&\text{for all } t\in [\kappa T,T].
\end{alignedat}
\end{equation}
Now we use \eqref{eqn:Deltavt}, \eqref{eqn:norm-alpha}, and Lemma \ref{lem:Deltanorms} to show that the total distortion $|\Delta \int_0^T \langle v,\LLL_v\mathcal{X}\rangle\,dt|$ is uniformly bounded independently of the trajectory $x$ and its length $T$.  To this end, it suffices to obtain a uniform bound for the integral $\int_1^{T-1} |\Delta \langle v,\LLL_v\mathcal{X}\rangle|\,dt$.  On $[1,T-1]$, \eqref{eqn:Deltavt} gives
\[
\|\Delta v(t)\| \leq Q_{17} (T-t)^{-1} d(\bar x,\bar y)^\alpha,
\]
and \eqref{eqn:norm-alpha} yields
\[
\|x\|^\alpha, \|y\|^\alpha \leq \begin{cases}
Q_{18} t^{-1} & t\in [1,kT], \\
Q_{18} (T-t)^{-1} & t\in [kT,T-1].\end{cases}
\]
Together with \eqref{eqn:big-Delta} and the bounds on $\Delta(\|x\|^\alpha)$ and $\|\Delta\hat{x}\|$ from Lemma \ref{lem:Deltanorms}, these show that for every $t\in [1,\kappa T]$, we have
\[
|\Delta\langle v,\LLL_v\mathcal{X}\rangle|
\leq Q_{19} (T-t)^{-1} t^{-1} d(\bar{x},\bar{y})^\alpha,
\]
while for $t\in [\kappa T,T-1]$, we have
\[
|\Delta\langle v,\LLL_v\mathcal{X}\rangle|
\leq Q_{20} (T-t)^{-2}d(\bar{x},\bar{y})^\alpha.
\]
Integrating gives the uniform upper bound
\begin{multline*}
\int_1^{T-1} \frac{|\Delta\langle v,\LLL_v\mathcal{X}\rangle|}
{d(\bar{x},\bar{y})^\alpha} \,dt
\leq \int_1^{\kappa T} Q_{19} (T-t)^{-1} t^{-1} \,dt
+ \int_{\kappa T}^{T-1} Q_{20} (T-t)^{-2}\,dt \\
\leq Q_{19}((1-\kappa)T)^{-1} \kappa T + Q_{20} \int_1^\infty s^{-2}\,ds
= Q_{19} \frac \kappa{1-\kappa} + Q_{20}.
\end{multline*}
Together with \eqref{eqn:intexpansion}, this yields 
\begin{equation}\label{eqn:bdd-distortion}
\log\bigg(\frac{\|D\ph_{T-1}(\ph_1(x))|_{T_x W}\|}
{\|D\ph_{T-1}(\ph_1(y))|_{T_y W}\|}
\bigg)
\leq \bigg(Q_{19} \frac{\kappa}{1-\kappa} + Q_{20}\bigg) d(\bar{x},\bar{y})^\alpha,
\end{equation}
which gives the desired bounded distortion estimate and completes the proof.

\bibliographystyle{amsalpha}
\bibliography{eh-srb}

\end{document}